\documentclass[12pt]{article}
\usepackage{amssymb}
\usepackage{graphicx}
\usepackage[all]{xy}
\usepackage{color}
\usepackage{latexsym}
\usepackage{amssymb,amsmath,amstext}
\usepackage{epsfig}
\usepackage{graphics}
\usepackage{subfigure}
\usepackage{euscript}
\usepackage{ifthen}
\usepackage{wrapfig}
\usepackage{amsthm}
\usepackage{multicol}
\usepackage{verbatim}
\usepackage[colorlinks=true, urlcolor=blue, linkcolor=blue, citecolor=blue]{hyperref}
\usepackage[margin=1in]{geometry}

\numberwithin{equation}{section}
\theoremstyle{plain}%
\newtheorem{theorem}{Theorem}
\numberwithin{theorem}{section}
\newtheorem{proposition}[theorem]{Proposition}
\newtheorem{lemma}[theorem]{Lemma}
\newtheorem{corollary}[theorem]{Corollary}

\newtheorem{remark}[theorem]{Remark}
\newtheorem{conjecture}[theorem]{Conjecture}

\theoremstyle{definition}
\newtheorem{example}[theorem]{Example}

\newcommand{\C}{\mathbb{C}}
\newcommand{\Z}{\mathbb{Z}}

\newcommand{\PP}{\mathbb{P}}

\newcommand{\R}{\mathbb{R}}

\newcommand{\sP}{\mathcal{P}}

\newcommand{\sL}{\mathcal{L}}

\newcommand{\cE}{\mathcal{E}}
\newcommand{\cR}{\mathcal{R}}

\newcommand{\cN}{\mathcal{N}}
\newcommand{\OO}{\mathrm{O}}
\newcommand{\leaveout}[1]{}

\newcommand{\barX}{\overline{X}}

\newcommand{\sO}{\mathcal{O}}

\newcommand{\ED}{\mathrm{EDdegree}}
\newcommand{\aED}{\mathrm{aEDdegree}}
\newcommand{\corrR}[1]{\mathcal{E}_#1}
\newcommand{\Xsing}{X_\mathrm{sing}}

\newcommand{\Ysing}{Y_\mathrm{sing}}
\newcommand{\dd}{\mathrm{d}}

\newcommand{\TT}{J}
\newcommand{\dto}{\to}
\newcommand{\mymarginpar}[1]{\marginpar{\tiny #1}}

\newcommand{\generic}{general}
\newcommand{\tr}{\mathrm{tr}}

\allowdisplaybreaks

\date{}
\begin{document}

\title{\bf The Euclidean Distance Degree \\ of an Algebraic Variety}

\author{Jan Draisma, Emil Horobe\c{t},
Giorgio Ottaviani, \\
Bernd Sturmfels and Rekha R. Thomas}

\maketitle

\begin{abstract} \noindent
The nearest point map of a real algebraic variety
with respect to Euclidean distance is an algebraic function.
 For instance, for varieties of low rank matrices,
the Eckart-Young Theorem states that this map
is given by the singular value decomposition.
This article develops a theory of such
nearest point maps from the perspective of
computational algebraic geometry.
The Euclidean distance degree of a variety
is the number of critical points of the squared distance
to a \generic{} point outside the variety.
Focusing on varieties seen in applications, we  present
numerous tools for exact computations.
\end{abstract}


\section{Introduction}

Many models in the sciences and engineering are expressed
as sets of real solutions to systems of polynomial equations in $n$ unknowns.
For such a  real algebraic variety $X \subset \R^n$, we consider the following problem:
given  $u \in \R^n$, compute
 $u^* \in X$ that minimizes the squared Euclidean distance
 $d_u(x) =   \sum_{i=1}^n (u_i-x_i)^2$ from
 the given point $u$. This  optimization
problem arises in a wide range of applications.
For instance, if
$u$ is a noisy sample from $X$,
where the error model is a standard Gaussian in $\R^n$,
then $u^*$ is the maximum likelihood estimate
for~$u$.

In order to find $u^*$ algebraically, we consider the set of solutions
in $\C^n$ to the equations defining $X$. In this manner, we
regard $X$ as a complex variety in $\C^n$,
and we examine all complex critical points of the squared
distance function  $d_u(x) =  \sum_{i=1}^n (u_i-x_i)^2$ on $X$.  Here we
only allow those critical points $x$ that are non-singular on $X$. The
number of such critical points is constant on a dense open subset of
data $u \in \R^n $.
That number is called the {\em Euclidean distance degree} (or ED degree)
of the variety $X$, and denoted as {\rm EDdegree}(X).

The ED degree of a variety $X$  measures the algebraic complexity of writing the optimal
solution $u^*$ of $d_u(x)$ over $X$. It is a function of the input data and an important invariant of the optimization problem. This paper describes the basic properties of the ED degree using tools from computational and classical algebraic geometry. In many situations, our techniques offer formulas for this invariant.  Our goal is to establish the foundations of ED
degree so that it can be specialized in specific instances to
solve the optimization problem.

Using Lagrange multipliers,
and the observation that  $\nabla d_u =  2(u-x) $,
our problem amounts to computing
all regular points $x \in X$ such that $u-x
= (u_1-x_1,\ldots,u_n-x_n) $ is
perpendicular to the tangent space $T_x X $ of $X$ at $x$.
Thus, we seek to solve the constraints
\begin{equation}
\label{eq:sys1}
x \in X \, , \,\,
x \not\in X_{\rm sing} \quad {\rm and} \quad
u-x \perp T_x X,
\end{equation}
where $X_{\rm sing}$ denotes the singular locus of $X$.
The ED degree of~$X$ counts the solutions $x$.

\begin{example} \label{ex:cardioid}
We illustrate our problem for
a plane curve.
Figure~\ref{fig:logo} shows the
{\em cardioid}
$$
X \,\, = \,\, \bigl\{ (x,y) \in \R^2\,:\,
(x^2+y^2+x)^2 = x^2 + y^2\bigr\} .
$$
For general data $(u,v)$ in $\R^2$,
the cardioid $X$ contains precisely three points $(x,y) $
whose tangent line is perpendicular to  $(u-x,v-y)$.
Thus ${\rm EDdegree}(X) = 3$. All three critical points
$(x,y)$ are real, provided $(u,v)$ lies outside the
{\em evolute}, which is the small inner cardioid
\begin{equation}
\label{eq:innercardioid}
 \bigl\{\, (u,v) \in \R^2\,:\,
27 u^4+54 u^2v^2+27 v^4+54 u^3+54 uv^2+36 u^2+9v^2+8u \,=\, 0 \, \bigr\} .
\end{equation}
The evolute
is called the {\em ED discriminant} in this paper.
If $(u,v)$ lies inside the evolute then two of the
critical points are complex, and the unique real solution
maximizes $d_u$.
 \hfill $\diamondsuit$
 \end{example}

\begin{figure}[h]
\begin{center}
\vskip -0.3cm
\includegraphics[scale=0.6]{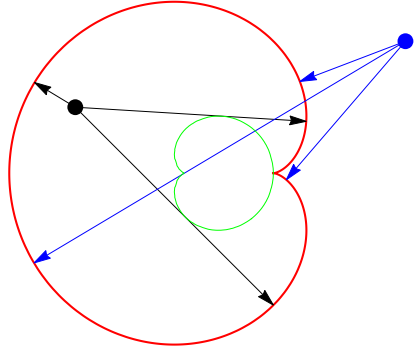}
\vskip -0.4cm
\caption{The cardioid has ED degree three. The inner cardioid is the ED discriminant.}
 \label{fig:logo}
\end{center}
\end{figure}

Readers familiar with algebraic statistics \cite{DSS}
may note that the ED degree
of a variety $X$ is  an additive analogue
of its {\em ML degree} (maximum likelihood degree).
 Indeed,
if $X$ represents a statistical model for discrete data then
maximum likelihood estimation leads to
polynomial equations which we can write
in a form that looks like (\ref{eq:sys1}),
with $u/x = (u_1/x_1,\ldots, u_n/x_n)$:
\begin{equation}
\label{eq:sys2}
x \in X \, , \,\,
x \not\in X_{\rm sing} \quad {\rm and} \quad
u/x \perp T_x(X).
\end{equation}
See Example~\ref{ex:twobytwo} and \cite{HKS, HS} for details.
Here, the optimal solution $\hat u$ minimizes the
Kullback-Leibler distance  from the distribution $u$
to the model $X$.
Thus, ED degree and ML degree are close cousins.

%In Example \ref{ex:twobytwo}, for real data $u$,
%both solutions $u^*$ are real.
For most varieties  $X$ and most data $u$,
%however,
the number of real critical points
is much smaller than ${\rm EDdegree}(X)$. To quantify
that difference we also study the expected number of real critical points of $d_u$ on $X$.
This number, denoted ${\rm aEDdegree}(X)$
and called the {\em average ED degree},
 depends on the underlying probability distribution on $\R^n$.
 For instance, for the cardioid $X$ in
Example \ref{ex:cardioid}, the invariant ${\rm aEDdegree}(X)$ can be
any real number between $1$ and $3$.
The specific value
depends on how we sample the data points $(u,v)$ from $\R^2$.

This paper is organized as follows. In Section~\ref{sec:Equations} we
rigorously define ED degree for affine and projective varieties,
and show how the ED degree of $X$ and all critical points
of $d_u$ can be computed in practice.
The projective case  is important because many varieties
 in applications are defined by homogenous equations.
For the most part, our exposition assumes no prerequisites beyond
undergraduate mathematics. We follow the book by Cox, Little and O'Shea \cite{CLO},
and we illustrate the main concepts with code in {\tt Macaulay2} \cite{M2}.

Section \ref{sec:FirstApp} is devoted to case studies in
control theory,  geometric modeling, computer vision,
and low rank matrix completion.
New results include formulas for the ED degree for the
Hurwitz stability problem and for the number of
critical formations on the line, as in \cite{AH}.

In Section~\ref{sec:Correspondence} we introduce the {\em ED correspondence}
$\mathcal{E}_X$, which is the variety of pairs $(x,u)$ where $x \in X$ is critical for $d_u$.
The ED correspondence is of vital importance for the computation of
average ED degrees, in that same section. We show how to derive
parametric representations of $\mathcal{E}_X$, and how these
translate into  integral representations for  ${\rm aEDdegree}(X)$.

Duality plays a key role in both algebraic geometry
and  optimization theory \cite{RS}.  Every projective variety
$X \subset \PP^n$ has a dual variety $X^* \subset \PP^n$,
whose points are the hyperplanes tangent to $X$.
In  Section~\ref{sec:Duality} we prove that
${\rm EDdegree}(X) = {\rm EDdegree}(X^*)$,
we express this number as  the sum of the classical polar classes
\cite{Holme}, and we lift the ED correspondence to the
conormal variety of $(X,X^*)$. When
$X$ is smooth and toric, we obtain a combinatorial
formula for ${\rm EDdegree}(X)$ in terms of the
volumes of faces of the corresponding polytope.

In Section~\ref{sec:geomop} we study the
 behavior of the ED degree under linear projections and under
intersections with linear subspaces. We also examine the
fact that the ED degree can go up or can go down
 when passing from an affine
variety in $\C^n$ to its projective closure in $\PP^n$.

In  Section~\ref{sec:Chern} we express ${\rm EDdegree}(X)$
in terms of Chern classes when $X$ is smooth and projective,
and we apply this to classical Segre and Veronese varieties.
We also study the {\em ED discriminant} which is the locus
of all data points $u$ where
two critical points of $d_u$ coincide.
For instance, in Example \ref{ex:cardioid}, the
ED discriminant is the inner cardioid.
Work of Catanese, Trifogli 
and others \cite{CT, JP} offers degree formulas for  ED discriminants
in various situations.

As we will see in Example~\ref{ex:eckartyoung}, the ED degree of the
variety of bounded-rank matrices can be derived from the 
{\em Eckart-Young Theorem}. The singular value decomposition
furnishes the critical points.
The case of
multidimensional
{\em tensors}, while of equally fundamental importance, is much more
involved. In Section~\ref{sec:Tensors}, following \cite{DH,Friedland,FO},
we give an account of recent results on the ordinary and average ED
degree of the variety of rank one tensors.

Even though the problem of minimizing Euclidean distance to a variety 
arises in a number of applications, there seems to be no systematic study 
of this problem in the generality that we address here. This paper aims to 
lay down the foundations for solving this problem
  using tools from algebraic geometry and 
computational algebra. In Section~\ref{sec:FirstApp} we will see many 
formulas for the ED degree in specific instances, some of which are new 
while others reprove existing results in the literature using our tools and 
uniform framework. The remaining sections explore several different aspects 
of  the ED degree and offer many possible directions 
in which the general theory can be further developed or tailored to particular 
applications. Our ED degree umbrella brings under it a variety of applications and theoretical tools, and 
draws on previous work that relates to subjects such as the  
Catanese-Trifogli formula for the ED discriminant (Section~\ref{sec:Chern}) 
and the work of Piene and Holme on duality (Section~\ref{sec:Duality}).

\section{Equations defining critical points}
\label{sec:Equations}

An algebraic variety $X$ in $\R^n$ can be described either implicitly,
by a system of polynomial equations in $n$ variables, or parametrically,
as the closure of the image of a polynomial map $\psi : \R^m \rightarrow
\R^n$.  The second representation arises frequently in applications,
but it is restricted to varieties $X$ that are unirational. The first
representation exists for any variety $X$.  In what follows we
start with the implicit representation, and we derive
the polynomial equations that characterize the critical points of
the squared distance function $d_u = \sum_{i=1}^n(x_i-u_i)^2$ on~$X$.
The function $d_u$ extends to a polynomial function on $\C^n$. So,
 if $x$ is a complex point in $X$ then $d_u(x)$ is usually
a complex number, and that number can be zero even if $x \not= u$.
The Hermitian inner product and its induced metric
on $\C^n$ will not appear in this paper.

Fix a radical  ideal $I_X =\langle f_1, \ldots,f_s \rangle \subset
\R[x_1,\ldots,x_n]$ and $X = V(I_X)$ its variety  in $\C^n$.
Since ED degree is additive over  the components of $X$, we may assume
that $X$ is irreducible and that $I_X$ is a prime ideal.
The formulation (\ref{eq:sys1}) translates into a
system of polynomial equations as follows.  We write  $J(f)$ for the
$s \times n$ Jacobian matrix, whose entry in row $i$ and column $j$
is the partial derivative $\partial f_i /\partial x_j$.  The singular
locus $X_{\rm sing}$ of $X$ is defined by
$$ I_{X_{\rm sing}} \,\,= \,\,
I_X + \bigl\langle \hbox{$c \times c$-minors of  $J(f)$} \bigr\rangle , $$
where $c$ is the codimension of $X$. The ideal $I_{\Xsing}$ can in fact
be non-radical, but that does not matter for our purposes. We now augment
the Jacobian matrix $J(f)$ with the row vector $u-x$ to get an $(s+1)
\times n$-matrix. That matrix has rank $\leq c$ on the critical points
of $d_u$ on $X$. From the subvariety of $X$ defined by these
rank constraints we must remove  contributions from the singular locus
$X_{\rm sing}$.  Thus the {\em critical ideal} for
$u \in \C^n$ is the following saturation:
\begin{equation}
\label{eq:critideal}
\biggl( I_X + \biggl\langle
\hbox{$(c+1) \times (c+1) $-minors of $ \begin{pmatrix} u - x \\ J(f) \end{pmatrix} $} \biggr\rangle
\biggr) : \bigl(I_{X_{\rm sing}} \,\bigr)^\infty .
\end{equation}
Note that if $I_X$ were not radical, then the above ideal could have an empty variety.

\begin{lemma} \label{lem:EDd}
For \generic{} $u \in \C^n$, the variety of the  critical ideal in $\C^n$ is finite.
It consists precisely of the critical points of
the squared distance function $d_u$ on the manifold $X \backslash X_{\rm sing}$.
\end{lemma}

\begin{proof}
For fixed  $x \in X \backslash X_{\rm sing}$, the
Jacobian $J(f)$ has rank $c$, so the data points $u$ where
the $(c{+}1) \times (c{+}1)$-minors of
$\begin{pmatrix} u-x \\ J(f) \end{pmatrix}$
vanish form an affine-linear subspace in $\C^n$ of dimension $c$.
Hence the variety of pairs $(x,u) \in X \times \C^n$ that are zeros of (\ref{eq:critideal})
is irreducible of dimension $n$. The fiber of its projection into
the second factor over a \generic{} point $u \in \C^n$ must hence be finite.
\end{proof}

The {\em ED degree} of $X$ is defined to be the number of critical points
in Lemma \ref{lem:EDd}. It is the same
as the {\em normal class} of $X$ defined in \cite{JPb}, where it is studied
in detail for curves and surfaces.
We start with two examples that are familiar to all
 students of applied mathematics.

\begin{example}[Linear Regression]
Every linear space $X$ has ED degree $1$. Here the critical equations
(\ref{eq:sys1}) take the form $x \in X$ and $u-x \perp X$.
These linear equations have a unique solution $u^*$.
If $u$ and $X$ are real then $u^*$ is the unique point in $X$
that is closest to $u$.
\hfill $\diamondsuit$
\end{example}

\begin{example}[The Eckart-Young Theorem]
\label{ex:eckartyoung}  \rm
Fix positive integers $r \leq s \leq t$ and set $n = st$.
Let $X_r$  be the variety of
$s \times t$-matrices of rank $\leq r$. This determinantal variety has
\begin{equation}
\label{eq:EY}  {\rm EDdegree}(X_r) \,\, = \,\, \binom{s}{r}.
 \end{equation}
To see this, we consider a \generic{} real $s\times t$-matrix
$U$ and its   {\em singular value decomposition}
\begin{equation}
\label{eq:SVD}
 U \,\,=\,\,\, T_1 \cdot {\rm diag}(\sigma_1, \sigma_2, \ldots, \sigma_s) \cdot T_2.
\end{equation}
Here $ \sigma_1 > \sigma_2 > \cdots > \sigma_s$ are the singular values of $U$,
and  $T_1$ and $T_2$ are orthogonal matrices of format
$s \times s$ and $t \times t$ respectively.
According to the Eckart-Young Theorem,
$$ U^* \,\, = \,\,\,
T_1 \cdot  {\rm diag}(\sigma_1, \ldots, \sigma_r,0, \ldots, 0) \cdot T_2 $$
is the closest rank $r$ matrix to $U$. More generally, the
critical points of $d_U$ are
\[ T_1 \cdot {\rm
diag}(0,\ldots,0,\sigma_{i_1},0,\ldots,0,\sigma_{i_r},0,\ldots,0)
\cdot T_2 \]
where $I=\{i_1<\ldots<i_r\}$ runs over all $r$-element
subsets of $\{1,\ldots,s\}$.
This yields the formula (\ref{eq:EY}).
The case $r=1,s=t=2$ was featured in Example \ref{ex:twobytwo}.
\hfill $\diamondsuit$
\end{example}

\begin{example} \label{ex:twobytwo}
To compare the ED degree with the ML degree, we consider
 the algebraic function that takes a $2 \times 2$-matrix $u$ to its
 closest rank one matrix. By Example~\ref{ex:eckartyoung} we have
${\rm EDdegree}(X) = 2$, while
 ${\rm MLdegree}(X) = 1$. To see what this means, consider the instance
 $$ u = \begin{pmatrix} 3 & 5 \\ 7 & 11 \end{pmatrix} . $$
The closest rank $1$ matrix in the
maximum likelihood sense of \cite{DSS, HS}
has rational entries:
$$
\hat u \quad = \quad
\frac{1}{3  {+} 5 {+} 7 {+} 11} \begin{pmatrix}
(3{+}5)(3{+}7) & (3{+}5)(5{+}11) \\
(7{+}11)(3{+}7) & (7{+}11)(5{+}11) \end{pmatrix};
$$
it is the unique rank-one matrix with the same row sums and
the same column sums as $u$.
By contrast, when minimizing the Euclidean distance, we must solve
a quadratic equation:
$$ u^*  =
\begin{pmatrix} v_{11}  \!& \! v_{12} \\ v_{21} \! & \! v_{22} \end{pmatrix}
 \,\hbox{where
$ \,    v_{11}^2{-}3 v_{11}{-}\frac{437}{1300} = 0,
v_{12}= \frac{62}{41} v_{11}{+}\frac{19}{82} ,\,
 v_{21} = \frac{88}{41} v_{11}{+} \frac{23}{82} ,\,
 v_{22} = \frac{141}{41} v_{11}{+}\frac{14}{41} $. }
$$
This rank $1$ matrix arises from the Singular Value Decomposition,
as seen in Example \ref{ex:eckartyoung}.~$\diamondsuit$
\end{example}

\begin{example}
\label{ex:M2fermat}
The following {\tt Macaulay2} code computes the
ED degree of a variety in $\R^3$:
\begin{verbatim}
R = QQ[x1,x2,x3]; I = ideal(x1^5+x2^5+x3^5); u = {5,7,13};
sing = I + minors(codim I,jacobian(I));
M = (matrix{apply(# gens R,i->(gens R)_i-u_i)})||(transpose(jacobian I));
J = saturate(I + minors((codim I)+1,M), sing);
dim J, degree J
\end{verbatim}
We chose a  random vector ${\tt u}$ as input for the above computation.
The output reveals that the
Fermat quintic cone $\{(x_1,x_2,x_3) \in \R^3 \,:\, x_1^5+x_2^5+x_3^5 = 0\}$
has ED degree $23$.
\hfill $\diamondsuit$
\end{example}

Here is a general  upper bound on the ED degree
in terms of the given polynomials $f_i$.

\begin{proposition} \label{prop:implgeneric}
Let $X $ be a variety of codimension $c$ in $\C^n$ that is cut out by
polynomials $f_1,f_2, \ldots,f_c, \ldots , f_s\,$ of degrees
$\,d_1\ge d_2\ge\cdots \ge d_c\ge\cdots \ge d_s$.
Then
$$
 {\rm EDdegree}(X) \,\,\,\le \,\,\,\,
d_1 d_2 \cdots d_c \cdot \!\!\!\!\!\!\!\!\!\!\!\!
\sum_{i_1+i_2+\cdots+i_c \leq n-c} \!\!\! \!\! (d_1-1)^{i_1} (d_2-1)^{i_2} \cdots (d_c-1)^{i_c}.
$$
Equality holds when $X$ is a \generic{} complete intersection of codimension $c$
(hence $c=s$).
\end{proposition}

In Section~\ref{sec:Chern}, this result will be derived from our
Chern class formula given in Theorem \ref{edformula} and from Theorem
\ref{thm:affineproj}. The latter relates the ED degree of an affine variety
and of its projective closure.  A similar bound for the ML degree appears
in \cite[Theorem 5]{HKS}.

Many varieties arising in applications are rational
and they are presented by a parametrization
$\psi : \R^m \rightarrow \R^n$ whose coordinates
$\psi_i$ are rational functions in $m$ unknowns $\, t = (t_1,\ldots,t_m)$.
Instead of first computing the ideal of $X$ by implicitization
and then following the approach above,
we can  use the parametrization directly to compute
the ED degree of~$X$.

The squared distance function in terms of the parameters equals
$$ D_u(t) \,\, = \,\,  \sum_{i=1}^n (\psi_i(t) - u_i)^2 . $$
The equations we need to solve are given by
$m$ rational functions in $m$ unknowns:
\begin{equation}
\label{eq:para}
 \frac{\partial D_u}{\partial t_1} \,=\, \cdots \,= \,
\frac{\partial D_u}{\partial t_m} \,=\, 0 .
\end{equation}
The critical  locus in $\C^m$ is the set of
all solutions to (\ref{eq:para}) at which the
 Jacobian of $\psi$ has maximal rank.
The closure of the image of this set under $\psi$
coincides with the variety of (\ref{eq:critideal}).
Hence, if the parametrization $\psi$ is generically finite-to-one of degree $k$,
then the critical locus in $\C^m$
is finite, by Lemma \ref{lem:EDd},
and its cardinality equals $\,k \cdot {\rm EDdegree}(X)$.

In analogy to Proposition \ref{prop:implgeneric},
we can ask for the ED degree when general polynomials
are used in the parametrization of $X$.
Suppose that $n-m$ of the $n$ polynomials $\psi_i(t)$ have
degree $\le d$, while the remaining $m$ polynomials are
\generic{} of degree $d$. Since $u$ is general,
\eqref{eq:para} has no solutions at infinity, and all its
solutions  have multiplicity one.
Hence B\'ezout's Theorem implies
\begin{equation}
\label{eq:implgeneric}
 {\rm EDdegree}(X) \, = \, (2d-1)^m .
 \end{equation}
For specific, rather than general, parametrizations, the right-hand side is
just an upper bound on the ED degree of $X$. As the following example shows,
the true value can be smaller for several different reasons.

%It demonstrates the effect of scaling  coordinates.

\begin{example} \label{ex:twistedcubic}
Let $ m = 2, n = 4$ and consider the map
$\psi(t_1,t_2) = (t_1^3, t_1^2 t_2, t_1 t_2^2, t_2^3)$,
which has degree $k=3$.
Its image $X \subset \C^4$ is the cone over the
twisted cubic curve.
The system (\ref{eq:para}) consists of two quintics in $t_1,t_2$,
so B\'ezout's Theorem predicts  $25 = 5 \times 5$ solutions.
The origin  is a solution of multiplicity $4$ and maps to a singular
point of $X$, hence does not contribute to the ED degree.
 The critical locus in $\C^2$ consists
of $21 = 25 - 4 $ points. We conclude that the toric surface
$X$ has ${\rm EDdegree}(X) = 21/k = 7$.

Next we change the parametrization by scaling the middle two monomials as follows:
\begin{equation}
\label{eq:twistdistance1}
 \widetilde \psi(t_1,t_2) \,\,\,= \,\,\, \bigl( \, t_1^3 \,, \,\sqrt{3} t_1^2 t_2 \,, \,\sqrt{3} t_1 t_2^2
\,,\, t_2^3 \,\bigr).
\end{equation}
We still have $k=3$. This scaling is special in that $||\widetilde
\psi(t_1,t_2)||^2=(t_1^2+t_2^2)^3$, and this causes the ED degree to
drop. The function whose critical points we are counting has the form
$$ \widetilde D(t_1,t_2) \,= \,
(t_1^3 - a)^2 + 3(t_1^2 t_2 - b)^2 +
3(t_1 t_2^2 - c)^2 + (t_2^3 - d)^2 , $$
where $a,b,c,d$ are random scalars. A computation shows
that the number of complex critical points of $\widetilde D$
equals $9$.  So, the  corresponding toric surface
$\widetilde X$ has ${\rm EDdegree}(\widetilde X) = 9/k = 3$.
This is a special case of Corollary~\ref{cor:veronese} on
Veronese varieties that are scaled such that the norm on the
ambient space is a power of the norm on the parametrizing space.
\hfill $\diamondsuit$
\end{example}

The variety $X \subset \C^n$ is an
{\em affine cone} if $x \in X$ implies $\lambda x \in X$
for all $\lambda \in \C$.
This means that $I_X$ is a homogeneous ideal in $\R[x_1,\ldots,x_n]$.
By slight abuse of notation, we identify $X$
with the projective variety  given by $I_X$ in $ \PP^{n-1}$.
The former is the affine cone over the latter.

We define the {\em ED degree of a
projective variety} in $\PP^{n-1}$ to be the
ED degree of the corresponding affine cone in $\C^n$.
For instance, in Example \ref{eq:twistdistance1}  we considered
 two twisted cubic curves $X$ and $\widetilde X$ that lie in $\PP^3$. These curves have
ED degrees $3$ and $7$ respectively.

To take advantage of the homogeneity of the generators of $I_X$,
and of the geometry of projective space $\PP^{n-1}$, we replace
 (\ref{eq:critideal}) with the following homogeneous
ideal in $\R[x_1,\ldots,x_n]$:
\begin{equation}
\label{eq:critideal2}
\biggl( I_X + \biggl\langle
\hbox{$(c+2) \times (c+2) $-minors of $ \begin{pmatrix} u \\ x \\ J(f) \end{pmatrix} $} \biggr\rangle
\biggr) : \bigl(I_{X_{\rm sing}} \cdot  \langle x_1^2 + \cdots + x_n^2 \rangle \,\bigr)^\infty .
\end{equation}
The singular locus of an affine cone is the cone over the
singular locus of the projective variety.
They are defined by the same ideal $I_{X_{\rm sing}}$.
The {\em isotropic quadric} $\,Q = \{
x \in \PP^{n-1}: x_1^2 + \cdots + x_n^2 = 0\}$
plays a special role,  seen clearly
in the proof of Lemma~\ref{lem:critcone}. In particular, the role of $Q$
exhibits that the computation of ED degree is a metric
problem. Note that $Q$ has no real points.
The {\tt Macaulay2} code in Example \ref{ex:M2fermat}
can be adapted to verify ${\rm EDdegree}(Q) = 0$.

The following lemma concerns the   transition between
affine cones and projective varieties.

\begin{lemma}\label{lem:critcone}
Fix an affine cone $X \subset \C^n$ and a data point $u \in \C^n \backslash X$.
Let $x \in X \backslash \{0\}$ be such that the corresponding point
$[x]$ in $\PP^{n-1}$ does not lie in the isotropic quadric $Q$.
Then $[x]$ lies in the projective variety of~(\ref{eq:critideal2})
if and only if some scalar multiple
$\lambda x$ of $x$ lies in the affine variety of
(\ref{eq:critideal}). In that case, the scalar $\lambda$ is unique.
\end{lemma}

\begin{proof}
Since both ideals are saturated with respect to $I_{X_{\rm sing}}$, it
suffices to prove this under the assumption that $x \in X \backslash
X_{\rm sing}$, so that the Jacobian $J(f)$ at $x$ has rank $c$. If
$u- \lambda x$ lies in the row space of $J(f)$, then the span of $u,x,$
and the rows of $J(f)$ has dimension at most $c+1$. This proves
the only-if direction. Conversely, suppose that $[x]$ lies in the
variety of (\ref{eq:critideal2}). First assume that $x$ lies
in the row span of $J(f)$. Then $x = \sum
\lambda_i \nabla f_i(x)$ for some $\lambda_i \in \C$. Now recall that if
$f$ is a homogeneous polynomial in $\R[x_1,\ldots,x_n]$ of degree $d$,
then $x \cdot \nabla f(x) = d \,f(x)$. Since $f_i(x)=0$ for all $i$, we
find that $x \cdot \nabla f_i(x)=0$ for all $i$, which implies
that $x \cdot x=0$, i.e., $[x] \in Q$. This contradicts the
hypothesis, so the matrix $\begin{pmatrix} x \\
J(f) \end{pmatrix}$ has rank $c+1$. But then $u- \lambda x$ lies in
the row span of $J(f)$ for a unique $\lambda \in \C$.
\end{proof}

The condition on $X$ in the following corollary is
fulfilled by any projective variety that contains at least one real point.

\begin{corollary} \label{cor:EDaffproj}
Let $X$ be a variety in $ \PP^{n-1}$
that is not contained in the isotropic quadric $Q$,
and let $u$ be \generic{}. Then ${\rm EDdegree}(X)$ is equal to the
number of zeros of (\ref{eq:critideal2}) in~$\PP^{n-1}$.
\end{corollary}

\begin{proof}
Since $X \not\subseteq Q$ and
$u$ is \generic{}, none of the critical points of
$d_u$ in   $X \backslash X_{\rm sing}$ will lie in $Q$.
The claim follows from Lemma \ref{lem:critcone}.
For further details see Theorems \ref{thm:EDaffine} and
\ref{thm:EDproj}.
\end{proof}

Corollary \ref{cor:EDaffproj} implies that
Proposition \ref{prop:implgeneric} holds
almost verbatim for projective varieties.

\begin{corollary} \label{prop:implgeneric2}
Let $X $ be a variety of codimension $c$ in $\PP^{n-1}$ that is cut out by homogeneous
polynomials $F_1,F_2, \ldots,F_c, \ldots , F_s$ of degrees $d_1\ge d_2\ge\cdots \ge d_c\ge\cdots \ge d_s$.
Then
\begin{equation} \label{eq:implgeneric2} {\rm EDdegree}(X) \,\,\,\le\,\,\,\,
d_1 d_2 \cdots d_c \cdot \!\!\!\!\!\!\!\!\!\!\!\!
\sum_{i_1+i_2+\cdots+i_c \leq n-c-1} \!\!\! \!\! (d_1-1)^{i_1} (d_2-1)^{i_2} \cdots (d_c-1)^{i_c}.
\end{equation}
Equality holds when $X$ is a \generic{} complete intersection of codimension $c$ in $\PP^{n-1}$.
\end{corollary}

Fixing the codimension $c$ of $X$ is essential in
Proposition \ref{prop:implgeneric} and Corollary \ref{prop:implgeneric2}.
Without this hypothesis, the bounds do not hold.
In Example~\ref{ex:caution}, we display
homogeneous polynomials $F_1,\ldots,F_c$
of degrees $d_1,\ldots,d_c$ whose
variety has ED degree  larger than (\ref{eq:implgeneric2}).

\begin{example}
\label{ex:M2ferma2t}
The following {\tt Macaulay2} code computes the
ED degree of a curve in $\PP^2$:
\begin{verbatim}
R = QQ[x1,x2,x3]; I = ideal(x1^5+x2^5+x3^5);  u = {5,7,13};
sing = minors(codim I,jacobian(I));
M = matrix {u}||matrix {gens R}||(transpose(jacobian I));
J = saturate(I+minors((codim I)+2,M), sing*ideal(x1^2+x2^2+x3^2));
dim J, degree J
\end{verbatim}
The output confirms that the
Fermat quintic curve given by $x_1^5+x_2^5+x_3^5 = 0$ has ED degree $23$.
By contrast, as seen from Corollary  \ref{prop:implgeneric2},
a general curve of degree five in $\PP^2$ has ED degree $25$. Saturating with
$I_{X_{\rm sing}}$ alone in the fourth line of the code would yield $25$.
\hfill $\diamondsuit$
\end{example}

It should be stressed  that the ideals (\ref{eq:critideal})
and (\ref{eq:critideal2}), and our two {\tt Macaulay2}
code fragments, are  blueprints for first computations.
In order to succeed with larger examples, it is essential
that these formulations be refined.
For instance, to express rank conditions on a polynomial matrix $M$,
the determinantal constraints are often too large, and it is
better to add a matrix equation of the form $\Lambda \cdot M  = 0$,
 where $\Lambda$ is a matrix filled with new unknowns.
 This leads to a system of bilinear equations, so the
 methods of Faug\`ere {\it et al}.~\cite{FSS} can be used.
  We also recommend trying
tools from numerical algebraic geometry, such as {\tt Bertini}~\cite{Bertini}.

\section{First applications}
\label{sec:FirstApp}

The problem of computing the closest point on a variety
arises in numerous applications. In this section
we discuss some concrete instances,
and we explore the ED degree in each case.

\begin{example}[Geometric modeling]
\label{ex:dokken}
Thomassen {\it et al.} \cite{TJD} study the nearest point
problem for a parametrized surface $X$ in $\R^3$.
The three coordinates of their birational map
$\psi : \R^2 \rightarrow X \subseteq \R^3$ are
polynomials in the parameters $(t_1,t_2)$
 that have degree $d_1$ in $t_1$
and degree $d_2$ in $t_2$.
The image $X = \overline{\psi(\R^2)}$
is a {\em B\'ezier surface} of bidegree $(d_1,d_2)$.
It is shown in \cite[\S 3]{TJD} that
$$ {\rm EDdegree}(X) \,\,\, =  \,\,\, 4d_1 d_2 + (2d_1-1)(2d_2-1). $$
This refinement of the B\'ezout bound in (\ref{eq:implgeneric}) is the
intersection number in $\PP^1 \times \PP^1$ of a curve of
bidegree $(2d_1-1,d_2)$ with a curve of bidegree $(d_1,2d_2-1)$.
The authors of \cite{TJD} demonstrate how to solve the critical equations
$\,\partial D_u /\partial t_1 = \partial D_u/ \partial t_2 = 0 \,$
with resultants based on moving surfaces. \hfill $\diamondsuit$
\end{example}

\begin{example}[The closest symmetric matrix]
\label{ex:closesym}
Let $X$ denote the variety of  symmetric $s \times s$-matrices
of rank $\leq r$. The nearest point problem for $X$
asks the following question: given a symmetric $s \times s$-matrix
$U = (U_{ij})$, find the symmetric rank $r$ matrix $U^*$ that is closest to $U$.
There are two natural interpretations
of this question in the Euclidean distance context.
The difference lies in which of the following two functions we are minimizing:
\begin{equation}
\label{eq:sympara}
D_U \,= \,  \sum_{i=1}^s \sum_{j=1}^s  \bigl(\,U_{ij} - \sum_{k=1}^r t_{ik} t_{kj} \,\bigr)^2
 \qquad  {\rm or} \qquad
 D_U \,\, =  \sum_{1 \leq i \leq j \leq s} \bigl( \,U_{ij} - \sum_{k=1}^r t_{ik} t_{kj} \, \bigr)^2 .
\end{equation}
These unconstrained optimization problems use the parametrization of
symmetric $s \times s$-matrices of rank $r$ that comes from
multiplying an $s \times r$ matrix $T = (t_{ij})$ with its transpose.
The two formulations are dramatically different as far as the ED degree is concerned.
On the left side, the Eckart-Young Theorem applies,
and ${\rm EDdegree}(X) = \binom{s}{r}$ as in
Example~\ref{ex:eckartyoung}. On the right side,
 ${\rm EDdegree}(X)$ is much larger than $\binom{s}{r}$.
For instance, for $s=3$ and $r= 1 \,\,{\rm or}\,\, 2$,
\begin{equation}
\label{eq:deg3_13} \qquad
 {\rm EDdegree}(X) = 3 \qquad \hbox{and} \qquad
{\rm EDdegree}(X) = 13.
\end{equation}
The two ideals that represent the constrained optimization problems
equivalent to (\ref{eq:sympara}) are
\begin{equation}
\label{eq:symimplicit} \!\!
\biggl\langle \hbox{$2 {\times} 2$-minors of }
\begin{pmatrix}
\sqrt{2}x_{11} \! & \! x_{12} & \! x_{13} \\
x_{12} &\! \! \! \!\sqrt{2}x_{22} \! & \! x_{23} \\
x_{13} & \! x_{23} &\! \!\! \sqrt{2} x_{33}
\end{pmatrix}
\biggr\rangle \,\,\hbox{and}  \,\,
\biggl\langle \hbox{$2 {\times} 2$-minors of }
\begin{pmatrix}
x_{11} & \! x_{12} & \! x_{13} \\
x_{12} & \! x_{22} & \! x_{23} \\
x_{13} & \! x_{23} & \! x_{33}
\end{pmatrix}
\biggr\rangle. \quad
\end{equation}
These equivalences can be seen via a change of variables. For example,
for the left ideal in (\ref{eq:symimplicit}),
the constrained optimization problem is to minimize $\sum_{1 \leq i \leq j \leq 3} (u_{ij}-x_{ij})^2$ subject to the nine quadratic equations $2x_{11}x_{22}=x_{12}^2, \sqrt{2}x_{11}x_{23}=x_{12}x_{13}, \ldots,2x_{22}x_{33}=x_{23}^2$. Now making the change of variables $x_{ii} = X_{ii}$ for $ i=1,2,3$ and
$x_{ij} = \sqrt{2}X_{ij}$ for $1 \leq i < j \leq 3$, and similarly,
$u_{ii} = U_{ii}$ for $i=1,2,3$ and
$u_{ij} = \sqrt{2}U_{ij}$ for $1 \leq i < j \leq 3$, we get the problem
$$\begin{matrix}
\textup{minimize} & \!\! \sum_{i=1}^{3} (U_{ii}-X_{ii})^2 \,+\, \sum_{1 \leq i < j \leq 3} 2(U_{ij}-X_{ij})^2
\medskip \\
\textup{ subject to} & \quad X_{ik}X_{jl} = X_{il}X_{jk} \, \,\,\, {\rm for} \,\,\, 1 \leq i < j \leq 3, 1 \leq k < l \leq 3.
\end{matrix}$$
This is equivalent to the left problem in (\ref{eq:sympara}) for $r=1$ via the
parametrization  $X_{ij} = t_it_j$.
The appearance of $\sqrt{2}$ in
the left matrix $M(x)$ in (\ref{eq:symimplicit}) is analogous to the appearance
of $\sqrt{3}$ in Example \ref{ex:twistedcubic}: it is the special scaling that
relates the natural squared matrix norm $\tr M(x)^T M(x)$ on the ambient space to
(two times) the squared norm $||x||^2$, and this puts the variety
defined by the $2 \times 2$-minors of $M(x)$ into special position
relative to the isotropic quadric $Q$.
In Example \ref{ex:bothmer}
we discuss a general ED degree formula for
symmetric $s \times s$-matrices of rank $\leq r$ that works
for the version on the right.
The same issue for ML degrees is
the difference between ``scaled'' and ``unscaled''
in the table at the end of \cite[\S 5]{HKS}.
\hfill $\diamondsuit$
\end{example}

\begin{example}[Computer vision]
This article got started with the following problem from \cite{AST,
HartleySturm, SSN}.  A general projective camera is a 
$3 \times 4$ real matrix of rank three, that defines a linear map from 
$\PP^3$ to $\PP^2$ sending a ``world point'' $y \in \PP^3$ to its image $Ay \in \PP^2$.
This map is well-defined everywhere except at the kernel of $A$, which is called 
the center of the camera.

The {\em multiview variety} associated to $n$ cameras $A_1,A_2,\ldots,A_n$
is the closure of the image of the map $\PP^3
\dashrightarrow (\PP^2)^n , y \mapsto (A_1 y, A_2 y, \ldots, A_n y)$. This is an irreducible 
three-fold in $(\PP^2)^n$ and its defining prime ideal $I_n$ is multi-homogeneous and lives in the
polynomial ring $\R[x_{ij}: i=1,\ldots,n, j=0,1,2]$, where  $(x_{i0}
: x_{i1}:x_{i2})$ are homogeneous coordinates of the $i$-th  plane
$\PP^2$. Explicit determinantal generators and Gr\"obner bases for $I_n$
are derived in \cite{AST}.  If we dehomogenize $I_n$ by setting $x_{i0}
= 1$ for $i=1,2,\ldots,n$, then we get a $3$-dimensional affine variety
$X_n$ in $\R^{2n} = (\R^2)^n$ that is the space of dehomogenized 
images under the $n$ cameras.  Note that $I_n$ and $X_n$ depend on the
choice of the matrices $A_1,A_2,\ldots,A_n$.  This dependence is governed
by the Hilbert scheme in \cite{AST}.

The Euclidean distance problem for $X_n$ is
known in computer vision as {\em $n$-view triangulation}. Following
  \cite{HartleySturm} and \cite{SSN}, the data
$u \in \R^{2n}$ are $n$
noisy images of a point in $\R^3$ taken by the $n$ cameras.
The maximum likelihood
solution of the recovery problem with Gaussian noise
is the configuration $u^* \in X_n$ of minimum distance to $u$.
For $n=2$, the variety $X_2$ is a hypersurface cut out by
a  bilinear polynomial
$(1,x_{11}, x_{12}) M  (1,x_{21},x_{22})^T$, where
$M$ is a $3 \times 3$-matrix of rank $2$.
Hartley and Sturm \cite{HartleySturm}
studied the critical equations and found that ${\rm EDdegree}(X_2) = 6$.
Their computations were extended by Stew{\'e}nius {\it et al.} \cite{SSN} up to $n=7$:
$$
\begin{array}{r|rrrrrrr}
n & \,\, 2 & 3 & 4 & 5 & 6 & 7 \\
\hline
 {\rm EDdegree}(X_n)\,\, & \,\, 6 & 47 & 148 & 336 & 638 & 1081
 \end{array}
$$

This table suggests the conjecture that these ED degrees grow
as a cubic polynomial:

\begin{conjecture} \label{conj:6_47}
The Euclidean distance degree of the affine multiview variety $X_n$ equals
$${\rm EDdegree}(X_n) \quad = \quad \frac{9}{2} n^3 - \frac{21}{2} n^2 + 8 n - 4. $$
\end{conjecture}

At present we do not know how to prove this.
Our first idea was to replace
the affine threefold $X_n$ by a projective variety. For instance,
consider the closure $\overline{X}_n$ of $X_n$ in $\PP^{2n}$.
Alternatively, we can regard $I_n$ as a homogeneous
ideal in the usual $\Z$-grading, thus defining a
projective variety $Y_n$ in $\PP^{3n-1}$. However,  for $n \geq 3$,
the ED degrees of  both $\overline{X}_n$ and $Y_n$
are larger than the ED degree of $X_n$.
For instance, in the case of three cameras we~have
$$ {\rm EDdegree}(X_3) = 47 \,\,\,\,  < \,\,\,\,
{\rm EDdegree}(\overline{X}_3) \,=\, 112 \,\,\,\, < \,\,\,\,
{\rm EDdegree}(Y_3) \, = \,148. $$
Can one find a natural reformulation of
 Conjecture \ref{conj:6_47} in terms of projective geometry?
\hfill $\diamondsuit$
\end{example}

Many problems in engineering lead to minimizing the distance from a given
point $u$ to an algebraic variety.  One such problem is
detecting voltage collapse and blackouts in electrical power systems
\cite[page 94]{PabloThesis}. It is typical to model a power system
as a differential equation $\dot x = f(x,\lambda)$ where $x$ is the state and
$\lambda$ is the parameter vector of load powers. As $\lambda$ varies,
the state moves around. At
critical load powers, the system can lose equilibrium
 and this results in a blackout due to voltage collapse. The
set of critical $\lambda$'s form an algebraic variety $X$ that one wants to stay away from.
This is done by calculating the closest point on $X$ to the current set of parameters
$\lambda_0$ used by the power system.
A similar, and very well-known, problem from {\em control theory}
is to ensure the {\em stability} of a univariate polynomial.

\bigskip

\begin{table}[h]
\begin{center}
\begin{tabular}{c|ccll}
$n$ & $\ED(\Gamma_n)$ & $\ED(\bar{\Gamma}_n)$ & $\aED(\Gamma_n)$ & $\aED(\bar{\Gamma}_n)$ \\
\hline
3 & 5 & 2 &
1.162...      & 2                     \\
4 & 5 & 10 &
1.883...      & 2.068...               \\
5 & 13 & 6 &
2.142...      & 3.052...              \\
6 & 9 & 18 &
2.416...      & 3.53...               \\
7 & 21 & 10 &
2.66...       & 3.742...
\end{tabular}
\caption{ED degrees and average and ED degrees of small Hurwitz
determinants.} \label{tab:Hurwitz1}
\end{center}
\end{table}

\begin{example}[Hurwitz stability]
\label{ex:hurwitz}
Consider  a univariate polynomial with real coefficients,
$$\,u(z)\,\, = \,\, u_0 z^n + u_1 z^{n-1} + u_2 z^{n-2} + \cdots
+ u_{n-1} z + u_n. $$
We say that $u(z)$ is {\em stable}  if each of its $n$ complex zeros has negative real part.
It is an important problem in control theory
to check whether a given polynomial $u(z)$ is stable,
and, if not, to find a polynomial
$x(z)$ in the closure of the stable locus
that is closest to $u(z)$.

The stability of
$x(z) = \sum_{i=0}^n x_i z^i$ is characterized by the following {\em Hurwitz test}.
The $n$th {\em Hurwitz matrix} is an
$n \times n$ matrix with $x_1,\ldots,x_n$ on the diagonal. Above the diagonal entry
$x_i$ in column $i$, we stack as much of $x_{i+1}, x_{i+2}, \ldots,x_n$ consecutively, followed by zeros if there is extra room. Similarly, below $x_i$, we stack as much of $x_{i-1}, x_{i-2}, \ldots,x_1,x_0$ consecutively, followed by zeros if there is extra room.
The Hurwitz test says that $x(z)$ is stable if and only if every leading principal minor of $H_n$ is positive.
 For instance, for $n = 5$ we have
$$H_5 \,\,=\,\, \left( \begin{array}{ccccc}  x_1 & x_3 & x_5 & 0 & 0 \\ x_0 & x_2 & x_4 & 0 & 0 \\ 0 & x_1 & x_3 & x_5 & 0 \\ 0 & x_0 & x_2 & x_4 & 0 \\ 0 & 0 & x_1 & x_3 & x_5 \end{array} \right).$$
The ratio $\bar{\Gamma}_n = {\rm det}(H_n)/x_n$, which is the $(n-1)$st leading principal minor of $H_n$,
 is a homogeneous polynomial in the variables $x_0,\ldots, x_{n-1}$ of degree $n-1$.
Let $\Gamma_n$ denote the non-homogeneous polynomial
obtained by setting $x_0 = 1$ in $\bar{\Gamma}_n$.
We refer to $\Gamma_n$ resp.~$\bar{\Gamma}_n$ as
 the non-homogeneous resp.~homogeneous
{\em Hurwitz determinant}. Table \ref{tab:Hurwitz1}
 shows the ED degrees and the average ED degrees
 of both $\Gamma_n$ and $\bar{\Gamma}_n$
 for some small values of $n$.
The average ED degree was computed with respect to the
standard multivariate Gaussian distribution
in $\R^n$ or $\R^{n+1}$ centered
at the origin. For  the formal definition of
$\aED( \,\cdot \,)$ see Section~\ref{sec:Correspondence}.
The first two columns in Table \ref{tab:Hurwitz1}
seem to be oscillating by parity. Theorem \ref{thm:Hurwitz2}
explains this. Interestingly, the oscillating behavior
does not occur for average ED degree.
\hfill $\diamondsuit$
\end{example}

\begin{theorem}
\label{thm:Hurwitz2}
The ED degrees of the Hurwitz determinants
are given by the following table:
 $$
\begin{tabular}{l|cc}
& $\ED(\Gamma_n)$ & $\ED(\bar{\Gamma}_n)$ \\
\hline
$n=2m+1$ & $8m-3$ & $4m-2 $ \\
$n=2m$   &    $4m-3$ & $8m-6 $
\end{tabular}
 $$
\end{theorem}

\begin{proof}
The hypersurface $X  = V(\bar{\Gamma}_n)$ defines the boundary of the stability region.
If a polynomial $x(z)$ lies on $X$, then it has a complex root on the imaginary axis, so
it admits a factorization
$x(z) = (c z^2+d)(b_0 z^{n-2}+\cdots+b_{n-2})$. This representation
yields a parametrization of the hypersurface $X \subset \PP^n$ with parameters
$b_0,\ldots,b_{n-2},c,d$. We can rewrite this~as
\[
x :=\begin{bmatrix} x_0 \\ x_{1} \\ \vdots \\ x_n \end{bmatrix}\,\,=\,\,
\begin{bmatrix}
c & & & &\\
0 & c & & &\\
d & 0 & c & &\\
  & \ddots & \ddots & \ddots &\\
  & & d & 0 & c \\
  & & & d & 0 \\
  & & & & d
\end{bmatrix}
\cdot \begin{bmatrix} b_0 \\ b_1 \\ \vdots \\ b_{n-2}
\end{bmatrix} \,\,=:\,\, C \cdot b.
\]
Where this parametrization is regular and $x$ is a smooth point of $X$,
the tangent space $T_x X$ is spanned by the columns of $C$
and the vectors $b',b''$ obtained by appending or prepending two zeros to
$b$, respectively. Thus, for $u \in \C^{n+1}$,
the condition $u-Cb \perp T_x X$ translates into
\[ C^T (u-Cb)=0 \,\,\, \text{ and }
\,\,\, (b')^T (u-Cb)=0 \,\,\, \text{ and } \,\,\, (b'')^T (u-Cb)=0. \]
The first equation expresses $b$ as a rational homogenous
function in $c,d$, namely, $b=b(c,d)=(C^T C)^{-1} C^T u$. By Cramer's
rule, the entries of the inverse of a matrix are homogeneous rational
functions of degree $-1$.  Hence the degree
of $b(c,d)$  equals $-2+1=-1$. Let $\gamma=\gamma(c,d)$ be the
denominator of $(C^T C)^{-1}$, i.e., the lowest-degree polynomial in $c,d$
for which $\gamma \cdot (C^T C)^{-1}$ is polynomial; and let $N$ be the
degree of $\gamma$. Then $\gamma b'$ has entries that are homogeneous
polynomials in $c,d$ of degree $N-1$. Similarly, $\gamma u-\gamma Cb$
has degree $N$. Hence
\[ p(c,d)\,:=\,(\gamma  b') \cdot (\gamma u- \gamma Cb) \quad \text{ and  }  \quad
q(c,d) \,:=\,(\gamma  b'') \cdot (\gamma u- \gamma Cb)
\]
are homogeneous polynomials of degree $2N-1$ that vanish on the
desired  points $(c:d) \in \PP^1$.
 Indeed, if $p$ and $q$ vanish
on $(c:d)$ and $\gamma(c,d)$ is non-zero, then there is a unique $b$
that makes $(b,c,d)$ critical for the data $u$. It turns out that $p$ is divisible
by $d$, that $q$ is divisible by $c$, and that $p/d=q/c$. Thus $2N-2$ is an
upper bound for $\ED(X)$.

To compute $\gamma$, note that $C^T C$ decomposes into two blocks,
corresponding to even and odd indices. When $n=2m+1$
is odd, these two blocks are identical, and $\gamma$ equals their
determinant, which is $c^{2m} + c^{2m-2}d^2 + \cdots + d^{2m}$.
Hence $N = 2m$. When
$n=2m$ is even, the two blocks are distinct, and $\gamma$ equals the
product of their determinants, which is $(c^{2m} + c^{2m-2}d^2 + \cdots
+ d^{2m})(c^{2m-2}+\cdots+d^{2m-2})$.
Hence $N = 4m-2$.
 In both cases one can check that
$p/d$ is irreducible, and this implies that $\ED(X)=2N-2$.
This establishes the stated formula for ${\rm EDdegree}(\bar \Gamma_n)$.
A similar
computation can be performed in the non-homogeneous case,
by setting $x_0 = b_0 = c = 1$, leading to the formula for
${\rm EDdegree}(\Gamma_n)$.
\end{proof}

\begin{comment}
The parametrization also gives a grip on tangency to the isotropic
quadric $Q$: a point $a=Cb$ where $Q$ and $X$ are tangent must have $T_a
X$ contained in $a^\perp$. In particular, this implies that $C^T Cb=0$
and that $(b')^T Cb=0$ and $(b'')^T Cb=0$. We may work in the affine
patch where $d=1$. Then the first condition gives that $\gamma(c,1)=0$
and that $b$ lies in the kernel of $C^T C$. The roots of
$\gamma(c,1)$ are
all non-zero, and then the condition $(b'')^T Cb=0$ follows from the
other two conditions. When $n$ is odd, the kernel of $C^T C$ is at least
two-dimensional when $\gamma(c,1)=0$, since each of its two identical
blocks then has determinant zero. So we have a two-dimensional space
worth of $b$'s to choose from, which certainly contains a $b$ for
with $(b')^T Cb=0$. Thus when $n$ is odd, $Q$ and $X$ are tangent.
When $n$ is even, computational experiments suggest that the kernel of
$C^T C$ is one-dimensional for each root of $\gamma(c,1)$, and that the
corresponding $b$ does not fulfill $(b')^T C b=0$, so that $Q$ and $X$
are not tangent.
\end{comment}

\begin{example}[Interacting agents]
This concerns a problem we learned from  work of
Anderson and Helmke \cite{AH}.
Let $X$ denote the variety in $\R^{\binom{p}{2}}$
with parametric representation
\begin{equation}
\label{eq:CMpara}
 d_{ij}\,\, = \,\,(z_i-z_j)^2 \quad \hbox{for} \quad 1 \leq i < j \leq p.
 \end{equation}
Thus, the points in $X$
record the squared distances among $p$
interacting agents with coordinates  $z_1,z_2,\ldots,z_p$ on the line $\R^1$.
Note that $X$ is the cone over a projective variety in $\PP^{{\binom{p}{2}}-1}$.
The prime ideal of $X$ is given by the $2 \times 2$-minors of the
{\em Cayley-Menger matrix}
\begin{equation}
\label{eq:CMmatrix}
\begin{bmatrix}
2 d_{1p} & d_{1p} {+} d_{2p} {-} d_{12} &  d_{1p} {+} d_{3p} {-} d_{13} &
  \cdots &  d_{1p} {+} d_{p-1,p} {-} d_{1,p-1} \\
 d_{1p} {+} d_{2p} {-} d_{12} & 2 d_{2p} &  d_{2p} {+} d_{3p} {-} d_{23} &
  \cdots &  d_{2p} {+} d_{p-1,p} {-} d_{2,p-1} \\
   d_{1p} {+} d_{3p} {-} d_{13} &  d_{2p} {+} d_{3p} {-} d_{23} & 2 d_{3p} &
   \cdots &  d_{3p} {+} d_{p-1,p} {-} d_{3,p-1} \\
   \vdots & \vdots & \vdots & \ddots & \vdots \\
   d_{1p} {+} d_{p-1,p} {-} d_{1,p-1} \! & \!
   d_{2p} {+} d_{p-1,p} {-} d_{2,p-1} \! & \!
   d_{3p} {+} d_{p-1,p} {-} d_{3,p-1} \! &  \!
     \cdots &    2 d_{p-1,p}
 \end{bmatrix}
\end{equation}
Indeed, under the parametrization (\ref{eq:CMpara}),
the $(p-1) \times (p-1)$ matrix (\ref{eq:CMmatrix}) factors as $2Z^T  Z$,
where $Z$ is the row vector
$(z_1{-}z_p,z_2{-}z_p, z_3{-}z_p,\ldots,z_{p-1}{-}z_p)$.
We can form the Cayley-Menger matrix  (\ref{eq:CMmatrix})
for any finite metric space on $p$ points.
The metric space can be embedded in a Euclidean space
if and only if (\ref{eq:CMmatrix}) is positive semidefinite \cite[(8)]{Lau}.
That Euclidean embedding is possible in dimension $r$
if and only if the rank of $(\ref{eq:CMmatrix})$ is at most~$r$.

The following theorem is inspired by \cite{AH} and
provides a refinement of results therein.
In particular, it explains the findings
in \cite[\S 4]{AH}~for~$p \leq 4$. There is an extra factor of $1/2$
because of the involution $z \mapsto -z$ on the fibers of
the map (\ref{eq:CMpara}). For instance, for $p  = 4$,
our formula gives ${\rm EDdegree}(X) = 13$ while
\cite[Theorem 13]{AH} reports $26$ non-zero critical points.
The most interesting case occurs when
$p$ is divisible by $3$, and this will be explained in the proof.
\end{example}

\begin{theorem}
The ED degree of the Cayley-Menger variety $X \subset \PP^{\binom{p}{2}-1}$ equals
\begin{equation}
\label{eq:EDCM}
 {\rm EDdegree}(X) \,\, = \,\,
 \begin{cases}
  \frac{3^{p-1} - 1}{2} & \hbox{ if $\,p \equiv 1,2 \!\mod 3$} \\
  \frac{3^{p-1} - 1}{2} - \frac{p!}{3 ((p/3)!)^3}& \hbox{ if $\,p \equiv 0 \! \mod 3$}
 \end{cases}
 \end{equation}
  \end{theorem}

\begin{proof}
After the linear change of coordinates given by $x_{ii} = 2 d_{ip} $ and
$x_{ij} = d_{ip} + d_{jp} - d_{ij}$, the Cayley-Menger variety $X$
agrees with the variety of symmetric $(p-1) \times (p-1)$-matrices of rank $1$.
This is the Veronese variety of order $d=2$. The number $(3^{p-1}-1)/2$ is a special
instance of the formula in Proposition \ref{prop:Veronese}.
To show that it is valid here, we need to prove that
$X$ intersects the isotropic quadric $Q$ transversally,
i.e., the intersection $X \cap Q$ is non-singular.
If there are isolated nodal singular points, then their
number gets subtracted.

The parametrization (\ref{eq:CMpara}) defines
the second Veronese embedding
$\PP^{p-2} \rightarrow X \subset  \PP^{\binom{p}{2}-1}$,
where $\PP^{p-2}$ is the projective space of the quotient
$\C^p/\C \cdot (1,\ldots,1)$.
So $X \cap Q$ is isomorphic to its inverse image in
$\PP^{p-2}$
under this map. That inverse image
 is the hypersurface in $\PP^{p-2}$ defined by
 the homogeneous quartic $\, f = \sum_{1 \leq i < j \leq p} (z_i - z_j)^4$.
 We need to analyze the singular locus of the hypersurface
 $V(f)$ in $\PP^{p-2}$, which is the variety defined by all
 partial derivatives of $f$. Arguing modulo $3$
 one finds that if $p$ is not divisible by $3$ then $V(f)$ is
 smooth, and then we have ${\rm EDdegree}(X) =(3^{p-1}-1)/2$.
If $p$ is divisible by $3$ then $V(f)$ is not smooth,
but $V(f)_{\mathrm sing}$ consists of isolated nodes that
form one orbit under
 permuting coordinates. One representative is the point in $\PP^{p-2}$
 represented by the vector
$$ \bigl(0,0, \ldots , 0, \,1,1, \ldots ,1, \, \xi, \xi, \ldots , \xi \bigr) \in \C^p
\quad \hbox{where} \,\,\,\,\xi^2-\xi+1 = 0 . $$
The number of singular points of the quartic hypersurface $V(f)$ is equal to
$$ \frac{p!}{3 \cdot ((p/3)!)^3}. $$
For $p>0$ this is the number of words that start with the first letter of the ternary alphabet
$\{0,1,\xi\}$ and that contain each letter exactly $p$ times; see \cite[A208881]{OEIS}.
\end{proof}

\begin{comment}
\mymarginpar{
The problem studied in \cite{AH} make much sense
 for point agents in $\R^r$. It would
be very interesting to find a formula for the
ED degree of the variety of matrices (\ref{eq:CMmatrix}) of rank $\leq r$.
 Does it agree with the
ED degree of the $r$-th secant variety of the
second Veronese variety?
}
\end{comment}

\section{ED correspondence and average ED degree}
\label{sec:Correspondence}

The ED correspondence arises when
the variety $X$ is fixed but the data point $u$ varies. After studying
this, we restrict to the real numbers, and we introduce the average ED
degree,
making precise a notion that was hinted at in Example \ref{ex:hurwitz}.
The ED correspondence yields an integral formula for  $\aED(X)$.
This integral
can sometimes be evaluated in closed form. In other cases, experiments
show that evaluating the integral numerically is more efficient than
estimating $\aED(X)$ by sampling $u$ and counting real critical points.

We start with an irreducible affine variety $X \subset \C^n$ of
codimension $c$ that is defined over $\R$,
with prime ideal $I_X = \langle f_1,\ldots,f_s \rangle$
 in $\R[x_1,\ldots,x_n]$.
The {\em ED correspondence} $\cE_X$ is the
subvariety of $\C^n \times \C^n$ defined by the ideal \eqref{eq:critideal}
in the polynomial ring $\R[x_1,\ldots,x_n,u_1,\ldots,u_n]$. Now, the
$u_i$ are unknowns that serve as coordinates on the second
factor in $\C^n \times \C^n$.  Geometrically, $\cE_X$ is the
topological closure in $\C^n \times \C^n$ of the set of pairs $(x,u)$
such that $x \in X \backslash X_{\rm sing}$ is a critical
point of $d_u$. The following theorem implies and enriches Lemma~\ref{lem:EDd}.

\begin{theorem} \label{thm:EDaffine}
The ED correspondence $\mathcal{E}_X$ of an irreducible
subvariety $X \subseteq \C^n$ of codimension $c$
is an irreducible variety of
dimension $n$ inside $\C^n \times \C^n$. The first projection $ \pi_1 :
\mathcal{E}_X \rightarrow X \subset \C^n$ is an affine vector bundle
of rank $c$ over $X \backslash X_{\rm sing}$.   Over \generic{} data points
$u \in \C^n$, the second projection $\pi_2 : \mathcal{E}_X \rightarrow
\C^n$ has finite fibers $\pi_2^{-1}(u)$ of cardinality equal to $\ED(X)$.
If, moreover, we have $T_x X \cap (T_x X)^\perp = \{0\}$
at some point $x \in X \backslash \Xsing$, then $\pi_2$ is
a dominant map and $\,\ED(X)$ is positive.
\end{theorem}

In our applications, the variety $X$ always has real points that are smooth,
i.e.~in $X\backslash X_{\rm sing}$.
If this holds, then the last condition in Theorem \ref{thm:EDaffine}
is automatically satisfied: the tangent space at such a point is real and
intersects its orthogonal complement trivially.
 But, for instance, the hypersurface $ Q = V(x_1^2 + \cdots + x_n^2)$
 does not satisfy this condition: at any point $x \in Q$
the tangent space $T_x Q=x^\perp$ intersects its orthogonal
complement $\C x$ in all of $\C x$.

\begin{proof}
The affine vector bundle property follows directly from the
system~(\ref{eq:sys1}) or, alternatively, from the matrix representation
(\ref{eq:critideal}): fixing $x \in X \backslash X_{\rm sing}$, the fiber
$\pi_1^{-1}(x)$ equals $\{x\} \times (x+(T_x X)^\perp)$, where the second
factor is an affine  space of dimension $c$ varying smoothly
with $x$. Since $X$ is irreducible, so is $\cE_X$, and its dimension
equals $(n-c)+c=n$. For dimension reasons, the projection $\pi_2$ cannot
have positive-dimensional fibers over \generic{} data points $u$, so those
fibers are generically finite sets, of cardinality equal to $\ED(X)$.

For the last statement, note  that  the diagonal $\Delta(X):=\{(x,x)
\in \C^n \times \C^n \mid x \in X\}$ is contained in $\cE_X$. Fix a point
$x \in X \backslash \Xsing$ for which
$T_x X \cap (T_x X)^\perp = \{0\}$. Being an affine bundle over $X
\backslash \Xsing$, $\cE_X$ is smooth at the point $(x,x)$. The tangent
space $T_{(x,x)} \cE_X$ contains both the tangent space $T_{(x,x)}
\Delta(X)= \Delta(T_x X)$ and $\{0\} \times (T_xX)^\perp$. Thus the image
of the derivative at $x$ of $\pi_2: \cE_X \to \C^n$ contains both $T_x X$
and $(T_x X)^\perp$. Since these spaces have complementary dimensions
and intersect trivially by assumption, they span all of $\C^n$. Thus
the derivative of $\pi_2$ at $(x,x)$ is surjective onto $\C^n$, and
this implies that $\pi_2$ is dominant.  \end{proof}

\begin{corollary} \label{cor:unirational}
If $X$ is (uni-)rational then so is the ED correspondence $\mathcal{E}_X$.
\end{corollary}

\begin{proof}
Let $\psi: \C^{m} \rightarrow \C^n$ be a rational map that parametrizes
$X$, where $m=\dim X=n-c$. Its Jacobian $J(\psi)$ is an $n \times
m$-matrix of rational functions in the standard coordinates
$t_1,\ldots,t_m$ on $\C^m$.
The columns of $J(\psi)$ span the tangent space of $X$ at the point $\psi(t)$
for \generic{} $t \in \C^m$. The left kernel of $ J(\psi)$ is
a linear space of dimension $c$.  We can write down a basis $\{\beta_1
(t), \ldots,\beta_{c}(t)\}$ of that kernel by applying Cramer's rule
to the matrix $J(\psi)$. In particular, the $\beta_j$ will
also be rational functions in the $t_i$. Now the map
\[ \C^{m} \times \C^{c} \to \cE_X,\ (t,s) \mapsto
\biggl(\psi(t), \psi(t)+\sum_{i=1}^{c} s_i \beta_i(t) \biggr) \]
is a parametrization of $\cE_X$, which is birational
if and only if $\psi$ is birational.
\end{proof}

\begin{example}
The twisted cubic cone $X$ from Example~\ref{ex:twistedcubic} has the
parametrization $\psi: \C^2 \rightarrow  \C^4,\, (t_1,t_2) \mapsto
(t_1^3,t_1^2t_2,t_1t_2^2,t_2^3)$. We saw
that ${\rm EDdegree}(X) = 7$. Here is a parametrization of the ED
correspondence  $\mathcal{E}_X$ that is induced by the construction in the proof above:
\[ \begin{matrix}
 \C^2 \times \C^2  \,\rightarrow \,  \C^4 \times \C^4 \,\, , \quad
((t_1,t_2),(s_1,s_2)) \, \mapsto  \qquad \qquad \qquad \qquad \qquad \qquad \\
\big( (t_1^3 , t_1^2  t_2 , t_1  t_2^2 , t_2^3),\,
    (t_1^3 + s_1 t_2^2, \
    t_1^2 t_2 - 2 s_1 t_1 t_2 + s_2 t_2^2,\
    t_1 t_2^2 + s_1 t_1^2 - 2 s_2 t_1 t_2,\
    t_2^3 + s_2 t_1^2) \big).
\end{matrix}
\]
The prime ideal of $\mathcal{E}_X$ in
 $\R[x_1,x_2,x_3,x_4,u_1,u_2,u_3,u_4]$ can be computed from
  (\ref{eq:critideal}). It is minimally generated by
 seven quadrics and one quartic. It is important to note that these generators are
   homogeneous with respect to the usual $\mathbb{Z}$-grading
 but  not bi-homogeneous.

 The formulation (\ref{eq:critideal2}) leads to the subideal
 generated by all bi-homogeneous polynomials that vanish on
 $\mathcal{E}_X$. It has six minimal generators,
 three of degree $(2,0)$ and three of degree $(3,1)$.
Geometrically, this corresponds to the variety
$\mathcal{P} \cE_X \subset \PP^3 \times \C^4$ we introduce next.
$\diamondsuit$
\end{example}

If $X$ is an affine cone in $\C^n$, we consider the
closure of the image of $\cE_X \cap ((\C^n \backslash \{0\}) \times \C^n)$
 under the map $(\C^n \backslash \{0\}) \times \C^n \to \PP^{n-1}
\times \C^n,\ (x,u) \mapsto ([x],u)$. This closure is called the {\em
projective ED correspondence} of $X$, and it is denoted $\mathcal{P} \cE_X$. It has
the following properties.

\begin{theorem} \label{thm:EDproj}
Let $X \subseteq \C^n$ be an irreducible affine cone not contained in the
isotropic quadric $Q$. Then the projective ED correspondence
$\mathcal{P} \cE_X$ of $X$
is an $n$-dimensional irreducible variety in $\PP^{n-1}
\times \C^n$. It is the zero set of the ideal
\eqref{eq:critideal2}. Its projection onto the projective variety
in $\PP^{n-1}$ given by $X$
is a vector bundle over $X \backslash(\Xsing \cup Q)$ of rank $c+1$.
The fibers over \generic{} data points $u$ of its projection onto $\C^n$
are finite of cardinality equal to $\ED(X)$.
\end{theorem}

\begin{proof}
The first statement follows from Lemma~\ref{lem:critcone}: let $x \in
X \backslash (\Xsing \cup Q)$ and $u \in \C^n$. First, if $(x,u)
\in \cE_X$, then certainly $([x],u)$ lies in the variety of the ideal
\eqref{eq:critideal}. Conversely, if $([x],u)$ lies in the variety of
that ideal, then there exists a (unique) $\lambda$ such that $(\lambda x,
u) \in \cE_X$. If $\lambda$ is non-zero, then this means that $([x],u)$
lies in the projection of $\cE_X$. If $\lambda$ is zero, then $u \perp
T_x X$ and hence $(\epsilon x, \epsilon x+u) \in \cE_X$ for all $\epsilon
\in \C$. The limit of $([\epsilon x],\epsilon x + u)$ for $\epsilon \to
0$ equals $([x],u)$, so the latter point still lies in the closure of
the image of $\cE_X$, i.e., in the projective ED correspondence.
The remaining statements are proved as in the proof of
Theorem~\ref{thm:EDaffine}.
\end{proof}

We now turn our attention to the average ED degree of
a real affine variety $X$ in $\R^n$.
In applications, the data point
$u$ also lies in $\R^n$, and $u^*$ is the unique closest point to
$u$ in $X$. The quantity $\ED(X)$ measures the algebraic complexity of
writing the optimal solution $u^*$ as a function of the data $u$. But
when applying other, non-algebraic methods for finding $u^*$, the
number of {\em real-valued} critical points of $d_u$ for randomly
sampled data $u$ is of equal interest.  In contrast with the number of
complex-valued critical points, this number is typically not constant
for all \generic{} $u$, but rather constant on the connected components of
the complement of an algebraic hypersurface $\Sigma_X \subset \R^n$,
which we call the {\em ED discriminant}. To get, nevertheless, a
meaningful  count of the critical points, we propose to
average over all $u$ with respect to a measure on $\R^n$. In the remainder of
this section, we describe how to compute that average using the ED
correspondence. Our method is particularly useful in the setting of
Corollary~\ref{cor:unirational}, i.e., when $X$ and hence $\cE_X$ have
rational parametrizations.

We equip data space $\R^n$ with a volume form $\omega$ whose associated
density $|\omega|$ satisfies $\int_{\R^n} |\omega|=1$. A common choice for
$\omega$ is the standard multivariate Gaussian $\,\frac{1}{(2\pi)^{n/2}}
e^{-||x||^2/2}\, \dd x_1 \wedge \cdots \wedge \dd x_n$.  This choice is
natural when $X$ is an affine cone: in that case, the origin $0$ is a
distinguished point in $\R^n$, and the number of real critical points
will be invariant under scaling $u$.
Now we ask for the {\em expected number} of critical points of $d_u$ when $u$
is drawn from the probability distribution on $\R^n$ with density $|\omega|$. This
{\em average ED degree} of the pair $(X,\omega)$ is
\begin{equation}
\label{eq:aED1}
\aED(X,\omega)\,\,:= \,\, \int_{\R^n} \#\{\text{real critical points of } d_u \text{ on } X \} \cdot |\omega|.
\end{equation}
In the formulas below, we write
 $\cE_X$ for the set of real points of the ED
correspondence. Using the substitution rule from multivariate calculus,
we rewrite the integral in (\ref{eq:aED1}) as follows:
\begin{equation}
\label{eq:aED2}
 \aED(X,\omega)\,\,\,=\,\,
 \int_{\R^n} \# \pi_2^{-1}(u) \cdot |\omega|
 \,\,\, =  \,\, \int_{\corrR{X}} |\pi_2^*(\omega)|,
\end{equation}
where $\pi_2^*(\omega)$ is the pull-back of the volume form $\omega$
along the derivative of the map $\pi_2$.

\begin{figure}
\begin{center}
\includegraphics{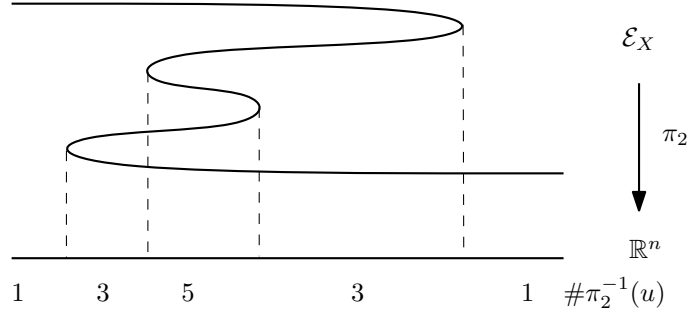}
\caption{The map from the ED correspondence $\corrR{X}$
to data space has four branch points.
The  weighted average of  the fiber sizes $1,3,5,3,1$
can be expressed as an integral over  $\corrR{X}$.} \label{fig:piV}
\end{center}
\end{figure}

See Figure~\ref{fig:piV} for a cartoon illustrating the computation in
(\ref{eq:aED2}). Note that $\pi_2^*(\omega)$ need not be a volume
form since it may vanish at some points---namely, at the {\em ramification
locus} of $\pi_2$, i.e., at points where the derivative of $\pi_2$
is not of full rank. This ramification locus is typically an algebraic
hypersurface in $\corrR{X}$,
and equal to the inverse image of the
ED discriminant~$\Sigma_X$.
The usefulness of the formula  (\ref{eq:aED2}), and a more explicit
version of it to be derived below, will depend on whether the strata
in the complement of the branch locus of $\pi_2$ are easy to describe.
We need such a description because the integrand will be a function
``without absolute values in it'' only on such open strata that lie over
the complement of $\Sigma_X$.

Suppose that we have a parametrization
$\phi:\R^n \dto \corrR{X}$ of the ED correspondence
that is generically one-to-one. For instance, if
$X$ itself is given by a birational parametrization $\psi$, then $\phi$
can be derived from $\psi$ using the method in the proof of Corollary
\ref{cor:unirational}.
We can then write the integral over $\corrR{X}$ in (\ref{eq:aED2})
more concretely as
 \begin{equation}
\label{eq:aED3}
\int_{\corrR{X}} |\pi_2^*(\omega)| = \int_{\R^n} |\phi^*
\pi_2^*(\omega)| = \int_{\R^n} |\det \TT_t (\pi_2 \circ \phi)|
\cdot f(\pi_2(\phi(t))) \cdot \dd t_1 \wedge \cdots \wedge \dd t_n.
\end{equation}
Here $f$ is the smooth (density) function on $\R^n$ such that $\omega_u=f(u) \cdot
\dd u_1 \wedge \cdots \wedge \dd u_n$. In the standard Gaussian case,
this would be $f(u) =  e^{-||u||^2/2}/(2\pi)^{n/2}$.  The determinant
in (\ref{eq:aED3}) is taken of the differential of $\pi_2 \circ \phi$.
To be fully explicit, the composition $\pi_2 \circ \phi$ is a map from
$\R^n$ to $\R^n$, and $\TT_t(\pi_2 \circ \phi)$ denotes its $n \times n$
Jacobian matrix at a point $t$ in the domain of $\phi$.

\begin{example} [ED and average ED degree of an ellipse]
\label{ex:ellipse2}
For an illustrative simple example, let $X $
denote the  ellipse in $\R^2$ with equation $x^2 + 4 y^2=4$.
We first compute $\ED(X)$. Let $(u,v) \in \R^2$ be a data
point. The tangent line  to the ellipse $X$ at $(x,y) $ has direction $(-4y,x)$.
Hence the condition that $(x,y) \in X$ is critical for $d_{(u,v)}$
translates into the equation $(u-x,v-y) \cdot (-4y,x)=0$, i.e., into
$3xy+vx-4uy=0$. For \generic{} $(u,v)$, the curve defined by the latter
equation and the ellipse intersect in $4$ points in $\C^2$, so $\ED(X)=4$.

\begin{figure}[h]
\centering
\includegraphics[scale=0.7]{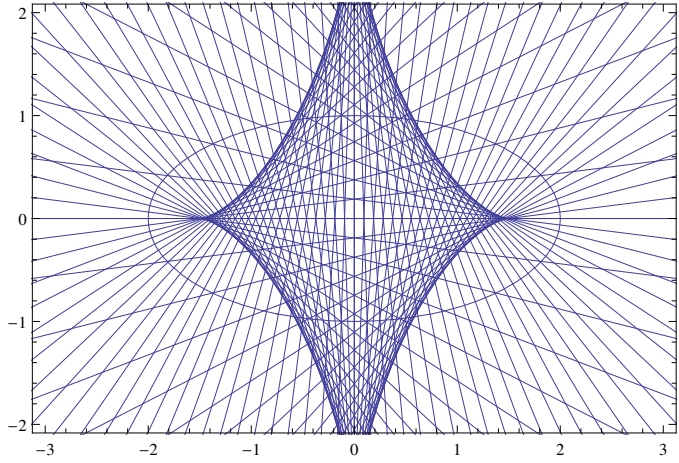}
\caption{Computing the average ED of an ellipse: the evolute
divides the plane into an inside region, where fibers or
$\pi_2$ have cardinality $4$, and an outside region, where
fibers of $\pi_2$ have  cardinality $2$. The average ED of the
ellipse is a weighted average of these numbers.}
\label{fig:edellips}
\end{figure}

Now we consider $\aED(X,\omega)$ where
$\omega=\frac{1}{2\pi}e^{-(u^2+v^2)/2} \dd u \wedge \dd v$
is the standard Gaussians centered at the midpoint $(0,0)$ of the ellipse.
Given $(x,y) \in X$, the $(u,v)$ for which $(x,y)$ is critical are precisely those
on the normal line. This is the line through $(x,y)$ with direction $(x,4y)$.
In Figure~\ref{fig:edellips} we  plotted some of these normal lines.
A colorful dynamic version of the same picture can be seen at
{\tt http://en.wikipedia.org/wiki/Evolute}.
The {\em evolute} of the ellipse $X$ is what we
named the ED discriminant. It is the sextic {\em Lam\'e curve}
$$ \Sigma_X \,= \,V( 64 u^6+48 u^4 v^2+12 u^2 v^4+v^6
-432 u^4+756 u^2 v^2-27 v^4+     972 u^2+243 v^2-729). $$
Consider the rational parametrization of $X$ given by
$\psi(t) = \left(\frac{8t}{1+4t^2}, \frac{4t^2-1}{1+4t^2} \right)$, $t \in \mathbb R$.
From $\psi$ we construct a parametrization $\phi$ of the surface
$\corrR{X}$ as in Corollary~\ref{cor:unirational}, so that
\[ \pi_2 \circ \phi: \R \times \R \to \R^2,(t,s) \mapsto
\left((s+1)\frac{8t}{1+4t^2} ,(4s+1)\frac{4t^2-1}{1+4t^2}
\right).
\]
The Jacobian determinant of $\pi_2 \circ \phi$ equals $\frac{-32(1+s+4(2s-1)t^2+16(1+s)t^4)}{(1+4t^2)^3}$, so $\aED(X)$~is
\[ \frac{1}{2\pi} \int_{-\infty}^\infty \left( \int_{-\infty}^\infty  \left|\frac{-32(1+s+4(2s{-}1)t^2+16(1{+}s)t^4)}{(1+4t^2)^3}\right| e^{\frac{-(1+4s)^2-8(7-8(-1+s)s)t^2-16(1+4s)^2t^4}{2(1+4t^2)^2}} \dd t \right) \dd s.\]
Numerical integration (using {\tt Mathematica 9}) finds the value
 $\,3.04658...$ in $0.2$ seconds.

The following experiment independently validates this
average ED degree calculation. We sample
 data points $(u,v)$ randomly from Gaussian
distribution. For each $(u,v)$ we compute the number
of real critical points, which is either $2$ or $4$,
and  we average these numbers.
The average value approaches $3.05...$, but it requires
$10^5$ samples to get two digits of accuracy.
The total running time is
 $38.7$ seconds, so much longer than the numerical integration.
\hfill $\diamondsuit$
\end{example}

\begin{example}
The cardioid $X$ from Example~\ref{ex:cardioid} can be parametrized by
\[
\psi:\R \to \mathbb{R}^2,\
t \mapsto\left(\frac{2 t^2-2}{(1 + t^2)^2},\ \frac{-4 t}{(1 + t^2)^2}\right).\]
From this we derive
the following parametrization of the ED-correspondence $\corrR{X}$:
$$
\phi:\R \times \mathbb{R}\to \mathbb{R}^2 \times \R^2,\,\,(t,s)
\mapsto \biggl(\psi(t),
\frac{2 (t^4 -1 + 4 s ( 3 t^2-1)}{(1 + t^2)^3},\ \frac{4 t (-1 - 6 s + ( 2 s-1) t^2)}{(1 + t^2)^3}\biggr).
$$
Fixing the standard Gaussian centered at $(0,0)$,
the integral (\ref{eq:aED3}) now evaluates as follows:
\[
\aED(X,\omega)\,\,=\,\,
\frac{1}{2\pi} \int_{\R^2} |\det \TT_{t,s} (\pi_2 \circ
\phi)| e^{-\frac{||\pi_2\circ\phi (t,s)||^2}{2}}\dd
t\dd s\ \,\, \approx \,\,2.8375.
\]
Thus, our measure gives little mass to the region inside the smaller cardioid
in Figure~\ref{fig:logo}. \hfill $\diamondsuit$
\end{example}

\begin{example} \label{ex:EYposi1}
Some varieties $X$ have the property that, for all data $u$, all the complex critical points 
have real coordinates.
If this holds then $\aED(X,\omega) = \ED(X)$, for any measure $|\omega|$ on data space.
One instance is the variety $X_r$ of  real $s\times t$ matrices of rank $\le r$,
by  Example~\ref{ex:eckartyoung}. We shall discuss the ED discriminant of
$X_r$ in Example \ref{ex:EYposi2}.
  \hfill $\diamondsuit$
\end{example}

We next present a family of examples where the integral (\ref{eq:aED3}) can be computed exactly.

\begin{example} \label{ex:rationalnormalaed}
We take $X$  as the cone over the {\em rational normal curve}, in a special
coordinate system, as in Example~\ref{ex:twistedcubic} and
Corollary \ref{cor:veronese}.
Fix $\R^2$ with the standard orthonormal basis $e_1,e_2$.
Let $S^n \R^2$ be the space of homogeneous
polynomials of degree $n$ in the variables $e_1,e_2$.
We identify this space with $\R^{n+1}$ by fixing the
basis $f_i:=\sqrt{\binom{n}{i}}\cdot e_1^ie_2^{n-i}$ for $ i=0,\ldots,n$. This ensures
that the natural action of the orthogonal group $\OO_2(\R)$ on polynomials in $e_1,e_2$ is
by transformations that are orthogonal with respect to the standard
inner product on $\R^{n+1}$.

Define $v,w:\mathbb{R}^2\to\mathbb{R}^2$ by $v(t_1,t_2):=t_1e_1+t_2e_2$
and $w(t_1,t_2):=t_2e_1-t_1e_2$. These two vectors
form an orthogonal basis of $\R^2$ for $(t_1,t_2) \neq (0,0)$.
Our surface $X$ is parametrized~by
\[ \psi: \R^2 \,\to\, S^n \R^2 =\R^{n+1},\quad (t_1,t_2) \,\mapsto\,
v(t_1,t_2)^n = \sum_{i=0}^n t_1^i t_2^{n-i} \sqrt{\binom{n}{i}} f_i. \]
For $n=3$, this parametrization specializes to the second parametrization
in Example \ref{ex:twistedcubic}.
Fix the standard Gaussian centered at the origin in $\R^{n+1}$.
In what follows, we shall prove
\begin{equation}
\aED(X) \,\, = \,\, \sqrt{3n-2}.
\end{equation}
We begin by parametrizing the ED correspondence,
as suggested in the proof of
Corollary~\ref{cor:unirational}.
For $(t_1,t_2) \neq (0,0)$,
the tangent space $T_{\psi(t_1,t_2)} X$ is
 spanned by $v(t_1,t_2)^n$ and $v(t_1,t_2)^{n-1}\cdot w(t_1,t_2)$. Since,
by the choice of scaling, the vectors $v^n,v^{n-1}w,\ldots,w^n$ form an orthogonal
basis of $\R^{n+1}$, we find that the orthogonal complement
 $(T_{\psi(t_1,t_2)} X)^\perp$ has the orthogonal basis
 \[w(t_1,t_2)^n,\ v(t_1,t_2)\cdot w(t_1,t_2)^{n-1},\ldots,\
 v(t_1,t_2)^{n-2}\cdot w(t_1,t_2)^2.\]
 The resulting parametrization
$ \,\phi: \R^2 \times \R^{n-1} \to \cE_X\,$ of the ED correspondence equals
$$
(t_1,t_2,s_0,...,s_{n-2}) \,\mapsto \,
\bigl(\,\psi(t_1,t_2),\ v(t_1,t_2)^n + s_0 w(t_1,t_2)^n + \cdots + s_{n-2} v(t_1,t_2)^{n-2}\cdot w(t_1,t_2)^2 \,\bigr).
$$

Next we determine the Jacobian $J=J(\pi_2 \circ \phi)$ at the
point $\psi(t_1,t_2)$. It is most convenient to do so relative to the
orthogonal basis $v(t_1,t_2),\ w(t_1,t_2),\ (1,0,\ldots,0),\ \ldots,
(0,\ldots,0,1)$ of $\R^2 \times \R^{n-1}$ and the orthogonal basis
$w(t_1,t_2)^n,\ldots,v(t_1,t_2)^n$ of $\R^{n+1}$.  Relative to these
bases,
\[ J \quad = \quad
\begin{bmatrix}
 * & * & 1 & 0 & \cdots & 0\\
 * & * & 0 & 1 & \cdots & 0\\
 \vdots & \vdots  & \vdots & \vdots & \ddots & \vdots\\
 * & * & 0 & 0 & \cdots & 1\\
 0 & n-2s_{n-2} & 0 & 0 & \cdots & 0\\
 n & * & 0 & 0 & \cdots & 0
\end{bmatrix},
\]
where the stars are irrelevant for ${\rm det}(J)$. For instance, an
infinitesimal change $v(t_1,t_2) \mapsto v(t_1,t_2)+\epsilon w(t_1,t_2)$
leads to a change $w(t_1,t_2) \mapsto w(t_1,t_2)-\epsilon
v(t_1,t_2)$ and to a  change of $\pi_2 \circ \phi$ in which
the coefficient of $\epsilon v(t_1,t_2)^{n-1}\cdot w(t_1,t_2)$ equals
$n-2s_{n-2}$. When computing the determinant of $J$, we must consider
that the chosen bases are orthogonal but not orthonormal:
the norm of $v(t_1,t_2)^i\cdot w(t_1,t_2)^{n-i}$, corresponding to the
$i$-th row, equals $\sqrt{(t_1^2 + t_2^2)^n}\binom{n}{i}^{-1/2}$; and
the norm of $v(t_1,t_2)$ and $w(t_1,t_2)$, corresponding to the first and
second column, equals $\sqrt{t_1^2+t_2^2}$. Multiplying the determinant
of the matrix above with the product of these scalars,
and dividing by the square of $\sqrt{t_1^2+t_2^2}$ for the first two
columns, we obtain the formula
\[ |\det J(\pi_2 \circ \phi)|\,\,=\,\,n \cdot |n-2s_{n-2}| \cdot (t_1^2+t_2^2)^{n(n+1)/2 - 1} \cdot \prod_{i=0}^n \binom{n}{i}^{-1/2}  . \]
Next, the squared norm of $u=\pi_2 \circ \psi(t_1,t_2,s_0,...,s_{n-2})$ equals
\[ ||u||^2=(t_1^2+t_2^2)^n \cdot \left(1 + \sum_{i=0}^{n-2} s_i^2 \binom{n}{i}^{-1}
 \right). \]
The average ED degree of $X$ relative to the standard Gaussian equals
\[
\aED(X)=\frac{1}{(2\pi)^{(n+1)/2}} \int
|\det J(\pi_2 \circ \psi)| e^{-||u||^2/2} \dd v_1 \dd v_2 \dd s_0 \cdots \dd s_{n-2}.
\]
parametrizing the regions where $\det J(\pi_2 \circ \psi)$ is positive or
negative by $s_{n-2} \in (-\infty,n/2)$ or $s_{n-2} \in (n/2,\infty)$,
this integral can be computed in closed form. Its value
equals $\sqrt{3n-2}$. Interestingly, this value
is the square root of the \generic{} ED degree
in Example~\ref{ex:rationalnormalcurve}.
 For more information see Section \ref{sec:Tensors}
and the article \cite{DH}
where  tensors of rank $1$ are treated.
\hfill $\diamondsuit$
\end{example}

\smallskip

We close this section with the remark that
 different applications require different choices of the
measure $|\omega|$ on data space. For instance, one might want to draw
$u$ from a product of intervals equipped with the uniform distribution,
or to concentrate the measure near $X$.

\section{Duality}
\label{sec:Duality}

This section deals exclusively with irreducible affine cones $X \subset
\C^n$, or, equivalently, with their corresponding projective varieties $X
\subset \PP^{n-1}$. Such a variety has a {\em dual variety} $Y:=X^*
\subset \C^n$, which is defined as follows,
where the line indicates the topological closure:
\[ Y \,\,:=\,\,
\overline{
\bigl\{y \in {\C}^n \mid \exists x \in X \backslash \Xsing: y \perp T_x X \bigr\} }. \]
See  \cite[Section 5.2.4]{RS} for an introduction to this
duality in the context of optimization.
Algorithm 5.1 in \cite{RS} explains how to compute
the ideal of $Y$ from that of $X$.

The variety $Y$ is an irreducible affine cone, so we can regard it as an irreducible
projective variety in $\PP^{n-1}$.  That projective variety parametrizes
hyperplanes tangent to $X$ at non-singular points, if one uses the
standard bilinear form on $\C^n$ to identify hyperplanes with points in
$\PP^{n-1}$.  We will prove $\ED(X)=\ED(Y)$. Moreover, for \generic{}
data $u \in \C^n$, there is a natural bijection between the critical
points of $d_u$ on the cone $X$ and the critical points of $d_u$ on the
cone $Y$. We then link this result to the literature  on
the {\em conormal variety} (cf.~\cite{Holme})
 which gives powerful techniques for
computing ED degrees of smooth varieties that intersect the
isotropic quadric $Q = V(x_1^2+\cdots + x_n^2)$
transversally.  Before dealing with the general case, we revisit the
example of the Eckart-Young Theorem.

\begin{example}\label{ex:SVDrevisited}
For the variety $X_r$ of $s\times t$ matrices ($s\le t$) of rank
$\le r$, we have $X_r^*=X_{s-r}$ \cite[Chap. 1, Prop. 4.11]{GKZ}. From
Example \ref{ex:eckartyoung} we see that ${\rm EDdegree}(X_r) = {\rm EDdegree}(X_{s-r})$.
There is a bijection between the critical points of $d_U$ on $X_r$
and on $X_{s-r}$. To see this, consider the
singular value decomposition (\ref{eq:SVD}).
For a subset $I=\{i_1,\ldots, i_r\}$ of $\{1,\ldots , s\}$, we set
\[ U_I \,\,=\,\,\, T_1 \cdot {\rm diag}(\ldots,\sigma_{i_1},\ldots,
\sigma_{i_2}, \ldots, \sigma_{i_r},\ldots) \cdot T_2,
\]
where the places of $\sigma_j$ for $j \not\in I$ have been filled with zeros
in the diagonal matrix.
 Writing $I^c$ for the complementary subset of size $s-r$, we have
$\,U=U_I+U_{I^c}$. This decomposition is orthogonal in the
sense that $\langle U_I,U_{I^c}\rangle =\textrm{tr}(U_I^tU_{I^c})=0$.
It follows that, if $U$ is real, then  $|U|^2=|U_I|^2+|U_{I^c}|^2$,
where $|U|^2=tr(U^tU)$.  As $I$ ranges over all $r$-subsets,
$U_I$ runs through the critical points of $d_U$ on the variety
$X_r$, and $U_{I^c}$ runs through the critical points of $d_U$
on the dual variety $X_{s-r}$. Since the formula above reads as
$|U|^2=|U-U_{I^c}|^2+|U-U_{I}|^2$, we conclude that
the proximity of the real critical points reverses
under this bijection. For instance,  if $U_I$ is
the real point on $X_r$ closest to $U$, then $U_{I^c}$ is
the real point on $X_{s-r}$ farthest from $U$.
For a similar result in the multiplicative context of maximum likelihood see \cite{DR}.
\hfill $ \diamondsuit$
\end{example}

The following theorem shows that the duality seen in Example
\ref{ex:SVDrevisited} holds in general.

\begin{theorem} \label{thm:dualED}
Let $X \subset \C^n$ be an irreducible affine cone, $Y \subset \C^n$
its dual variety, and  $u \in \C^n$ a \generic{} data point.
The map $x \mapsto u-x$ gives a bijection from the critical
points of $d_u$ on $X$ to the critical points of $d_u$ on
$Y$. Consequently, $\ED(X)=\ED(Y)$. Moreover, if $u$ is real, then
the map sends real critical points to real critical points, and hence
$\aED(X,\omega)=\aED(Y,\omega)$ for any volume form $\omega$.
The map is proximity-reversing: the closer a real critical
point $x$ is to the data point $u$, the further $u-x$ is from~$u$.
\end{theorem}

\begin{figure}[h]
\begin{center}
\includegraphics[scale=.6]{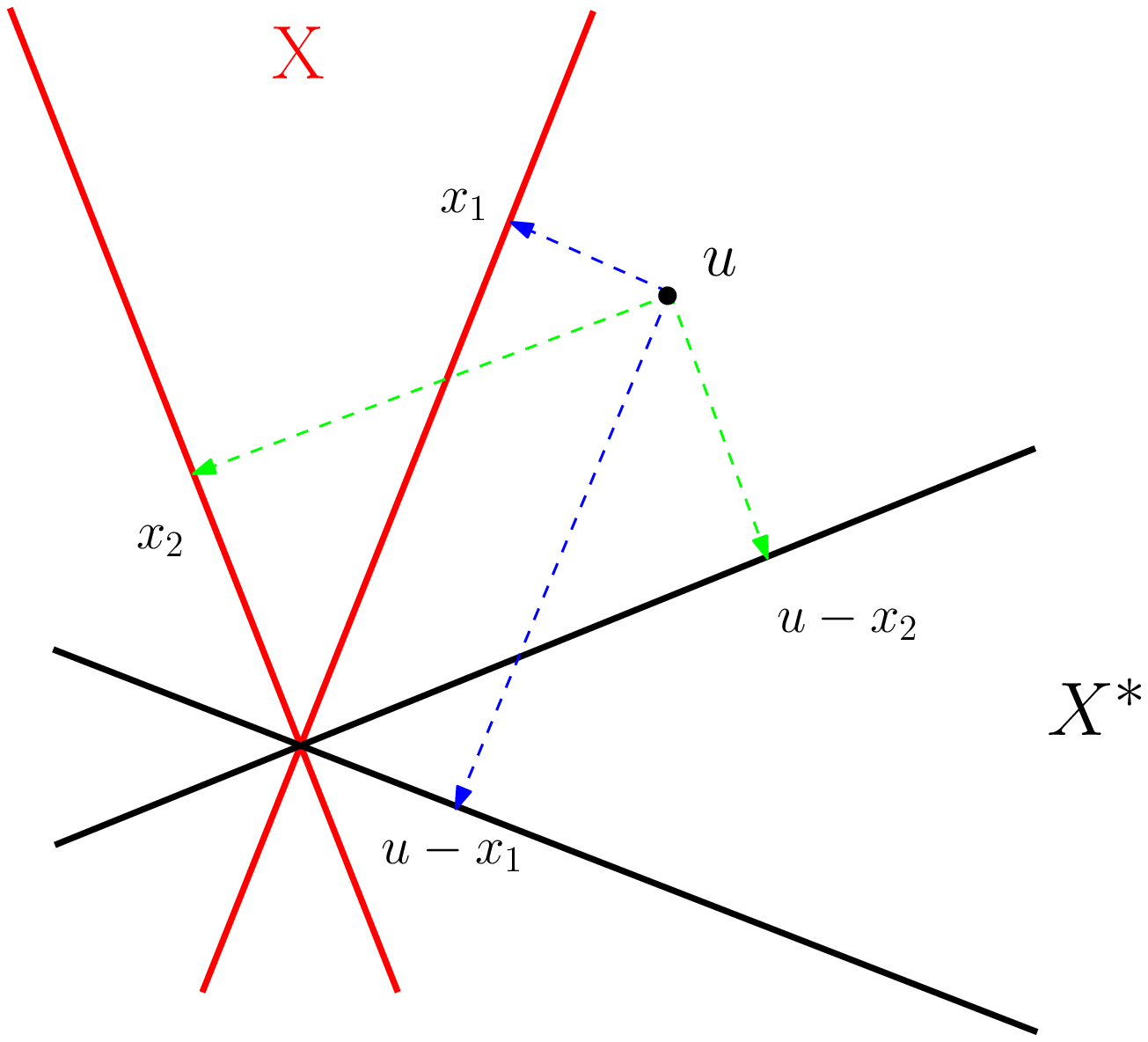} \,\,\,
\includegraphics[scale=.53]{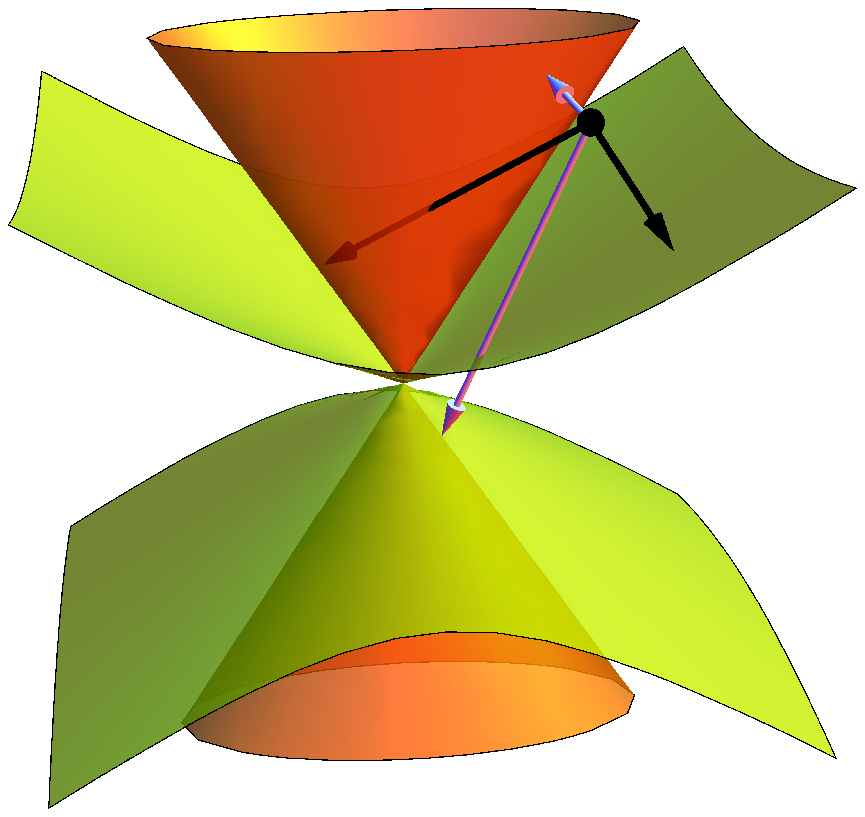}
\caption{The bijection between critical points on $X$ and
critical points on $X^*$.}
\label{fig:dualED}
\end{center}
\end{figure}

The statement of Theorem~\ref{thm:dualED}
is illustrated in Figure~\ref{fig:dualED}. On the left,
the variety $X$ is a $1$-dimensional affine cone in $\R^2$.
This $X$ is not irreducible but it visualizes
our duality in the simplest possible case.
The right picture shows the same
scenario in one dimension higher. Here
$X$ and $X^*$ are quadratic cones in $\R^3$,
corresponding to a dual pair of conics in $\PP^2$.

The proof of  Theorem~\ref{thm:dualED} uses properties of the {\em conormal variety},
which is defined as
\[ \cN_X \,\,:= \,\,\overline{ \bigl\{(x,y) \in \C^n \times \C^n \mid x \in X
\backslash \Xsing \,\,{\rm and} \,\, y \perp T_x X \bigr\}}. \]
The conormal variety is the zero set of the following ideal in $\R[x,y]$:
\begin{equation}
\label{eq:conormalideal}
N_X \,\, := \,\,
 \biggl( I_X \,+ \,
\biggl\langle \hbox{$(c+1) \times (c+1)$-minors of }
\begin{pmatrix}
y \\
J(f)
\end{pmatrix}
\biggr\rangle \biggr) : \bigl(\, I_{X_{\rm
sing}}\,\bigr)^\infty,
\end{equation}
where $f=(f_1,\ldots,f_s)$ is a system of homogeneous generators of $I_X$.
It is known that $\mathcal{N}_X$ is irreducible of dimension $n-1$. The
projection of $\mathcal{N}_X$ into the second factor $\C^n$ is the
dual variety $Y  = X^*$. Its ideal $I_Y$ is computed by elimination,
namely, by intersecting (\ref{eq:conormalideal}) with $\R[y]$.
An important property of the
conormal variety is the {\em Biduality Theorem}
\cite[Chapter 1]{GKZ}, which states that $\cN_X$
equals $\cN_Y$ up to swapping the two factors. In symbols, we have
\[ \cN_X \,=\,
\cN_Y \,= \,\overline{ \bigl\{(x,y) \in \C^n \times \C^n \mid \,y \in Y \backslash
\Ysing \,\, {\rm and} \,\, x \perp T_y Y \bigr\}}. \]
This implies  $(X^*)^* = Y^*=X$.
Thus the biduality relation in \cite[Theorem 5.13]{RS} holds.
 To keep the
symmetry in our notation, we will henceforth write
$\cN_{X,Y}$ for $\cN_X$ and $N_{X,Y}$ for $N_X$.

\begin{proof}[Proof of Theorem~\ref{thm:dualED}.]
The following is illustrated in Figure~\ref{fig:dualED}.
If $x$ is a critical point of $d_u$ on $X$, then $y:=u-x$ is orthogonal to $T_x
X$, and hence $(x,y) \in \cN_{X,Y}$. Since $u$ is general,
all $y$ thus obtained from critical points $x$ of $d_u$ are non-singular
points on $Y$. By the Biduality Theorem, we have $u-y=x \perp T_y Y$,
i.e., $y$ is a critical point of $d_u$ on $Y$. This shows that $x \mapsto
u-x$ maps critical points of $d_u$ on $X$ into critical points of $d_u$
on $Y$.  Applying the same argument to $Y$, and using that $Y^*=X$,
we find that, conversely, $y \mapsto u-y$ maps critical points of $d_u$
on $Y$ to critical points of $d_u$ on $X$. This establishes the bijection.

The consequences for $\ED(X)$ and $\aED(X,\omega)$ are
straightforward. For the last statement we observe that
$u-x \perp x \in T_x X$ for critical $x$. For $y = u-x$, this implies
\[ ||u-x||^2 + ||u-y||^2 \,\,= \,\, ||u-x||^2 + ||x||^2 \,\,= \,\,  ||u||^2 . \]
Hence the assignments that take real data points $u$ to $X$ and  $X^*$
are proximity-reversing.
\end{proof}

Duality leads us to define the {\em joint ED
correspondence}  of the cone $X$ and its dual $Y$ as
\begin{align*}
\cE_{X,Y}\,\,:&=\,\,\overline{\, \bigl\{(x,u-x,u) \in \C^n_x \times \C^n_y \times \C^n_u
\mid x \in X \backslash \Xsing \,\, {\rm and} \,\, u-x \perp T_x X \bigr\}}\\
&= \,\,\,\overline{\bigl\{(u-y,y,u) \in \C^n_x \times \C^n_y \times \C^n_u
\mid \, y \in Y \backslash \Ysing \,\, {\rm and} \,\, u-y \perp T_y Y \bigr\}}.
\end{align*}
The projection of $\cE_{X,Y}$ into $\C^n_x \times \C^n_u$ is the ED
correspondence $\cE_X$ of $X$, its projection into $\C^n_y \times
\C^n_u$ is $\cE_Y$, and its projection into $\C^n_x \times \C^n_y$ is
the conormal variety $\cN_{X,Y}$. The affine variety
$\cE_{X,Y}$ is irreducible of dimension $n$, since $\cE_X$
has these properties (by Theorem~\ref{thm:EDaffine}), and
the projection $\cE_{X,Y} \to \cE_X$ is birational with inverse
$(x,u) \mapsto (x,u-x,u)$.

Following Theorem \ref{thm:EDproj}, we also introduce the
 {\em projective joint ED correspondence}  $\mathcal{P} \cE_{X,Y}$.
 By definition, $\mathcal{P} \cE_{X,Y}$ is the closure of the image of $\,\cE_{X,Y} \cap \bigl( (\C^n
\backslash \{0\})^2 \times \C^n \bigr)$ in $\PP^{n-1}_x \times \PP^{n-1}_y \times \C^n_u$.

\begin{proposition} \label{prop:projjointED}
Let $X \subset \C^n$ be an irreducible affine cone,
let $Y \subset \C^n$ be the dual variety of $X$, and assume that neither $X$ nor $Y$ is
contained in $Q = V(q)$, where $q = x_1^2+\cdots+x_n^2$. Then $\mathcal{P} \cE_{X,Y}$ is
an irreducible $n$-dimensional variety in $\PP^{n-1}_x \times \PP^{n-1}_y \times \C^n_u$.
% BERND1002
It is the zero set of the tri-homogeneous ideal
\begin{equation}
\label{eq:jointED} \biggl( N_{X,Y} \,+ \, \biggl\langle
\hbox{$3 \times 3$-minors of the $3 \times n$-matrix}
\begin{pmatrix} \,\,u\,\, \\ x \\ y
\end{pmatrix}
\biggr\rangle \biggr) :\, \bigl\langle q(x) \cdot q(y)
\bigr\rangle^\infty \,\subset \, \R[x,y,u].
\end{equation}
\end{proposition}

\begin{proof}
The irreducibility of $\mathcal{P} \cE_{X,Y}$ follows from that of $\cE_{X,Y}$
which has the same dimension.

To see that $\mathcal{P} \cE_{X,Y}$ is defined by the ideal \eqref{eq:jointED},
note first that any point $(x,y,u)$ with $x \in X \backslash \Xsing$
and $y \perp T_xX$ and $x+y=u$ has $(x,y) \in \cN_{X,Y}$ and $\dim \langle
x,y,u \rangle \leq 2$, so that $([x],[y],u)$ is a zero of
\eqref{eq:jointED}. This shows that $\mathcal{P}\cE_{X,Y}$ is contained in the
variety of \eqref{eq:jointED}.

Conversely, let
$([x],[y],u)$ be in the variety of \eqref{eq:jointED}.
The points with $q(x)q(y) \neq 0$ are dense
in the variety of \eqref{eq:jointED}, so we may assume  $x,y \not
\in Q$. Moreover, since $(x,y) \in \cN_{X,Y}$, we may assume that $x,y$ are
non-singular points of $X$ and $Y$,  and that $x \perp T_y Y$
and $y \perp T_x X$.  This implies $x \perp y$. Since $x,y$
are not isotropic, they are linearly independent. Then $u=c x + d y$
for unique constants $c,d \in \C$. If $c,d \neq 0$, then we find that
$(cx,dy,u) \in \cE_{X,Y} \cap ((\C^n \backslash \{0\})^2 \times \C^n_u)$
and hence $([x],[y],u) \in \mathcal{P} \cE_{X,Y}$.  If $c \neq 0$ but $d=0$, then
$(cx, \epsilon y, u+\epsilon y) \in \cE_{X,Y}$ for all $\epsilon \neq 0$,
so that the limit of $([cx],[\epsilon y],u+\epsilon y)$ for $\epsilon \to
0$, which is $([x],[y],u)$, lies in $\mathcal{P} \cE_{X,Y}$. Similar arguments
apply when $d \neq 0$ but $c=0$ or when $c=d=0$.
\end{proof}

Our next result gives a formula for $\ED(X)$ in terms of the {\em polar classes}
of classical algebraic geometry \cite{Pie}.
These non-negative integers $\delta_i(X)$ are the coefficients of the class
\begin{equation}
\label{eq:bidegreeN}
[\cN_{X,Y}] \quad = \quad
 \delta_0(X) s^{n-1} t + \delta_1(X) s^{n-2} t^2  +
 \cdots + \delta_{n-2}(X) s t^{n-1}
 \end{equation}
of the conormal variety, when regarded as a subvariety of $\PP^{n-1} \times \PP^{n-1}$.
For topologists, the polynomial
(\ref{eq:bidegreeN}) is the class representing $\cN_{X,Y}$ in the cohomology ring
$H^*(\PP^{n-1} {\times} \PP^{n-1}) =  \Z[s,t]/\langle s^n,t^n \rangle$.
For commutative algebraists, it is the {\em multidegree} of the
$\Z^2$-graded ring $\R[x,y]/N_{X,Y}$. This is explained in \cite[Section 8.5]{MS}, and is
implemented in {\tt Macaulay2} with  the command {\tt multidegree}.
For geometers, the polar classes $\delta_i(X)$ have the following definition:
intersecting the $(n-2)$-dimensional subvariety $\cN_{X,Y}
\subset  \PP^{n-1} \times \PP^{n-1}$ with
an $n$-dimensional subvariety
$ L \times M $ where $L,M$ are general linear subspaces of $\PP^{n-1}$
of dimensions $n-j$ and $j$, respectively, one gets a finite number of
simple points.
The number $\delta_{j-1}(X)$ counts these points. The shift by
one is to ensure compatibility with Holme's paper \cite{Holme}.

So, for example, $\delta_0(X)$ counts the number of intersections of
$\cN_{X,Y}$ with $\PP^{n-1} \times M$ where $M$ is a general projective line.
These are the intersections of the dual variety $Y$
with $M$. Thus, if $Y$ is a hypersurface, then $\delta_0(X)$ is the
degree of $Y$, and otherwise $\delta_0(X)$ is zero.
In general, the first non-zero coefficient of (\ref{eq:bidegreeN})
is the degree of $Y$ and the last non-zero coefficient is the degree of $X$.
For all $i$, we have $\delta_i(Y)=\delta_{n-2-i}(X)$; see \cite[Theorem
2.3]{Holme}.

\begin{theorem} \label{thm:sumpolar}
If $\cN_{X,Y}$ does not intersect the diagonal
$\Delta(\PP^{n-1}) \subset \PP^{n-1} \times \PP^{n-1}$, then
$$ \ED(X) \, = \,\delta_0(X)+\cdots+\delta_{n-2}(X) \,= \,
\delta_{n-2}(Y) + \cdots + \delta_0(Y) \, = \, \ED(Y).$$
\end{theorem}

A sufficient condition for $\cN_{X,Y}$ not to intersect
$\Delta(\PP^{n-1})$ is that $X \cap Q$ is a transversal
intersection everywhere (i.e.~$X \cap Q$ is smooth)
and disjoint from $\Xsing$. Indeed, suppose that
$(x,x) \in \cN_{X,Y}$ for some $x \in X$. There exists a sequence
of points $(x_i,y_i) \in \cN_{X,Y}$ with $x_i \in X \backslash \Xsing$,
$y_i \perp T_{x_i} X$, such that $\lim_{i \to \infty}(x_i,y_i) \to
(x,x)$. Then $y_i \perp x_i$, so taking the limit we find $x \in Q$. If,
moreover, $X$ is smooth at $x$, then $T_{x_i} X$ converges to the tangent
space $T_x X$.  We conclude that $x \perp T_x X$, which means that $X$
is tangent to $Q$ at $x$.

\begin{proof}[Proof of Theorem~\ref{thm:sumpolar}.]
Denote by $Z$ the variety of linearly dependent triples $(x,y,u) \in \PP^{n-1}_x \times
\PP^{n-1}_y \times \C^n_u$.
By Proposition~\ref{prop:projjointED}, the intersection $(\cN_{X,Y} \times
\C^n) \cap Z$ contains the projective ED correspondence $\mathcal{P} \cE_{X,Y}$
as a component. The two are equal because  $(\cN_{X,Y} \times \C^n) \cap Z$ is
swept out by the $2$-dimensional vector spaces $\{(x,y)\} \times
\langle x,y \rangle$, as $(x,y)$ runs through the irreducible variety
$\cN_{X,Y}$, and hence it is irreducible. Here we are using that
$\cN_{X,Y} \cap \Delta(\PP^{n-1}) = \emptyset $.

Hence $\ED(X)$ is the length of a general
fiber of the map $\pi_3:(\cN_{X,Y} \times \C^n) \cap Z \to \C^n$. Next,
a tangent space computation shows that the intersection $(\cN_{X,Y}
\times \C^n) \cap Z$ is transversal, so an open dense subset of it is
a smooth scheme. By generic smoothness \cite[Corollary III.10.7]{Har},
the fiber $\pi_3^{-1}(u)$ over a \generic{} data point $u$ consists of simple points
only. This fiber is scheme-theoretically the same as $\cN_{X,Y} \cap Z_u$,
where $Z_u$ is the fiber in $Z$ over $u$. The cardinality of this
intersection is the coefficient of $s^{n-1} t^{n-1}$ in the product
 $[\cN_{X,Y}] \cdot [Z_u]$
in  $H^*(\PP^{n-1} {\times} \PP^{n-1}) =  \Z[s,t]/\langle s^n,t^n \rangle$.
The determinantal variety $Z_u$ has codimension $n-2$, and
$$ [Z_u] \,\, = \,\, s^{n-2}+s^{n-3}t+s^{n-4} t^2 + \cdots+s t^{n-3}+t^{n-2} . $$
This is a very special case of \cite[Corollary 16.27]{MS}.
By computing modulo $\langle s^n, t^n \rangle$, we find
$$
[\cN_{X,Y}] \cdot [Z_u] \,=\,
(\delta_0(X) s^{n-1} t + \cdots + \delta_{n-2}(X) s t^{n-1}) \cdot [Z_u]
\,= \,(\delta_0(X) + \cdots + \delta_{n-2}(X)) s^{n-1} t^{n-1}.
$$
This establishes the desired identity.
\end{proof}

\begin{remark} \rm
If $X$ and $Y$ are smooth then
$X \cap Q$ is smooth if and only if
$\Delta(\PP^{n-1}) \cap \cN_{X,Y} = \emptyset$
if and only if $Y \cap Q$ is smooth.
We do not know whether this holds when $X$ or $Y$ is singular.
Unfortunately it happens very rarely that $X$ and $Y$ are both smooth (see
\cite{Ein}).
\end{remark}

\begin{example}  \label{ex:bothmer}
Let $X$ be the variety of symmetric $s \times s$-matrices $x$ of
rank $\leq r$ and $Y$ the variety of  symmetric $s \times s$-matrices $y$
of rank $\leq s-r$.  These two determinantal varieties form a  dual
pair \cite[Example 5.15]{RS}.
Their conormal ideal $N_{X,Y}$ is generated by the relevant
minors of $x$ and $y$ and the entries of the matrix product $xy$.
The class $[N_{X,Y}]$ records the
 {\em algebraic degree of semidefinite programming}.
A formula was found by von Bothmer and Ranestad in \cite{BR}.
Using the package  {\tt Schubert2} in {\tt Macaulay2}
\cite{M2,Schubert2},
and  summing over the index $m$ in \cite[Proposition 4.1]{BR},
we obtain the following table of values for  $\ED(X)$:
$$\begin{array}{r|rrrrrrrr}
&s&=&2&3&4&5&6&7\\
\hline
r=1&&&4&13&40&121&364&1093\\
r=2&&&&13&122&1042&8683&72271\\
r=3&&&&&40&1042&23544&510835\\
r=4&&&&&&121&8683&510835\\
r=5&&&&&&&364&72271\\
r=6&&&&&&&&1093\\
\end{array}$$
In order for $X$ to satisfy the hypothesis in  Theorem \ref{thm:sumpolar},
it is essential that the coordinates are sufficiently general, so
that $X \cap Q$ is  smooth. The usual coordinates in $\C^{\binom{s+1}{2}}$
enjoy this property, and the table above records the ED degree for the
second interpretation in Example \ref{ex:closesym}.
Specifically, our number $13$ for $s=3$ and $r=2$
appeared on the right in (\ref{eq:deg3_13}).
The symmetry in the columns of our table
reflects the duality result in Theorem~\ref{thm:dualED}.
\hfill $\diamondsuit$
\end{example}

\begin{example}
Following \cite[Ex.~5.44]{RS},
{\em Cayley's cubic surface}
$X = V(f) \subset \PP^3_x$ is given~by
$$ f(x) \,\, = \,\, {\rm det}
\begin{pmatrix}  x_0  & x_1 & x_2 \\
     x_1 & x_0 & x_3 \\
      x_2 & x_3 & x_0
      \end{pmatrix}.
      $$
      Its dual in $\PP^3_y$ is the {\em quartic Steiner surface}
    $Y = V(g)$, with
     $ g\, = \, y_1^2 y_2^2+y_1^2 y_3^2+y_2^2 y_3^2
 -2y_0 y_1 y_2y_3$.
 The conormal ideal $N_{X,Y}$ is minimally
 generated by $18$ bihomogeneous polynomials
 in $\R[x,y]$:
 $$
\begin{matrix}
& \text{$f$ of degree $(3,0)$; $g$ of degree $(0,4)$;
$ \,q(x,y) =  x_0 y_0+x_1 y_1+x_2 y_2+x_3 y_3 $ of degree
$(1,1)$;}\\
&\text{six generators of degree $(1,2)$, such as $\,\,x_2 y_1 y_2+x_3 y_1
y_3+x_0 y_2 y_3 $; and}\\
&\text{nine generators of degree $(2,1)$, such as $\,x_0 x_1
y_2-x_2 x_3 y_2+x_0^2y_3-x_3^2 y_3.$}
\end{matrix}
$$
The conormal variety $\mathcal{N}_{X,Y}$ is a surface in
$\PP^3_x \times \PP^3_y$ with class $\,4 s^3 t + 6 s^2 t^2 + 3
s t^3$, and hence
$$ {\rm EDdegree}(X) \, = \, {\rm EDdegree}(Y) \, = \, 4+6+3 \, = \, 13. $$
Corollary \ref{cor:section}
relates this to the number $13$ in
(\ref{eq:deg3_13}).
The projective joint ED correspondence $\mathcal{P}\mathcal{E}_{X,Y}$ is defined by
the above equations together with the four $3 \times 3$-minors
of the matrix
$$
\begin{pmatrix}
\,\,u\,\, \\ x \\ y
\end{pmatrix} \,\, = \,\,
 \begin{pmatrix}
u_0 & u_1 & u_2 & u_3 \\
x_0 & x_1 & x_2 & x_3 \\
y_0 & y_1 & y_2 & y_3
\end{pmatrix}.
$$
For fixed scalars $u_0,u_1,u_2,u_3 \in \R$, this imposes a
codimension $2$ condition. This cuts out $13$ points in
$\,\mathcal{N}_{X,Y} \subset X \times Y \subset \PP^3_x \times
\PP^3_y$. These represent the critical points of $d_u$ on $X$
or $Y$. \hfill $\diamondsuit$
 \end{example}

Armed with Theorem~\ref{thm:sumpolar}, we can now use the  results described
in Holme's article \cite{Holme} to express the ED degree of a smooth projective variety
$X$ in terms of its Chern classes.

\begin{theorem}\label{th:chernpower2}
Let $X$ be a smooth irreducible subvariety of dimension $m$ in
$ \PP^{n-1}$, and suppose that  $X$ is transversal to
the isotropic quadric $Q$. Then
\begin{equation}
\label{eq:chernpower2}
 \ED(X) \,\,= \,\, \sum_{i=0}^{m} (-1)^i \cdot (2^{m+1-i}-1) \cdot \deg(c_i(X)).
\end{equation}
\end{theorem}

Here $c_i(X)$ is the $i$th Chern class of
the tangent bundle of $X$.
For more information on Chern classes,
and alternative formulations of
Theorem \ref{th:chernpower2}, we
refer the reader to Section \ref{sec:Chern}.

\begin{proof}
By Theorem~\ref{thm:sumpolar} we have $\ED(X)=\sum_{i=0}^{n-2}
\delta_i(X)$. We also saw that $\delta_i(X)=0$ for $i>m$, so we may let $i$
run from $0$ to $m$ instead. Substituting the expression
\[ \delta_i(X)\,\,=\,\,\sum_{j=i}^{m} (-1)^{m-j} \binom{j+1}{i+1} \deg(c_{m-j}(X)) \]
from \cite[Page 150]{Holme},
and summing over all values of the index $i$, yields the theorem.
\end{proof}

\begin{corollary} \label{cor:cherncurve}
Let $X$ be a smooth irreducible curve of degree $d$ and genus $g$ in
$ \PP^{n-1}$, and suppose that  $X$ is transversal to $Q$. Then
\begin{equation}
\label{eq:curvedegreegenus}
 \ED(X) \,\,= \,\, 3d+2g-2.
\end{equation}
\end{corollary}
\begin{proof}
We have from \cite[App. A \S 3]{Har}  that $\deg (c_0(X))=d$ and $\deg (c_1(X))=2-2g$.
\end{proof}

\begin{example} \label{ex:caution}
Consider a $2 \times 3$ matrix with entries in $\R[x_1,x_2,x_3,x_4]$ where the
first row contains general linear forms, and the second row contains general quadratic forms. The ideal $I$ generated by its three maximal minors defines a smooth irreducible curve in $\PP^3$ of degree $7$ and
genus $5$, so
 Corollary \ref{cor:cherncurve} gives
 $\textup{EDdegree}(V(I)) = 3\cdot 7 + 2\cdot 5 - 2 = 29$.
 This exceeds the bound of $27$ we would get
  by taking  $n=4, c=3, d_1=d_2=d_3=3$ in   (\ref{eq:implgeneric2}).
   However, while ideal $I$ has $s=3$ generators, the codimension
   of its variety $V(I)$ is $c=2$.
  Applying Corollary \ref{prop:implgeneric2} to
     $c=2, d_1=d_2=3$, we get the correct bound of $45$.
     This is the ED degree for the
 complete intersection of two cubics in $\PP^3$, and it exceeds $29$ as desired.
  \hfill $\diamondsuit$
\end{example}

The formula (\ref{eq:chernpower2})
 is particularly nice for a
(projectively normal) smooth {\em toric} variety $X$ in $\PP^{n-1}$. 
According to \cite{FultonToric}, this can be represented by
 a simple lattice polytope
$P \subset \R^m$ with $|P \cap \Z^m| = n$, and
$c_{m-j}(X)$ is the sum of classes corresponding to all
$j$-dimensional faces of $P$. The degree of this class
is its {\em normalized volume}. Therefore, Theorem \ref{th:chernpower2} implies

\begin{corollary} \label{cor:toricnice}
Let $X \subset \PP^{n-1}$ be an $m$-dimensional smooth projective toric
variety, with coordinates such that $X$ is transversal to $Q$.
If $\,V_j$ denotes the sum of the
normalized volumes of all $j$-dimensional faces of the
simple lattice polytope $P$ associated with $X$,  then
\[ \ED(X)\,\,= \,\, \sum_{j=0}^m (-1)^{m-j} \cdot (2^{j+1}-1) \cdot V_j. \]
\end{corollary}

\begin{example} \label{ex:rationalnormalcurve}
Consider a rational normal curve $X$ in $\PP^{n}$ in \generic{}
coordinates (we denote the ambient space as $\PP^{n}$ instead of $\PP^{n-1}$,
to compare with Example \ref{ex:rationalnormalaed}). The associated polytope $P$ is a segment of integer length
$n$. The formula above yields
\[ \ED(X)\,\,= \,\,
 (2^2-1)  \cdot V_1 -(2^1-1) \cdot V_0 \,\,= \,\, 3 n -2.
\]
In special coordinates, the ED degree can drop
to $n$; see Corollary~\ref{cor:veronese}.
Interestingly, in those special coordinates, the square root
of $3n-2$ is the average ED degree, by
Example \ref{ex:rationalnormalaed}.

All Segre varieties and Veronese varieties
are smooth toric varieties, so we can compute their ED degrees
(in \generic{} coordinates) using  Corollary \ref{cor:toricnice}.
For Veronese varieties, this can be used
to verify the  $r=1$ row in the table of
Example \ref{ex:bothmer}. For instance, for
$s=3$, the toric variety $X$ is the Veronese
surface in $\PP^5$, and the polytope
is a regular triangle with sides of lattice length $2$.
Here, $\ED(X) =
7 \cdot V_2 -  3 \cdot V_1 + V_0 \,\, = \,\,
7 \cdot 4 - 3 \cdot 6 + 3 \, = \, 13 $.
\hfill
$\diamondsuit$
\end{example}

\section{Geometric Operations} \label{sec:geomop}

Following up on our discussion of duality,
this section studies the behavior of the ED degree
of a variety under other natural operations.
We begin with the dual operations of projecting from a point and
intersecting with a hyperplane. Thereafter we discuss
  homogenizing and dehomogenizing. Geometrically,
these correspond to passing from an affine variety to its
projective closure and vice versa. We saw
in the examples of Section \ref{sec:FirstApp}
that the ED degree can go up  or go down
under homogenization. We aim to
 explain that phenomenon.

Our next two results are corollaries to Theorem~\ref{thm:sumpolar}
and results of Piene in \cite{Pie}. We work in the setting
of Section \ref{sec:Duality}, so $X$ is an irreducible projective variety in
$\PP^{n-1}$ and $X^*$ is its dual, embedded into the same
$\PP^{n-1}$ by way of the quadratic form
$q(x,y) = x_1 y_1 + \cdots + x_n y_n$.
The polar classes satisfy $\delta_i(X) = \delta_{n-2-i}(X^*)$.
These integers are zero for $i \geq {\rm dim}(X)$
and $i \leq {\rm codim}(X^*)-2$,
and they are strictly positive for all other values of  the index $i$.
The first positive $\delta_i(X)$ is the degree of $X^*$,
and the last positive $\delta_i(X)$ is the degree of $X$.
The sum of all $\delta_i(X)$ is the common
ED degree of $X$ and $X^*$.
 See \cite{Holme} and our discussion above.

Fix a \generic{} linear map $\pi :\C^n \to \C^{n-1}$.
This induces a rational map $\pi: \PP^{n-1} \dashrightarrow \PP^{n-2}$,
 whose base point lies outside $X$. The image $\pi(X)$
is an irreducible closed subvariety in $\PP^{n-2}$.
Since the projective space $\PP^{n-2}$ comes with a coordinate system
$(x_1:x_2:\cdots:x_{n-1})$, the ED degree of $\pi(X)$ is well-defined.
If ${\rm codim}(X) =1$ then $\pi(X) = \PP^{n-2}$
has ED degree  $1$ for trivial reasons.
Otherwise,  $X$ maps birationally onto  $\pi(X)$,
and the ED degree is preserved:

\begin{corollary} \label{cor:projection}
Let $X$ satisfy the assumptions of Theorem~\ref{thm:sumpolar}.
If ${\rm codim}(X) \geq 2$ then
\begin{equation}
\label{eq:projformula}
 \ED(\pi(X)) \,\,= \,\,\ED(X) .
 \end{equation}
\end{corollary}

\begin{proof}
Piene \cite{Pie} showed that $\delta_i(\pi(X)) = \delta_i(X)$  for all $i$.
Now use Theorem~\ref{thm:sumpolar}.
\end{proof}

\begin{example}
Let $I$ be the prime ideal generated by the  $2 \times 2$-minors
of the symmetric $3 \times 3$-matrix
whose six entries are \generic{}
 linear forms in $\R[x_1,x_2,x_3,x_4,x_5,x_6]$.
The {\em elimination ideal} $\,J = I \cap \R[x_1,x_2,x_3,x_4,x_5]$
is minimally generated by seven cubics.
Its variety $\pi(X) = V(J)$ is a random projection of the Veronese surface $X=V(I)$
from $\PP^5$ into $\PP^4$.
Example \ref{ex:bothmer} tells us that $\ED(X) = 13$.
By plugging $J = I_{\pi(X)}$ into the formula
(\ref{eq:critideal2}), and running {\tt Macaulay2}
as in  Example \ref{ex:M2ferma2t},
we verify $\ED(\pi(X)) = 13$. \hfill $\diamondsuit$
\end{example}

If $X$ is a variety of high codimension, then Corollary \ref{cor:projection} can be applied
repeatedly until the image $\pi(X)$ is a hypersurface. In other words,
we can take $\pi$ to be a \generic{} linear projection
$\PP^{n-1} \dashrightarrow \PP^{d}$ provided $d > {\rm dim}(X)$.
Then $\pi(X)$ also satisfies the assumptions of Theorem~\ref{thm:sumpolar},
and the formula (\ref{eq:projformula})
remains valid. This technique is particularly useful when $X$ is
a smooth toric variety as in Corollary \ref{cor:toricnice}. Here, $X$ is
 parametrized by  certain monomials,
  and $\pi(X)$ is parametrized by \generic{} linear combinations
of those monomials.

\begin{example}
Consider a surface in $\PP^3$ that is
parametrized by four homogeneous
polynomials  of degree $d$ in
three variables. That surface can be
represented as $\pi(X)$ where
$X$ is the $d$-fold Veronese embedding
of $\PP^2$ into $\PP^{\binom{d+2}{2}-1}$,
and $\pi$ is a random projection into $\PP^3$.
By applying Corollary \ref{cor:toricnice} to
the associated lattice triangle $P = {\rm conv}\{ (0,0), (0,d), (d,0)\}$,
and using Corollary \ref{cor:projection}, we find
$\,   \ED(\pi(X)) = \ED(X)   =  7 d^2 -9 d + 3$.
This is to be compared to the  number $4 d^2 - 4 d + 1$,
which is the ED degree  in  (\ref{eq:implgeneric})
for the {\em affine} surface in $\C^3$ parametrized by
three inhomogeneous polynomials of degree  $d$ in two variables.

A similar distinction arises for B\'ezier surfaces in $3$-space.
The ED degree of the affine surface
in Example \ref{ex:dokken} is $8d_1d_2 - 2d_1 - 2d_2 + 1$,
while $\ED(\pi(X)) = 14 d_1 d_2 - 6 d_1 - 6 d_1 + 4$
for  the projective surface $\pi(X)$ that is given
by four bihomogeneous polynomials $\psi_i$ of degree $(d_1,d_2)$
in $2+2$ parameters.
Here, the toric surface is $X = \PP^1 \times \PP^1$,
embedded in $\PP^{(d_1+1)(d_2+1)-1}$ by the line bundle
$\mathcal{O}(d_1,d_2)$, and the lattice polygon is the square
$P = [0,d_1] \times [0,d_2]$.
\hfill $\diamondsuit$
\end{example}

In the previous example we computed the ED degree of a variety
by expressing it as a linear projection from a high-dimensional space
with desirable combinatorial properties. This is reminiscent
of the technique of {\em lifting} in  optimization theory, where
one simplifies a problem instance by optimizing over a
higher-dimensional constraint set that projects onto
the given constraint set. It would be desirable to
develop this connection further, and to  find a more direct
proof of Corollary \ref{cor:projection} that
works for both projective and affine varieties.

The operation dual to projection is taking linear sections.
Let $H$ be a \generic{} hyperplane in $\PP^{n-1}$.
Then $X \cap H$ is a subvariety of codimension $1$ in $X$.
In particular, it lives in the same ambient space $ \PP^{n-1}$,
with the same coordinates $(x_1:\cdots:x_n)$, and this
defines the ED degree of $X \cap H$. By Bertini's
Theorem, the variety $X \cap H$ is irreducible
provided ${\rm dim}(X) \geq 2$.

\begin{corollary}
 \label{cor:section}
Let $X \subset \PP^{n-1}$ satisfy the assumptions of Theorem~\ref{thm:sumpolar}. Then
$$ \ED(X\cap H) \quad = \quad
\begin{cases}
\ED(X) - {\rm degree}(X^*) & {\rm if} \,\,{\rm codim}(X^*) = 1,\\
\qquad \ED(X) & {\rm  if} \,\,{\rm codim}(X^*) \geq 2.
\end{cases}
$$
\end{corollary}

\begin{proof}
Piene \cite{Pie} showed that  $\delta_i(X \cap H)=\delta_{i+1}(X)$ for all $i \geq 0$.
By Theorem~\ref{thm:sumpolar}, the desired ED degree is the sum of
these numbers, so it equals $\ED(X) - \delta_0(X)$.
However, we know that $\delta_0(X)$ equals
the degree of $X^*$ if $X^*$ is a hypersurface
and it is zero otherwise.
\end{proof}

\begin{example}
Let $X_r$ be the projective variety of symmetric
$3 \times 3$-matrices of rank $\leq r$.
We know that $\,X_r^* = X_{3-r}\,$ and
 $\,\ED(X_2) = \ED(X_1) = 13$.
 If $H$ is a \generic{} hyperplane in $\PP^5$ then
 $\, \ED(X_2 \cap H) = 13\,$ but $\, \ED(X_1 \cap H) = 13-3 = 10 $.
 \hfill $\diamondsuit $
\end{example}

If $X$ is a variety of high dimension in $\PP^{n-1}$ then
Corollary \ref{cor:section} can be applied repeatedly
until a \generic{} linear section is a curve. This motivates
the following definition which parallels its analogue
in the multiplicative setting of likelihood geometry \cite[\S 3]{HS}.
The {\em sectional ED degree} of the variety $X$
is the following  binary form of degree $n-1$ in $(x,u)$:
\begin{equation}
\label{eq:secEDdeg}
 \sum_{i=0}^{{\rm dim}(X)-1} \!\! \ED(X \cap L_i)
 \cdot x^{i} \cdot u^{n-1-i}
 \end{equation}
where $L_i$ is a \generic{} linear section of codimension $i$.
 Corollary \ref{cor:section} implies that, for
 varieties in \generic{} coordinates as in Theorem~\ref{thm:sumpolar},  this equals
$$ \sum_{0 \leq i \leq j < {\rm dim}(X)}  \!\! \delta_j(X)  \cdot x^{i} \cdot u^{n-1-i} . $$
It is desirable to get a better understanding
of the sectional ED degree also  for varieties in special coordinates.
For instance, in light of \cite[Conjecture 3.19]{HS}, we may ask how
(\ref{eq:secEDdeg}) is related to the bidegree of
the projective ED correspondence, or to the tridegree
of the joint projective ED correspondence.
For a concrete application, suppose that
$X$ is a determinantal variety, in the special coordinates
of the Eckart-Young Theorem (Example \ref{ex:eckartyoung}).
Minimizing the squared distance function $d_u$ over
a linear section $X \cap L_i$ is known as
{\em structured low-rank matrix approximation}.
This problem has numerous applications in
engineering; see \cite{CFP}. After this paper had been written,
a study, including computation of EDdegree
for low-rank matrices constrained in linear or affine subspaces,
was published in \cite{OSS}.

\smallskip
We now change the topic to homogenization.
Geometrically, this is the passage from an affine variety $X \subset \C^n$ to its
projective closure $\overline X \subset \PP^n$. This is a standard
operation in algebraic geometry \cite[\S 8.4]{CLO}. Homogenization
often preserves the solution set to a given geometric problem,
but the analysis is simpler in  $\PP^n$ since projective space
is compact. Algebraically, we proceed as follows.
 Given the ideal $I_X  = \langle f_1,\ldots,f_s \rangle
 \subset \R[x_1,\ldots,x_n]$, we introduce a
new variable $x_0$, representing the
hyperplane at infinity, $H_\infty = \PP^n \backslash \C^n  = V(x_0)$.
Given a polynomial $f \in \R[x_1,\ldots,x_n]$ of degree $d$,
its {\em homogenization} $\overline f \in \R[x_0,x_1,\ldots,x_n]$
is defined by $\overline{f}(x_0,\ldots,x_n)=x_0^d \cdot f(x_1/x_0,\ldots,x_n/x_0)$.
The ideal $I_{\overline X}$ of the projective variety $\overline X $
is generated by $\{ \overline{f} \,: \,f \in I_X \}$. It can be computed
(e.g.~in {\tt Macaulay2}) by saturation:
\[ I_{\overline{X}}\,\,= \,\, \langle \overline{f}_1,\ldots,
\overline{f}_s \rangle : \langle
x_0 \rangle^\infty \quad \subseteq\,\, \R[x_0,x_1,\ldots,x_n]. \]
One might naively hope that $\ED(X)=\ED(\barX)$.
But this is false in general:

\begin{example} \label{ex:upordown}
Let $X$ be the cardioid in Example \ref{ex:cardioid}.
Written in the notation above,
its projective closure is the quartic curve $\overline X \subset \PP^2$
whose defining homogeneous ideal equals
$$ I_{\overline{X}}  \,\, = \,\, \langle \,
x_0^2 x_2^2 - 2 x_0 x_1^3 - 2 x_0 x_1 x_2^2
-x_1^4 - 2x_1^2x_2^2 -x_2^4 \, \rangle. $$
For this curve we have
$$  \ED(X)\,=\, 3 \,\, < \,\, 7\, = \, \ED(\overline{X}). $$
By contrast, consider the affine surface $Y = V(x_1 x_2 - x_3) \subset \C^3$.
Its projective closure is the $2 {\times} 2$-determinant
$\overline Y = V(x_1 x_2 - x_0 x_3) \subset \PP^3$.
Here the inequality goes in the other direction:
\begin{equation}
\label{eq:otherdirection}
  \ED(Y)\,=\, 5 \,\, > \,\, 2\, = \, \ED(\overline{Y}).
  \end{equation}
The same phenomenon was seen in
our study of Hurwitz determinants in Theorem \ref{thm:Hurwitz2}.
\hfill $ \diamondsuit $
\end{example}

To explain what is going on here,  we recall that $\ED(\barX)$ is
defined as the ED degree of the affine cone over
the projective variety $\barX \subset \PP^n$, which
we also denote by $\barX$. Explicitly,
\[  \barX \,\, = \,\, \{\,(t,tx) \mid  x \in X , t \in \C\,\} \,\, \subset \,\, \C^{n+1}. \]
The ED degree of $\barX$  is for the fixed quadratic form $x_0^2 +x_1^2 +
\cdots + x_n^2$ that cuts out the isotropic quadric $Q \subset
\PP^n$. This is just one of the infinitely many quadratic
forms on $\C^{n+1}$ that restrict to the given form
$x \cdot x = x_1^2 +
\cdots + x_n^2$ on $\C^n$. That is one reason why the ED degrees of $X$
and of $\barX$ are not as closely related as one might hope. Nevertheless,
we will now make the relation more explicit.  The affine variety $X$ is
identified with the intersection of the cone $\barX$ with the hyperplane
$\{x_0=1\}$. Its part at infinity is denoted
 $\,X_\infty:=\barX \cap H_\infty$.

The data point $(1,0) \in \C^{n+1}$ plays a special role, since it is
the orthogonal projection of the vertex $(0,0)$ of the cone
$\barX$ onto the affine hyperplane $\{x_0=1\}$. The following lemma relates
the critical points for $u = 0$ on $X$ to the critical points for $u = (1,0)$ on $\barX$.

\begin{lemma} \label{lem:10}
Assume that all critical points of $d_0$ on $X$ satisfy
$x \cdot x \neq -1$. Then the map
\[  x \,\mapsto\, \left(\frac{1}{1+(x \cdot x)}, \frac{1}{1+(x \cdot x)} x\right)  \]
is a bijection from the critical points of $d_0$ on $X$ to the critical
points of $d_{(1,0)}$ on $\barX \backslash X_\infty$.
\end{lemma}

\begin{proof}
Let $t \in \C \backslash \{0\}$ and $x \in X \backslash \Xsing$. The point $(t,tx) \in \barX$
is critical for $d_{(1,0)}$ if and only if $(1-t,-tx)$ is perpendicular to
$T_{(t,tx)}\barX$. That space is spanned by $\{0\} \times T_x X$
and $(1,x)$. Hence $(1-t,-tx)$ is perpendicular to  $T_{(t,tx)}\barX$ if and only
if $x \perp T_x X$ and $(1-t) - t(x\cdot x)=0$. The first
condition says that $x$ is critical for $d_0$, and the
second gives $t  = 1/(1+(x \cdot x))$.
\end{proof}

If, under the assumptions in Lemma \ref{lem:10}, the number of critical points
of $d_0$ equals the ED degree of $X$, then we can conclude
$\ED(X) \leq \ED(\barX)$, with equality if none of the critical points
of $d_{(1,0)}$ on $\barX$ lies at infinity. To formulate a
condition that guarantees equality, we fix the isotropic quadric
$\,Q_\infty = \{ x_1^2 + \cdots + x_n^2=0\}\,$ in $ H_\infty$.
Our condition is:
\begin{equation}
\label{eq:suffcondition}
 \hbox{The intersections $\,X_\infty=\barX \cap H_\infty\,$
and $\,X_\infty \cap Q_\infty\,$ are both transversal.}
\end{equation}

\begin{lemma} \label{lem:TransverseQinfty}
If (\ref{eq:suffcondition}) holds then none of the
critical points of $d_{(1,0)}$ on $\barX$ lies in $X_\infty$.
\end{lemma}

\begin{proof}
Arguing by contradiction, suppose that $(0,x_\infty) \in X_\infty$
is a critical point of $d_{(1,0)}$ on $\barX$. Then $(1,-x_\infty)$
is perpendicular to $T_{(0,x_\infty)} \barX$, and hence $(0,x_\infty)$
is perpendicular to $H_\infty \cap T_{(0,x_\infty)} \barX$. By
transversality of $\barX$ and $H_\infty$, the latter is the tangent
space to $X_\infty$ at $(0,x_\infty)$. Hence
$T_{(0,x_\infty)} X_\infty$ is contained in $(0,x_\infty)^\perp$, and
$X_\infty$ is tangent to $Q_\infty$ at $(0,x_\infty)$.
\end{proof}

Fix  $v \in \C^n$ and consider
the affine translate $X_v:=X - v=\{x-v \mid x \in X\}$. Its projective
closure $\barX_v$  is isomorphic to $\barX$ as a projective variety in $\PP^n$.
However, the metric properties of the corresponding cones in $\C^{n+1}$ are
rather different. While $\ED(X_v)=\ED(X)$ holds trivially,
it is possible that   $\ED(\barX_v) \neq \ED(\barX)$.
Here is a simple example:

\begin{example}
Consider the unit circle $X = \{x_1^2+x_2^2=1\}$ in the plane.
Then $\ED(X)=\ED(\barX)=2$.
 For \generic{} $v \in \R^2$, the translated circle
 $X_v$ has $\ED(\barX_v)=4$.
$\diamondsuit$
\end{example}

Affine translation sheds light on the behavior of the
ED degree under homogenization.

\begin{proposition} \label{prop:generictranslates}
Let $X$ be an irreducible variety in $\C^n$, and let $v \in
\C^n$ be a \generic{} vector. Then $\ED(X)
\leq \ED(\barX_v)$, and equality holds if
the hypothesis (\ref{eq:suffcondition}) is satisfied.
\end{proposition}

The hypothesis (\ref{eq:suffcondition}) simply says
that $X_\infty$ and $X_\infty \cap Q_\infty$ are smooth.
Note that this
does not depend on  the extension of the quadric $Q_\infty$  to $\C^{n+1}$.

\begin{proof}
Since translation of affine varieties preserves ED degree, the inequality
follows from Lemma~\ref{lem:10} provided  $x'\cdot x' \not= -1$
for all critical points $x'$ for $d_0$ on $X_v$.
These are the points $x'=x-v$ with $x$ critical for $d_v$, i.e., with $(x,v) \in
\cE_X$. The expression $ (x-v) \cdot (x-v)$ is not constant $-1$
on the irreducible variety $\cE_X$, because  it is zero on
the diagonal $\Delta(X) \subset \cE_X$. As a consequence, the variety
of pairs $(x,v) \in \cE_X$ with $(x-v) \cdot (x-v)=-1$ has dimension
$\leq n-1$. In particular, it does not project dominantly onto the second factor
$\C^n$. Taking $v$ outside that projection, and such that the number
of critical points of $d_v$ on $X$ is equal to $\ED(X)$, ensures that
we can apply Lemma~\ref{lem:10}. The second statement follows from
Lemma~\ref{lem:TransverseQinfty} applied to $X_v$ and the fact that $X$
and $X_v$ have the same behavior at infinity.
\end{proof}

Our main result on homogenization links the discussion above to the
polar classes of $\barX$.

\begin{theorem} \label{thm:affineproj}
For any irreducible affine variety $X$ in $\C^n$ we have the two  inequalities
\[ \ED(X) \,\,\leq \,\,\sum_{i=0}^{n-1} \delta_i(\barX)  \quad
\text{ and } \quad
\ED(\barX) \,\,\leq \,\,\sum_{i=0}^{n-1} \delta_i(\barX), \]
with equality on the left if (\ref{eq:suffcondition}) holds, and  equality on the
right if  the conormal variety $\cN_{\barX}$ is disjoint from the diagonal $\Delta(\PP^n)$
in $\PP^n_x \times \PP^n_y$.
The equality on the right holds in particular if
$X\cap Q$ is smooth and disjoint from $X_{sing}$ (see the statement after Theorem \ref{thm:sumpolar}).
\end{theorem}

\begin{proof}
We claim that for \generic{} $v \in \C^n$ the conormal variety
$\cN_{\barX_v}$ does not intersect $\Delta(\PP^n)$. For this we need to
understand how $\cN_{\barX_v}$ changes with $v$. The $(1+n)
\times (1+n)$ matrix
\[ A_v:=\begin{pmatrix} 1 & 0 \\ -v & I_n \end{pmatrix} \]
defines an automorphism $\PP^n_x \to \PP^n_x$ that maps $\barX$
isomorphically onto $\barX_v$. The second factor $\PP^n_y$
is the dual of $\PP^n_x$ and hence
transforms contragradiently, i.e., by the matrix $A_v^{-T}$. Hence the
pair of matrices $(A_v,A_v^{-T})$ maps $\cN_{\barX}$ isomorphically onto
$\cN_{\barX_v}$. Consider the variety
\[ Z\,:=\, \bigl\{\,(x,y,v) \in \cN_{\barX} \times \C^n \mid A_v x = A_v^{-T} y \bigr\}. \]
For fixed $(x,y)=((x_0:x_\infty),(y_0:y_\infty)) \in \cN_{\barX}$ with
$x_0 \neq 0$, the equations defining $Z$ read
\[ x_0 = c (y_0+v^T y_\infty) \quad \text{ and } \quad -x_0 v+x_\infty=c y_\infty,
\]
for $v \in \C^n$ and a scalar $c$ reflecting that we work in projective
space.  The second equation expresses $v$ in $c,x,y$.
Substituting that expression into the first equation gives a system
for $c$ with at most $2$ solutions. This shows that
 $\dim Z$ is at most $\dim \cN_{\barX}=n-1$, so the
image of $Z$ in $\C^n$ is contained in a proper subvariety of $\C^n$. For
any $v$ outside that subvariety, $\cN_{\barX_v} \cap \Delta(\PP^n) = \emptyset $.
For those $v$,
Theorem~\ref{thm:sumpolar} implies that
$\ED(\barX_v)$ is the sum of the polar classes of $\barX_v$,
which are also those of $\barX$ since they are projective invariants.
Since  $\ED(\barX_v)$ can only go down as
$v$ approaches a limit point, this yields the second inequality,
as well as the sufficient condition for equality there. By applying
Proposition~\ref{prop:generictranslates}, we establish the first
inequality, as well as the sufficient condition (\ref{eq:suffcondition}) for equality.
\end{proof}

\begin{example}
Consider the quadric surface $\overline Y = V(x_0 x_3-x_1 x_2) \subset
\PP^3$ from Example \ref{ex:upordown}. This is the toric
variety whose polytope $P$ is the unit square.
By Corollary~\ref{cor:toricnice}, the sum of the polar classes
equals $7 V_2 - 3 V_1 + V_0 = 14-12+4=6$. Comparing this with
(\ref{eq:otherdirection}),
we find that neither of the two inequalities
in Theorem~\ref{thm:affineproj} is an equality. This is consistent with
the fact that $Y_\infty:=\overline{Y} \cap H_\infty = V(x_1 x_2)$ is not
smooth at the point $(0:0:0:1)$ and the fact that $\overline{Y}$ and
$Q$ are tangent at the four points $(1:a_1:a_2:a_1 a_2)$ with $a_1,a_2 = \pm i$.
\hfill $\diamondsuit$
\end{example}

\begin{example}
Consider the threefold $\,\overline Z = V(x_1 x_4 - x_2
x_3 - x_0^2 - x_0x_1)\,$ in $\PP^4$. Then $Z_\infty$ is isomorphic to
$\overline{Y}$ from the previous example and smooth in $\PP^3$, but
$Z_\infty \cap Q_\infty$ is isomorphic to the $\overline{Y} \cap Q$
from the previous example and hence has four non-reduced points.  Here, we have
$\,\ED(Z) = 4 < 8 = \ED(\overline Z) = \sum_{i=0}^3 \delta_i(\overline Z)$.
If we replace $x_1 x_4$ by $2 x_1 x_4$ in the equation
defining $\overline Z$, then the four non-reduced points
disappear. Now $Z_\infty \cap Q_\infty$ is smooth, we have
$\ED(Z) = 8$, and both inequalities in Theorem~\ref{thm:affineproj}
hold with equality.
\hfill $\diamondsuit$
\end{example}

\begin{example}
Let $X$ be the cardioid
from Examples \ref{ex:cardioid} and  \ref{ex:upordown}.
This curve violates both  conditions for equality in
Theorem~\ref{thm:affineproj}. Here
$X_\infty = V(x_1^4+2x_1^2x_2^2+x_2^4) $ agrees with  $Q_\infty =
V(x_1^2+x_2^2)$ as a subset of $  H_\infty \simeq \PP^1$, but it has
multiplicity two at the two points.
\hfill $\diamondsuit$
\end{example}

\section{ED discriminant and Chern Classes}
\label{sec:Chern}

Catenese and Trifogli \cite{CT, Tri} studied ED discriminants
under their classical name {\em focal loci}. We present
some of their results, including a formula for the ED degree
in terms of Chern classes, and we discuss a range of applications.
We work in the projective setting, so  $X$ is
a subvariety of $ \PP^{n-1}$, equipped with homogeneous coordinates
$(x_1:\ldots:x_n)$ and $\sP\mathcal{E}_X
\subset \PP^{n-1}_x \times \C^{n}_u$ is its
projective ED correspondence.
By Theorem \ref{thm:EDproj}, the ED degree
is the size of the general fiber of the map
$\sP\mathcal{E}_X \rightarrow \C^{n}_u$.
The branch locus
  of this map is the closure of
 the set of data points $u$ for which there are fewer
 than $\ED(X)$ complex critical points.
Since the variety $\sP\mathcal{E}_X
\subset \PP^{n-1}_x \times \C^{n}_u$ is defined by bihomogeneous equations
in $x$, $u$, also the branch locus is defined by homogeneous equations
and it is a cone in $\C^{n}_u$.
Hence the branch locus defines a projective variety $\Sigma_X\subset\PP^{n-1}_u$,
which we call the {\em ED discriminant}.
The ED discriminant $\Sigma_X$    is typically
a hypersurface,
by the Nagata-Zariski Purity Theorem,
and we are interested
in its degree and defining polynomial.

\begin{remark} \rm
In applications, the {\em uniqueness} of the closest real-valued point
$u^* \in X$ to a given data point $u$ is relevant. In many cases,
e.g. for symmetric tensors of rank one \cite{Friedland}, this closest
point is unique for $u$ outside an algebraic hypersurface
that strictly contains $\Sigma_X$.
\end{remark}

\begin{example} \label{ex:quadric2}
Let $n=4$ and consider the
quadric surface $X = V(x_1x_4- 2 x_2 x_3) \subset \PP^3_x$. This
is the $2 \times 2$-determinant in general coordinates, so
$\ED(X) = 6$. The ED discriminant
 is a irreducible surface of degree $12$
in $\PP^3_u$.
Its defining polynomial  has $119$ terms:
$$ \begin{matrix}
 \Sigma_X\, = \,
65536u_1^{12}+835584 u_1^{10} u_2^2+835584 u_1^{10} u_3^2
-835584 u_1^{10} u_4^2
+9707520 u_1^9 u_2 u_3 u_4 \\
+3747840 u_1^8 u_2^4
-7294464 u_1^8 u_2^2 u_3^2 + \,\cdots \,
+835584 u_3^2 u_4^{10}+65536 u_4^{12}. \,\,\,
\end{matrix}
$$
This ED discriminant can be computed using the following {\tt Macaulay2} code:
\begin{verbatim}
R = QQ[x1,x2,x3,x4,u1,u2,u3,u4];  f = x1*x4-2*x2*x3;
EX = ideal(f) + minors(3,matrix {{u1,u2,u3,u4},{x1,x2,x3,x4},
 {diff(x1,f),diff(x2,f),diff(x3,f),diff(x4,f)} });
g = first first entries gens eliminate({x3,x4},EX);
toString factor discriminant(g,x2)
\end{verbatim}
Here {\tt EX} is the ideal of the ED correspondence in $\PP^3_x \times \PP^3_u$.
The command {\tt eliminate} maps that threefold into
$\PP^1_{(x_1:x_2)} \times \PP^3_u$. We print
the discriminant of that hypersurface over $\PP^3$.
\hfill $\diamondsuit$
\end{example}

If $X$ is a general hypersurface of degree $d$ in $\PP^{n-1}$ then,
by Corollary \ref{prop:implgeneric2},
\begin{equation}
\label{eq:EDhypersurface}
\ED(X) \,\,= \,\, d \cdot  \frac{(d-1)^{n-1}-1}{d-2} .
\end{equation}
Trifogli \cite{Tri} determined the degree of the
ED discriminant $\Sigma_X$ for such a
hypersurface $X$:

\begin{theorem}[Trifogli]
If $X$ is a general hypersurface of degree $d$ in $\PP^{n-1}$ then
\begin{equation}
\label{eq:EDtrifogli}
 {\rm degree}(\Sigma_X) \,\, = \,\, d(n-2)(d-1)^{n-2} \,+\, 2d(d-1) \frac{(d-1)^{n-2}-1}{d-2} .
 \end{equation}
\end{theorem}

\begin{example} \label{ex:trifogliEX}
A general plane curve $X$ has
$\ED(X) = d^2$ and ${\rm degree}(\Sigma_X) = 3 d(d-1)$.
These are the numbers seen for
the ellipse $(d=2)$ in Example \ref{ex:ellipse2}.
For a plane quartic $X$, we expect
$\ED(X) = 16$ and ${\rm degree}(\Sigma_X)= 36$,
in contrast to the numbers $3$ and $4$
for the cardioid in Example \ref{ex:cardioid}.
A general surface in $\PP^3$ has
$\ED(X) = d(d^2-d+1)$
and ${\rm degree}(\Sigma_X) =  2d(d-1)(2d-1)$.
For quadrics $(d=2)$ we get
$6$ and $12$, as in Example \ref{ex:quadric2}.~$\diamondsuit$
\end{example}

\begin{example}
The ED discriminant $\Sigma_X$  of a plane curve $X$
was already studied in the 19th century under the name {\it evolute}.
%or {\it caustic}.
Salmon  \cite[page 96, art.~112]{Sal} showed that a curve $X \subset
\PP^2$ of degree $d$ with with $\delta$ ordinary nodes and $k$ ordinary
cusps has ${\rm degree}(\Sigma_X) = 3 d^2-3d-6\delta-8k$.  For affine
$X \subset \C^2$, the same holds provided that $\overline{X} \subseteq
\PP^2$ is not tangent to the line $H_\infty$ and neither of the two
isotropic points on $H_\infty$ is on $\overline{X}$. Curves with more
general singularities are considered in \cite{Cat, JP} in the context of {\em
caustics}, which are closely related to evolutes.
\hfill $\diamondsuit$
\end{example}

\begin{example}
\label{ex:EYposi2}
Let $X_r$ be the determinantal variety of Examples \ref{ex:eckartyoung}, \ref{ex:SVDrevisited}
and \ref{ex:EYposi1}.  The ED discriminant $\Sigma_{X_r}$ does not depend on $r$ and
equals the discriminant of the characteristic polynomial of the symmetric matrix $UU^t$.
This polynomial has been expressed as a sum of squares in \cite{Ily}.
The set of real points in the hypersurface $V(\Sigma_{Xr})$  has codimension two
in the space of real $s\times t$ matrices; see \cite[\S 7.5]{Stu}.
This explains why the complement of  this ED discriminant in the space of real matrices is connected. In particular, if $U$ is real then all critical points are real, hence $\aED(X_r)={s\choose r}$.
A computation reveals that
 $\Sigma_{X_r}$ is reducible for $s = 2$. It has two components
if $t \geq 3$, and it has four components if $t=2$.
\hfill $\diamondsuit$
\end{example}

We  comment on the relation between duality and the
ED discriminant $\Sigma_X$. Recall that $\Sigma_X$ is the projectivization of the branch locus of the
covering $\sP \mathcal{E}_X \rightarrow \C^n_u$. By the
results in Section \ref{sec:Duality}, this is also the
branch locus of $\sP \mathcal{E}_{X,Y} \rightarrow \C^n_u$,
and hence also of $\sP \mathcal{E}_Y \rightarrow \C^n_u$. This implies
that the ED discriminant of a variety $X$ agrees with that of its dual
variety $Y=X^*$.

\begin{example}
Let $X \subset \PP^2_x$ denote the cubic Fermat curve given by
$x_0^3+x_1^3+x_2^3 = 0 $. Its dual $Y$ is the sextic curve in
$\PP^2_y$ that is defined by $\, y_0^6 +y_1^6 + y_2^6 -2 y_0^3
y_1^3 -2 y_0^3 y_2^3 -2 y_1^3 y_2^3 $. This pair of curves
satisfies ${\rm EDdegree}(X) = {\rm EDdegree}(Y) = 9$. The ED
discriminant $\Sigma_X = \Sigma_Y$ is an irreducible curve of
degree $18$ in $\PP^2_u$. Its defining polynomial has $184$
terms:
$$
\Sigma_X = 4 u_0^{18}-204 u_0^{16} u_1^2+588 u_0^{15} u_1^3
-495 u_0^{14} u_1^4+2040 u_0^{13} u_1^5 -2254 u_0^{12}
u_1^6+2622 u_0^{11} u_1^7 + \cdots +  4 u_2^{18}.
$$
The computation  of the ED discriminant  for
larger examples  in {\tt Macaulay2} is difficult.
 \hfill $\diamondsuit$
\end{example}

The formulas \eqref{eq:EDhypersurface} and \eqref{eq:EDtrifogli} are best
understood and derived using
modern intersection theory; see \cite{Fulton} or  \cite[Appendix A]{Har}.
That theory goes far beyond the
techniques from \cite{CLO} used in the earlier sections
but is indispensable for more general formulas,
especially for varieties $X$ of codimension $\geq 2$. We
briefly sketch some of the required vector bundle techniques.

A vector bundle $\cE \to X$ on a smooth, $m$-dimensional projective
variety $X$ has a {\em total Chern class} $c(\cE)=c_0(\cE)
+ \ldots + c_m(\cE)$, which resides in the cohomology ring
$H^*(X)=\bigoplus_{i=0}^{m} H^{2i}(X)$. In particular, the {\em top
Chern class} $c_m(\cE)$ is an integer scalar multiple of the class of
a point, and that integer is commonly denoted $\int c(\cE)$.  If $\cE$
has rank equal to $\dim X=m$, and if $s:X \to \cE$ is a global section
for which $V(s):=\{x \in X \mid s(x)=0\}$ consists of finitely many
simple points, then the cardinality of $V(s)$ equals $\int c(\cE)$.
To apply this to the computation of ED degrees, we shall find $\cE$ and
$s$ such that the variety $V(s)$ is the set of critical points of $d_u$,
and then compute $\int c(\cE)$ using vector bundle tools. Among these
tools are {\em Whitney's sum formula} $c(\cE)=c(\cE') \cdot c(\cE'')$
for any exact sequence $0 \to \cE' \to \cE \to \cE'' \to 0$ of vector
bundles on $X$, and the fact that the total Chern class of the pull-back
of $\cE$ under a morphism $X' \to X$ is the image of $c(\cE)$ under the
ring homomorphism $H^*(X) \to H^*(X')$.

Here is our repertoire of  vector bundles on $X$:
the trivial bundle $X \times \C^n$ of rank $n$; the {\em
tautological line bundle} pulled back from $\PP^{n-1}$, which is $\cR_X := \{(x,v) \in
X \times \C^{n} \mid v \in x\}$
(also often denoted by $\sO_X(-1)$,
while the dual $\cR_X^*$ is denoted by $\sO_X(1)$);
the {\em tangent bundle} $TX$ whose fibers are the tangent spaces
$T_xX$; the {\em cotangent bundle} $T^*X$ whose fibers are
their duals $(T_xX)^*$; and the {\em normal bundle } $N_X$ whose fibers are the
quotient $T_x\PP^n/T_xX$. From these building blocks, we can construct new
vector bundles using direct sums, tensor products, quotients, duals,
and orthogonal complements inside the trivial bundle $X \times \C^n$.

\begin{theorem}\label{edformula} {\rm (Catanese-Trifogli)}
Let $X$ be an irreducible smooth subvariety of $\PP^{n-1}$
and assume that $X$
intersects the isotropic quadric $Q= V(x_1^2+\cdots+x_{n}^2)$
transversally, i.e.~$X \cap Q$ is smooth.
Then the EDdegree of $X$ can be computed
in $H^*(X)$ by either of the expressions
\begin{equation}
\label{eq:CTformula}
 {\rm EDdegree}(X) \,\,\, = \,\,\,
\int \frac{c(\cR^*_X) \cdot c(T^*X \otimes \cR^*_X)}{c(\cR_X)}
\quad = \quad
\int\frac{1}{c(\cR_X) \cdot c(N^*_X\otimes \cR^*_X)}.
\end{equation}
\end{theorem}

\begin{proof}
The first expression is  stated
after Remark 3 on page 6026 in \cite{CT}, as a formula for
the {\em inverse} of the total Chern
class of what they call {\em Euclidean normal bundle} (for simplicity we tensor it by $\cR_X^*$, differently from \cite{CT}).  The total space of that
bundle,  called {\em normal variety} in \cite{CT, Tri}, is precisely our
projective ED correspondence $\sP \mathcal{E}_X$ from
Theorem~\ref{thm:EDproj}.

A \generic{} data point $u \in \C^n$ gives rise to a section $x \mapsto [(x,u)]$
of the quotient bundle $(X \times \C^n) / \sP \mathcal{E}_X$, whose zero
set is exactly the set of critical points of $d_u$.  By Whitney's sum
formula, the total Chern class of this quotient is $1/c(\sP \cE_X)$. This
explains the inverse and the first formula. The second formula is seen using the identity
 \begin{equation}
 \label{eq:usingtheidentity}
 \frac{1}{c(N^*_X\otimes\cR_X^*)}\,=\,\frac{c(T^* X \otimes\cR_X^*)}{c(T^*\PP^{n-1}\otimes\cR_X^*)}\,=\,c(\cR_X^*)\cdot c(T^* X \otimes\cR_X^*),
 \end{equation}
where the second equality follows from the Euler sequence
\cite[Example II.8.20.1]{Har}. \end{proof}

\begin{remark} \rm
The ED degree of a smooth projective variety $X$
can also be interpreted as the top Segre class \cite{Fulton} of the Euclidean normal bundle of $X$.
\end{remark}

We shall now relate this discussion to the earlier formula in
Section \ref{sec:Duality}, by offering a second proof of Theorem
\ref{edformula}. This proof is based on a Chern class computation and
Theorem~\ref{th:chernpower2}, and hence independent of the proof by Catanese
and Trifogli.

\begin{proof}[Second proof of Theorem~\ref{edformula}.]
If $\cE$ is a vector bundle of rank $m$ and $\sL$ is a line bundle then
\begin{equation}\label{eq:tensorchern}
c_k(\cE\otimes
\sL)\,\,=\,\,\sum_{i=0}^k{{r-i}\choose{k-i}}c_i(\cE)c_1(\sL)^{k-i}.
\end{equation}
This formula is \cite[Example 3.2.2]{Fulton}.
By definition, we have $c_i(X) = (-1)^i c_i(T^*X)$.
Setting $c_1(\cR^*_X)=h$, the formula (\ref{eq:tensorchern}) implies
$$c(T^*X \otimes \cR^*_X) \,= \, \sum_{k=0}^m\sum_{i=0}^k
{{m-i}\choose{k-i}}(-1)^ic_i(X)h^{k-i}
\, = \,\sum_{i=0}^m(-1)^ic_i(X)\sum_{t=0}^{m-i} {{m-i}\choose t}h^t. $$
We have  $c(\cR^*_X)=1+h$ and
$1/c(\cR_X)=1/(1-h) = \sum_{i=0}^mh^i$.
The equation above implies
$$\frac{c(\cR^*_X) \cdot c(T^*X \otimes \cR^*_X)}{c(\cR_X)} \,\,=\,\,
\sum_{i=0}^m(-1)^ic_i(X)\left(\sum_{t=0}^{m-i}
{{m-i}\choose t}h^t\right)\left(1+2\sum_{j=1}^mh^j\right). $$
The integral on the left hand side in  (\ref{eq:CTformula})
is the coefficient of $h^{m-i}$ in the polynomial in $h$
that is obtained by multiplying
the two parenthesized sums.
That coefficient equals
$$ 1+2\sum_{j=0}^{m-i-1}{{m-i}\choose j} \,\, = \,\,\, 2^{m-i+1}-1.$$
We conclude that Theorem \ref{th:chernpower2} is in fact
equivalent to the first formula in Theorem \ref{edformula}.
The second formula follows from (\ref{eq:usingtheidentity}),
as argued above.
\end{proof}

The Catanese-Trifogli formula in (\ref{eq:CTformula}) is
most useful when
$X$ has low codimension. In that case, we compute the relevant
class in the cohomology ring of the ambient projective space $\PP^n$,
and pull back to $X$. This yields the following proof of Proposition
\ref{prop:implgeneric}.

\begin{proof}[Proof of Proposition~\ref{prop:implgeneric}.]
First consider the case where 
 $X $ is a \generic{} hypersurface
of degree $d$ in $\PP^{n-1}$.
We compute in $H^*(\PP^{n-1})=\Z[h]/\langle h^{n} \rangle$. The line bundle $\cR_X$ is the
pull-back of $\cR_{\PP^{n-1}}$, whose total Chern class is $1-h$. Since ${\rm codim}(X)=1$,
the vector bundle $N_X$ is a line bundle.
By \cite[Example II.8.20.3]{Har}, we have  $N_X=(\cR_X^*)^{\otimes d}$, so that
$N^*_X\otimes \cR_X^*=(\cR_X)^{\otimes (d-1)}$.
In $H^*(\PP^{n-1})$ we~have
\[ \frac{1}{c(\cR_{\PP^{n-1}}) \cdot c(\cR_{\PP^{n-1}}^{\otimes d-1})}
\,\, = \,\, \frac{1}{(1-h)(1-(d-1)h)}. \]
The coefficient of $h^{n-2}$ in this expression equals $\sum_{i=0}^{n-2}
(d-1)^i$, and since the image of $h^{n-2}$ in $H^*(X)$ under pull-back
equals $d = {\rm degree}(X)$ times the class of a point, we find
\[ \ED(X) \,\,= \,\, \int \frac{1}{c(\cR_X) \cdot c(\cN^*_X)}
\,\,= \,\,\, d \cdot \sum_{i=0}^{n-2} (d-1)^i. \]
A similar reasoning applies  when $X$ is a general  complete
intersection of $c$ hypersurfaces of degrees $d_1,\ldots,d_c$. Again,
by working in  $H^*(\PP^{n-1}) = \Z[h]/\langle h^{n} \rangle$, we evaluate
$$
{\rm EDdegree}(X) \,\, \,= \,\,
\int\frac{1}{(1-h)
\prod_{i=1}^c (1-(d_i-1)h) },
$$
where $\int$ refers to the coefficient of the point class in the
pull-back to $X$. To compute this, we expand the
integrand as a series in $h$.
The coefficient
of $h^{n-c-1}$ in that series, multiplied by
${\rm degree}(X) = d_1 \cdots d_c$, is the formula in  (\ref{eq:implgeneric2}).
 Proposition  \ref{prop:implgeneric} then follows from
Theorem~\ref{thm:affineproj}.
Here is the argument.
After a transformation (if necessary) of the given equations  $f_1,\ldots,f_s$,
 the variety $X'$ cut out by
the first $c$ of them is a complete intersection.
Then $X$ is an irreducible component of
$X'$. This implies ${\rm EDdegree}(X)\le{\rm EDdegree}(X')$.
Now, by semicontinuity,
${\rm EDdegree}(X')$ is at most the value
for a \generic{} complete intersection.
%\hfill $\diamondsuit $
\end{proof}

If $X$ is a low-dimensional
variety then Theorem \ref{th:chernpower2} may be more useful, especially if $X$
is a variety whose cohomology ring  we  understand well. We illustrate
this scenario with a computation that generalizes
 Example \ref{ex:rationalnormalcurve}
 from $X \simeq \PP^1$ to higher dimensions.

\begin{proposition} \label{prop:Veronese}
After a  change of  coordinates that creates a transverse intersection
with the isotropic quadric $Q$ in $\PP^{{{m+d}\choose d}-1}$,
the $d$-th Veronese embedding of $\,\PP^{m}$  has ED degree
\begin{equation}
\label{eq:genericveronese}
 \frac{(2d-1)^{m+1} - (d-1)^{m+1}}{d} .
 \end{equation}
\end{proposition}

\begin{proof}
We write $i_d\colon\PP^{m-1} \to X$ for the $d$th-Veronese embedding in
question.  So, $X$ denotes the image of $\PP^{m-1}$ in $ \PP^{{{m+d-1}\choose
d}-1}$ under the map given by a sufficiently general basis for the space of
homogeneous polynomials of degree $d$ in $m$ variables.
We have $c_i(X)={m + 1\choose i}h^i$, so that
$\deg c_i(X)\,=\, \int  (dh)^{m-i} c_i(X) \,=\,
{m\choose i}d^{m-i}$.
From Theorem \ref{th:chernpower2} we now get
\[ \ED(X)\,\,=\,\,\sum_{i=0}^{m}(-1)^i(2^{m+1-i}-1){m+1\choose i}d^{m-i}.
\]
Using the Binomial Theorem, we see that this alternating sum is equal to
(\ref{eq:genericveronese}).
\end{proof}

Theorem~\ref{edformula} requires $X$ to be smooth. Varieties
with favorable desingularizations are also amenable to Chern class
computations, but the computations become more technical.

\begin{example}
Let $X_r$ denote the  variety of $s\times t$ matrices of rank
$\le r$, in \generic{} coordinates so that $X_r$ intersects  $Q$ transversally.
Its ED degree can be computed by the desingularization in \cite[Proposition 6.1.1.a]{We}.
 The Chern class formula amounts to  a nontrivial computation in the ring of
symmetric functions. We implemented this in {\tt Macaulay2} as follows:

\begin{verbatim}
loadPackage "Schubert2"
ED=(s,t,r)->
(G = flagBundle({r,s-r}); (S,Q) = G.Bundles;
X=projectiveBundle (S^t); (sx,qx)=X.Bundles;
d=dim X; T=tangentBundle X;
sum(d+1,i->(-1)^i*(2^(d+1-i)-1)*integral(chern(i,T)*(chern(1,dual(sx)))^(d-i))))
\end{verbatim}

The first values of $\ED(X_r)$  are summarized in the following table
$$\begin{array}{r|rrrrrrrrrrrr}
&(s,t)&=&(2,2)&(2,3)&(2,4)&(2,5)&(3,3)&(3,4)&(3,5)&(4,4)&(4,5)&(5,5)\\
\hline
r=1&&&6&10&14&18&39&83&143&284&676&2205\\
r=2&&&&&&&39&83&143&1350&4806&55010\\
r=3&&&&&&&&&&284&676&55010\\
r=4&&&&&&&&&&&&2205\\
\end{array}$$
The $r=1$ row can also be computed
with  $P$ a product of two simplices in
Corollary \ref{cor:toricnice}.
\hfill $\diamondsuit$
\end{example}

Using the formalism of Chern classes,
Catanese and Trifogli \cite[page 6030]{CT} derive a general
formula for the degree of the ED discriminant $\Sigma_X$.
Their formula is a complicated expression
in terms of the Chow ring of the ED correspondence $\mathcal{P} \mathcal{E}_X$.
Here are two easier special cases.

\begin{example}
If $X$ is a general smooth curve in $\PP^n$
 of degree $d$ and genus $g$ then
$${\rm degree}(\Sigma_X) \,\, = \,\, 6 (d+g-1) . $$
For instance, the rational normal curve $X$ in
general coordinates in $\PP^n$,
as discussed in Example \ref{ex:rationalnormalcurve}, has
$\,{\rm degree}(X) = n ,\, \ED(X)  = 3n-2$, and $\,
{\rm degree}(\Sigma_X) = 6n-6 $.

If $X$ is a general smooth surface in $\PP^n$
of degree $d$, with Chern classes $c_1(X), c_2(X)$, then
$${\rm degree}(\Sigma_X) \,\, = \,\,
2 \cdot \bigl(\,15 \cdot d + c_1(X)^2 + c_2(X) - 9 \cdot {\rm deg} \,c_1(X) \,\bigr). $$
The formulas in  Example \ref{ex:trifogliEX} can be derived from
these expressions, as in \cite[page 6034]{CT}.
\hfill $\diamondsuit$
\end{example}

\section{Tensors of Rank One} \label{sec:Tensors}

%{In Example~\ref{ex:eckartyoung} we saw that the ED degree of the variety
%of rank one $s \times t$ {\em matrices} with $s \leq t$ equals $s$, and
%that the singular points of the squared distance function $d_u$, where
%$u$ is a real $s \times t$-matrix, are obtained from the singular value
%decomposition of $u$ by setting all but one of the singular values equal
%to zero. In particular, all critical points are real, and the average
%ED degree also equals $s$.}

In this section, we present a brief account of recent work on
multidimensional {\em tensors} of rank one \cite{Friedland}. For these,
the ED degree is computed in \cite{FO}, and the average ED degree is
computed in \cite{DH}. Our discussion includes partially symmetric tensors,
and it represents a step towards extending the Eckart-Young theorem from
matrices to tensors.

We consider real tensors $x = (x_{i_1 i_2 \cdots i_p})$
of format $m_1 \times m_2 \times \cdots \times m_p$.
The space of such tensors is the tensor product
$\R^{m_1} \otimes \R^{m_2} \otimes \cdots \otimes \R^{m_p}$,
which we  identify with $\R^{m_1 m_2 \cdots m_p}$.
The corresponding projective space
$\PP(\R^{m_1} \otimes \cdots \otimes\R^{m_p})$
is likewise identified with $ \PP^{m_1m_2 \cdots m_p-1}$.

A tensor $x$ has {\em rank one} if $x = t_1 \otimes t_2 \otimes \cdots  \otimes t_p$
for some vectors  $t_i \in \R^{m_i}$. In coordinates,
\begin{equation}
\label{eq:segrepara}
 \qquad x_{i_1 i_2 \cdots i_p} \,\, = \,\,
t_{1 i_1} t_{2 i_2} \cdots t_{p i_p}
\qquad \hbox{for} \quad 1 \leq i_1 \leq m_1,
 \ldots, 1 \leq i_p \leq m_p.
 \end{equation}
 The set $X$ of all tensors of rank one is an algebraic
 variety in $\R^{m_1 m_2 \cdots m_p}$. It is the cone
 over the {\em Segre variety}
$ \PP(\R^{m_1})\times\cdots \times\PP(\R^{m_p})
= \PP^{m_1-1} \times \cdots \times \PP^{m_p-1}$ in its natural embedding
  in $\PP^{m_1m_2 \cdots m_p-1}$.
  By slight abuse of notation, we use the symbol $X$ also for that
  Segre variety.

\begin{theorem} {\rm (\cite[Theorem 4]{FO})}.
\label{thm:tensor1}
The ED degree of the Segre variety $X $
of rank $1$ tensors of format
 $m_1 {\times} \cdots {\times} m_p$ equals the
coefficient of  the monomial $z_1^{m_1-1}\cdots z_p^{m_p-1}$
in the polynomial
$$ \qquad \quad \prod_{i=1}^p \frac{(\widehat z_i)^{m_i}-z_i^{m_i}}{\, \widehat z_i\,-\,z_i}
\qquad \hbox{ where $\,\,\,\widehat z_i \, = \, z_1 {+} \cdots {+} z_{i-1} + z_{i+1} {+} \cdots {+} z_p$.}
$$
\end{theorem}

The embedding (\ref{eq:segrepara}) of the Segre variety $X$ into
$\PP^{m_1 m_2 \cdots m_p-1}$ is not transversal to the isotropic quadric
$Q$, so our earlier formulas do not apply.  However, it is natural in
the following sense. The Euclidean distance on each factor
$\R^{m_i}$ is preserved under the action by the rotation group $SO(m_i)$.
The product group  $SO(m_1)\times \cdots\times SO(m_p)$ embeds in
the group $SO(m_1\cdots m_p)$, which acts by rotations  on the tensor
space $\R^{m_1 m_2 \cdots m_p}$.  The Segre map (\ref{eq:segrepara})
from $\R^{m_1} \times \cdots \times \R^{m_p}$ to $\R^{m_1 m_2 \cdots m_p}$
 is $SO(m_1)\times \cdots\times SO(m_p)$-equivariant.
This group invariance becomes crucial when, in a short while, we pass
to partially symmetric tensors.

For $p=2$, when the given tensor $u$ is a
matrix, Theorem \ref{thm:tensor1}  gives the Eckart-Young formula
${\rm EDdegree}(X) = {\rm min}(m_1,m_2)$.
The fact that singular vectors are the eigenvectors of $u^T u$
or $u u^T$, can be interpreted as a characterization of the ED
correspondence $\mathcal{E}_X$. The following generalization to
arbitrary tensors, due to Lim \cite{Lim}, is the key
ingredient used in \cite{FO}.
Suppose that $u = (u_{i_1 i_2 \cdots i_p})$ is a given tensor,
and we seek to find its best rank one approximation
$x^* = (x^*_{i_1 i_2 \cdots i_p}) = (t^*_{1 i_1} t^*_{2 i_2} \cdots t^*_{p i_p})$.
Then we have the {\em singular vector equations}
\begin{equation}
\label{eq:eigentuple}
 u \cdot  (t^*_1 \otimes \cdots \otimes t^*_{i-1} \otimes t^*_{i+1} \otimes \cdots \otimes t^*_p)
\, =\, \lambda t^*_i,
\end{equation}
 where the scalars $\lambda$'s are the {\em singular values} of the tensor $u$.
 The dot in (\ref{eq:eigentuple}) denotes tensor contraction.
In the special case $p=2$, these are
the equations, familiar from linear algebra, that characterize
the singular vector pairs of a rectangular matrix \cite[(1.1)]{FO}.
Theorem \ref{thm:tensor1} is proved in \cite{FO} by counting the number of solutions
to (\ref{eq:eigentuple}). The arguments used are based on Chern class techniques as
described in Section $6$.

Consider the ED correspondence $ \sP\mathcal{E}_X $,
introduced before Theorem \ref{thm:EDproj}, but now regarded as a
subvariety of $ \PP^{m_1 \cdots m_p-1} \times  \PP^{m_1 \cdots m_p-1} $.
Its equations  can be derived as follows.
The proportionality conditions of
(\ref{eq:eigentuple}) are expressed as quadratic equations given by $2\times 2$ minors.
This leads to a system of bilinear equations in $(x,u)$. These
equations, together with the
quadratic binomials in $x$ for the Segre variety $X$,
 define the ED correspondence $\sP\mathcal{E}_X$.

\begin{example}
\label{ex:222tensor}
 Let $p = 3$, $m_1 = m_2 = m_3 = 2$,
and abbreviate $a=t^*_1,b=t^*_2,c= t^*_3$, for the Segre embedding of
$X = \PP^1 \times \PP^1 \times \PP^1 $ into $\PP^7$. This toric
threefold is defined by the ideal
\begin{equation}
\label{eq:P1P1P1} \begin{matrix}
\langle \,x_{101} x_{110} - x_{100} x_{111}\, , \,\,\, x_{011} x_{110} - x_{010} x_{111}\, , \,\,\,   x_{011} x_{101} - x_{001} x_{111} \\ \,\,\,
   x_{010} x_{100} - x_{000} x_{110}\, , \,\,\, x_{001} x_{100} - x_{000} x_{101}\, , \,\,\, x_{001} x_{010} - x_{000} x_{011} \\
\,\,\,\,\,\,\, x_{010} x_{101} - x_{000} x_{111}\, , \,\,\, x_{011} x_{100} - x_{000} x_{111}\, , \,\,\,
 x_{001} x_{110} - x_{000} x_{111} \,\rangle.
 \end{matrix}
\end{equation}
The six singular vector equations
(\ref{eq:eigentuple}) for the $2 {\times} 2 {\times} 2$-tensor $x$ reduces to the proportionality between the columns of the following
three matrices
$$ \begin{pmatrix}
u_{000} b_0 c_0 + u_{001} b_0 c_1 + u_{010} b_1 c_0+ u_{011} b_1 c_1 &
 a_0 \\
u_{100} b_0 c_0 + u_{101} b_0 c_1 + u_{110} b_1 c_0 +u_{111} b_1 c_1 &
 a_1 \end{pmatrix}$$
$$\begin{pmatrix}u_{000} a_0 c_0 + u_{001} a_0 c_1 + u_{100} a_1 c_0 + u_{101} a_1 c_1 &
b_0 \\
u_{010} a_0 c_0 + u_{011} a_0 c_1 + u_{110} a_1 c_0 + u_{111} a_1 c_1 &
b_1  \end{pmatrix}$$
$$\begin{pmatrix}u_{000} a_0 b_0 + u_{010} a_0 b_1  + u_{100} a_1 b_0 + u_{110} a_1 b_1 &
c_0 \\
u_{001} a_0 b_0 + u_{011} a_0 b_1 + u_{101} a_1 b_0 + u_{111} a_1 b_1 &
c_1
\end{pmatrix}
$$
We now take the three determinants, by
using $a_i b_j c_k = x_{ijk}$, this gives the bilinear equations
\begin{equation}
\label{eq:threebilinear}
\!\! \begin{matrix}
u_{000} x_{100} + u_{001} x_{101} + u_{010} x_{110} + u_{011} x_{111} \! & \!
= \!&\! u_{100} x_{000} + u_{101} x_{001} + u_{110} x_{010} + u_{111} x_{011}, \\
u_{000} x_{010} + u_{001} x_{011}+u_{100} x_{110} + u_{101} x_{111} \! & \!
= \!&\! u_{010} x_{000}+u_{011} x_{001}+u_{110} x_{100}+u_{111} x_{101}, \\
u_{000} x_{001} + u_{010} x_{011} + u_{100} x_{101} + u_{110} x_{111} \! &  \!
= \!&\! u_{001} x_{000} + u_{011} x_{010} + u_{101} x_{100}+u_{111} x_{110}. \\
\end{matrix}
\end{equation}
The ED correspondence $\,\sP \mathcal{E}_X \subset \PP^7 \times \PP^7\,$
of $\,X = \PP^1 {\times} \PP^1 {\times} \PP^1\,$ is defined by
(\ref{eq:P1P1P1}) and (\ref{eq:threebilinear}).

By plugging the binomials (\ref{eq:P1P1P1})
 into (\ref{eq:critideal2}), we verify
${\rm EDdegree}(X) = 6$, the number from Theorem \ref{thm:tensor1}.
By contrast, if we scale the $x_{ijk}$
so that $X$ meets the isotropic
quadric $Q$ transversally, then
${\rm EDdegree}(X) =
15 \cdot 6 - 7 \cdot 12 + 3 \cdot 12 - 1 \cdot 8 = 34 $,
by Corollary \ref{cor:toricnice}.
\hfill $\diamondsuit $
\end{example}

Our duality results in Section \ref{sec:Duality} have nice
consequences for rank one tensor approximation.
It is known \cite[Chapter XIV]{GKZ} that the dual variety
$Y = X^*$ is a hypersurface if and only if
\begin{equation}
\label{eq:boundaryformat}
 2 \cdot \max(m_1,m_2,\ldots,m_p)
\,\,\leq \,\, m_1 + m_2 + \cdots + m_p - p + 2 .
\end{equation}
In that case, the polynomial defining $Y$ is
the {\em hyperdeterminant} of format $m_1 \times m_2 \times \cdots \times m_p$.
For instance, in Example \ref{ex:222tensor},
where $P$ is the $3$-cube, we get the
{\em $2 \times 2 \times 2$-hyperdeterminant}
$$ \begin{matrix}
Y &=& V \bigl(\,
x_{000}^2 x_{111}^2-2 x_{000} x_{001} x_{110} x_{111}-2x_{000} x_{010} x_{101} x_{111}
  -2 x_{000} x_{011} x_{100} x_{111} \\ & & \qquad + 4 x_{000} x_{011} x_{101} x_{110}
  + x_{001}^2 x_{110}^2+4 x_{001} x_{010} x_{100} x_{111} -
  2 x_{001} x_{010} x_{101} x_{110} \\ &  &
  - 2 x_{001}  x_{011} x_{100} x_{110}
  + x_{010}^2 x_{101}^2- 2 x_{010} x_{011} x_{100} x_{101} + x_{011}^2 x_{100}^2 \,  \bigr).
  \end{matrix}
$$
The following result was proved for $2 \times 2 \times 2$-tensors
by Stegeman and Comon \cite{SC}. However, it holds
for arbitrary $m_1,\ldots,m_p$. The proof
  is an immediate consequence
of Theorem \ref{thm:dualED}.

\begin{corollary}
Let $u$ be a tensor and $u^*$ its best rank one approximation.
Then $u-u^*$ is in the dual variety $Y$.
 In particular, if  (\ref{eq:boundaryformat}) holds then
the hyperdeterminant of $u-u^*$ is zero.
\end{corollary}

This result explains the fact, well known in the
numerical multilinear algebra community,
that tensor decomposition and best rank one approximation
are unrelated for $p\ge 3$.
The same argument gives the following generalization
to arbitrary toric varieties $X_A$.
Following \cite{GKZ}, here $A$ is a point configuration,
whose convex hull is the polytope $P$ in
Corollary \ref{cor:toricnice}.
Fix a projective toric variety $X_A \subset \PP^n$
whose dual variety $(X_A)^*$ is a hypersurface.
The defining polynomial of that hypersurface
is the {\em A-discriminant} $\Delta_A$.
See \cite{GKZ} for details.

\begin{corollary}\label{prop:discrimvanish}
Given a general point $u \in \R^{n+1}$,
let $x$ be a point in the cone over $X_A$
which is critical for the squared distance function $d_u$.
The $A$-discriminant
$\Delta_A$ vanishes at $u-x$.
\end{corollary}

 The construction of singular vectors and the ED degree formula in Theorem \ref{thm:tensor1}
 generalizes to partially symmetric tensors.
Corollary \ref{prop:discrimvanish} continues to apply in this setting.
We denote by $S^a\R^m$ the $a$-th symmetric power of $\R^m$.
Fix positive integers $\omega_1, \ldots, \omega_p$.
We consider the embedding of the Segre variety $X=\PP(\R^{m_1})\times\cdots \times\PP(\R^{m_p})$
 into the space of tensors $\PP(S^{\omega_1}\R^{m_1}\otimes\cdots \otimes S^{\omega_p}\R^{m_p})$,
sending $(v_1,\ldots, v_p)$ to $v_1^{\omega_1}\otimes\cdots\otimes v_p^{\omega_p}$.
The image is called a {\em Segre-Veronese variety}.
When $p=1$ we get the classical
{\em Veronese variety} whose points are symmetric decomposable tensors in
$\PP(S^{\omega_1}\R^{m_1})$.
A symmetric tensor $x \in S^{\omega_1}\R^{m_1}$ corresponds to a homogeneous polynomial of degree $\omega_1$
in $m_1$ indeterminates. Such a polynomial  sits in the Veronese variety
$X$ if it can be expressed as the power of a linear form.

At this point, it is extremely important to note the correct choice of coordinates
on the space $S^{\omega_1}\R^{m_1}\otimes\cdots \otimes S^{\omega_p}\R^{m_p}$.
We want the group  $SO(m_1) \times \cdots \times SO(m_p)$ to act by
rotations on that space, and our Euclidean distance
must be compatible with that action.
In order for this to happen, we must include
square roots of appropriate multinomial coefficients in
the parametrization of the Segre-Veronese variety.
We saw this Example \ref{ex:twistedcubic} for the twisted cubic curve
($p=1, m_1 = 2, \omega_1 = 3$) and in Example \ref{ex:closesym} for  symmetric
matrices ($p=1, \omega_2=3$). In both examples, the
Euclidean distances comes from the ambient
space of all tensors.

\begin{example}
Let $p = 2, m_1 = 2, m_2 = 3, \omega_1 = 3, \omega_2 = 2$.
The corresponding space
$S^3 \R^2 \otimes S^2 \R^3$ of partially symmetric tensors
has dimension $24$. We regard this as a subspace
in the $72$-dimensional space of
$2 {\times} 2 {\times} 2 {\times} 3 {\times} 3$-tensors.
With this, the coordinates on $S^3 \R^2 \otimes S^2 \R^3$ are
$x_{ijklm}$ where $1 \leq i \leq j \leq k \leq 2$
and $1 \leq l \leq m \leq 3$, and the squared distance function is
$$ \begin{matrix} d_u(x) & = &
 \,\quad (u_{11111} - x_{11111})^2 +
2 (u_{11112} - x_{11112})^2
+ \cdots +    (u_{11133} - x_{11133})^2 \\
& &
+ \, 3  (u_{12111} - x_{12111})^2
+ 6  (u_{12112} - x_{12112})^2
+ \cdots +
(u_{22233} - x_{22233})^2.
\end{matrix}
$$
In the corresponding projective space $\PP^{23} = \PP(S^3 \R^2 \otimes S^2 \R^3)$,
the threefold $X = \PP^1 \times \PP^2$ is
embedded by the line bundle $\mathcal{O}(3,2)$.
It is cut out by {\em scaled} binomial equations such as
$3 x_{11111} x_{22111} - x_{12111} x_{12111}$.
The ED degree of this
Segre-Veronese variety $X$ equals $27$.
\hfill $\diamondsuit$
\end{example}

\begin{theorem} {\rm (\cite[Theorem 5]{FO})}.
\label{thm:tensor2}
Let $X \subset \PP(S^{\omega_1}\C^{m_1}\otimes \cdots \otimes S^{\omega_p}\C^{m_p})$ be the Segre-Veronese variety of  partially symmetric tensors of rank one. In
the invariant coordinates described above, the ED degree of $X$ is the
coefficient of  the monomial $z_1^{m_1-1}\cdots z_p^{m_p-1}$
in the polynomial
$$ \qquad \quad \prod_{i=1}^p \frac{(\widehat z_i)^{m_i}-z_i^{m_i}}{\, \widehat z_i\,-\,z_i}
\qquad \hbox{ where $\,\,\,\widehat z_i \, = \,
(\sum_{j=1}^p \omega_j z_j) - z_i$.}
$$
\end{theorem}

The critical points of $d_u$ on $X$ are characterized by the
singular vector equations (\ref{eq:eigentuple}),
obtained by restricting from
ordinary tensors to partially symmetric tensors.
Of special interest is the case $p = 1$,  with
$m_1 = m $ and $\omega_1 =\omega$. Here $X$ is the
Veronese variety of  symmetric $m \times m \times \cdots \times m$
 tensors with $\omega$ factors that have rank one.

 \begin{corollary} \label{cor:veronese}
 The Veronese variety $X \subset \PP(S^\omega \C^m)$,
 with $SO(m)$ invariant coordinates,~has
  $$  {\rm EDdegree}(X) \,\, = \,\,  \frac{(\omega-1)^m - 1}{\omega-2}.  $$
 \end{corollary}

This is the formula in \cite{CS} for the number of eigenvalues of a tensor.
Indeed, for symmetric tensors,  the eigenvector equations of \cite{CS}
translate into (\ref{eq:eigentuple}).
This is well-known in the
matrix case $(\omega=2)$: computing eigenvalues and computing singular
values is essentially equivalent.
At present, we do not know how
to extend our results to tensors of rank $r \geq 2$.

\smallskip

We now shift gears and examine the {\em average} ED degrees of rank one tensors.
As above, we write $X$ for the cone over the Segre variety, given by its
distinguished embedding (\ref{eq:segrepara}) into $\R^{m_1 m_2 \cdots
m_p}$.  We fix the standard Gaussian distribution $\omega$ centered
at the origin in $\R^{m_1 m_2 \cdots m_p}$.

In \cite{DH} the average ED degree of  $X$ is
expressed in terms of the average absolute value of the determinant on
a Gaussian-type matrix ensemble constructed as follows.  Set $m:=\sum_i
(m_i-1)$ and let $A=(a_{k\ell})$ be the symmetric $m \times m$-matrix with
$p \times p$-block division into blocks of sizes $m_1-1,\ldots,m_p-1$
whose upper triangular entries $a_{k\ell},\,1 \leq k \leq \ell \leq m$,
are
\[ a_{k\ell}=\begin{cases}
U_{k\ell} & \text{if $k,\ell$ are from distinct blocks,}\\
U_0 & \text{if $k=\ell$, and}\\
0 & \text{otherwise}.
\end{cases}
\]
Here $U_0$ and the $U_{k\ell}$ with $k<\ell$ in distinct blocks
are independent  normally distributed scalar random variables.
 For instance, if $p=3$ and $(m_1,m_2,m_3)=(2,2,3)$, then
\[ A=\begin{bmatrix}
U_0 & U_{12} & U_{13} & U_{14} \\
U_{12} & U_0 & U_{23} & U_{24} \\
U_{13} & U_{23} & U_0 & 0 \\
U_{14} & U_{24} & 0 & U_0
\end{bmatrix}
\]
with $U_0, U_{12}, U_{13}, U_{14}, U_{23},U_{24} \sim N(0,1)$ independent.

\begin{theorem}[\cite{DH}] \label{thm:aedtensor}
The average ED degree of the Segre variety $X$ relative to the standard Gaussian
distribution on $\R^{m_1 m_2 \cdots m_p}$ equals
\[ \aED(X) \,\, = \,\, \frac{\pi^{p/2}}{2^{m/2} \cdot \prod_{i=1}^p
\Gamma\left(\frac{m_i}{2}\right)} \cdot \mathbb{E}(|\det(A)|), \]
where $\mathbb{E}(|\det(A)|)$ is the expected absolute
determinant of the random matrix $A$.
\end{theorem}

The proof of this theorem, which can be seen as a first step in
{\em random tensor theory}, is a computation similar to that in
Example~\ref{ex:rationalnormalaed}, though technically more difficult.
Note the dramatic decrease in dimension: instead of sampling tensors
$u$ from an $m_1 \cdots m_p$-dimensional space and computing the critical
points of $d_u$, the theorem allows us to compute the average ED degree by
sampling $m \times m$-matrices and computing their determinants. Unlike in
Example~\ref{ex:rationalnormalaed}, we do not expect that there exists
a closed form expression for $\mathbb{E}(|\det(A)|)$, but existing
asymptotic results on the expected absolute determinant, e.g. from
\cite{TV}, should still help in comparing $\aED(X)$ with $\ED(X)$ for
large $p$.  The following table from \cite{DH} gives some values
for the average ED degree  and compares them with
Theorem \ref{thm:tensor1}:
\[
\centering
    \begin{tabular}{l|l|l}
      Tensor format & average ED degree & ED degree \\ \hline
    $n\times m$ & $\min(n,m)$  & $\min(n,m)$ \\
    $2^3=2\times2\times2$         & 4.287         & 6
\\
    $2^4$         & 11.06       & 24        \\
    $2^5$         & 31.56       & 120        \\
    $2^6$         & 98.82       & 720        \\
    $2^7$         & 333.9       & 5040        \\
    $2^8$         & $1.206\cdot 10^3$         & 40320
\\
    $2^9$         & $4.611\cdot 10^3$         & 362880
\\
    $2^{10}$      & $1.843\cdot 10^4$         & 3628800 \\
    $2\times 2 \times 3$         & 5.604         & 8
\\
    $2\times 2 \times 4$         & 5.556         & 8
\\
    $2\times 2 \times 5$         & 5.536         & 8
\\
    $2\times 3 \times 3$         & 8.817         & 15
\\
    $2\times 3 \times 4$         & 10.39         & 18
\\
    $2\times 3 \times 5$         & 10.28         & 18
\\
    $3\times 3\times 3$          & 16.03         & 37
\\
    $3\times 3\times 4$          & 21.28         & 55 \\
    $3\times 3\times 5$         & 23.13         & 61
\\
    \end{tabular}
\]
It is known from \cite{FO} that $\ED(X)$ stabilizes outside the range
(\ref{eq:boundaryformat}), i.e., if the $m_i$ are ordered increasingly,
for $m_p-1 \geq \sum_{i=1}^{p-1} (m_i-1)$. This can be derived
from Theorem~\ref{thm:tensor1}. For $\aED(X)$ we observe
a similar behavior experimentally, except that the average seems to
slightly {\em decrease} with $m_p-1$ beyond this bound.
At present we have neither a geometric explanation for this
phenomenon nor a proof using the formula in Theorem~\ref{thm:aedtensor}.

\section*{Epilogue}
We conclude our investigation of the Euclidean distance degree by loosely
paraphrasing Hilbert and Cohn-Vossen in their famous book {\em Anschauliche
Geometrie} \cite[Chapter I, \S 1]{HCV}:

\begin{quote}
{\em  The simplest curves are the planar curves. Among them, the simplest one
is the line} (ED degree 1). {\em The next simplest curve is the circle} (ED
degree 2). {\em After that come the parabola} (ED degree 3),
{\em and, finally, general conics} (ED degree 4).
\end{quote}

\bigskip

\noindent
{\bf Acknowledgements.}\\
Jan Draisma was supported by a Vidi grant from the
Netherlands Organisation for Scientific Research (NWO), and
Emil Horobe\c{t} by  the NWO Free Competition grant
{\em Tensors of bounded rank}. Giorgio Ottaviani is member of
GNSAGA-INDAM. Bernd Sturmfels was supported
by  the NSF (DMS-0968882), DARPA (HR0011-12-1-0011),
and the Max-Planck Institute f\"ur Mathematik in Bonn, Germany.
Rekha Thomas was supported by the NSF (DMS-1115293).

\bigskip

\noindent
\footnotesize {\bf Authors' addresses:}

\noindent
Jan Draisma, TU Eindhoven, P.O. Box 513, 5600 MB Eindhoven,
The Netherlands, {\tt j.draisma@tue.nl};\\ and Centrum Wiskunde \& Informatica,
Science Park 123, 1098 XG Amsterdam, The Netherlands

\noindent Emil Horobe\c{t}, TU Eindhoven, P.O. Box 513, 5600 MB Eindhoven, The Netherlands,
{\tt e.horobet@tue.nl}

\noindent Giorgio Ottaviani, Universit\`a di Firenze, viale Morgagni 67A, 50134 Firenze,
Italy, {\tt ottavian@math.unifi.it}

\noindent Bernd Sturmfels,  University of California, Berkeley, CA 94720-3840, USA,
{\tt bernd@math.berkeley.edu}

\noindent Rekha Thomas, University of Washington, Box 354350, Seattle, WA 98195-4350, USA, {\tt rrthomas@uw.edu}


\begin{thebibliography}{10}

\setlength{\itemsep}{-1mm}

\bibitem{AST} C.~Aholt, B.~Sturmfels and
R.~Thomas: {\em A Hilbert scheme in computer vision},
 Canadian J.~Mathematics {\bf 65} (2013), no. 5, 961--988.

\bibitem{AH} B.~Anderson and U.~Helmke:
{\em Counting critical formations on a line}, SIAM J. Control Optim. {\bf 52}  (2014) 219--242.

\bibitem{Bertini}
D.~Bates, J.~Hauenstein, A.~Sommese, and C.~Wampler:
{\em Numerically Solving Polynomial Systems with Bertini},
SIAM, 2013.

\bibitem{BR} H.-C.G.~von Bothmer and K.~ Ranestad: {\em A general formula for the algebraic degree in semidefinite programming},
Bull. London Math. Soc. {\bf 41} (2009) 193--197.

\bibitem{CS} D.~Cartwright and
 B.~Sturmfels: {\em The number of eigenvectors of a tensor},  Linear Algebra and its Applications {\bf 438} (2013) 942--952.

\bibitem{Cat} F.~Catanese: {\em Caustics of plane curves, their birationality and matrix projections}, 
in Algebraic and Complex Geometry (eds.~A.~Fr\"uhbis-Kr\"uger et al),
Springer Proceedings in Mathematics and Statistics {\bf 71} (2014) 109--121.

\bibitem{CT}
F.~Catanese and C.~Trifogli: {\em Focal loci of algebraic varieties I},
Commun. Algebra {\bf 28} (2000) 6017--6057.

\bibitem{CFP}
M.~Chu, R.~Funderlic, and R.~Plemmons:
{\em  Structured low rank approximation},
  Linear Algebra Appl. {\bf 366} (2003) 157--172.

\bibitem{CLO}
D.~Cox, J.~Little, and D.~O'Shea:
{\em Ideals, Varieties, and Algorithms. An Introduction to Computational
Algebraic Geometry and Commutative Algebra},
 Undergraduate Texts in Mathematics. Springer-Verlag, New York, 1992.

\bibitem{DH}
J.~Draisma and E.~Horobe\c{t}:
{\em The average number of critical rank-one approximations to a tensor},
{\tt  arxiv:1408.3507}.

\bibitem{DR}
J.~Draisma and J.~Rodriguez:
{\em Maximum likelihood duality for determinantal varieties},
Int.~Math.~Res.~Not.~{\bf 20} (2014) 5648--5666.

 \bibitem{DSS}
M.~Drton, B.~Sturmfels and S.~Sullivant: {\em Lectures on Algebraic Statistics},
Oberwolfach Seminars, Vol 39, Birkh\"auser, Basel, 2009.

\bibitem{Ein} L.~Ein: {\em Varieties with small dual varieties, I},
Invent. Math. {\bf 86} (1) (1986) 63--74.

\bibitem{Friedland} S.~Friedland: {\em
Best rank one approximation of real symmetric tensors can be chosen
symmetric}, Front. Math. China {\bf 8}(1) (2013) 19--40.

\bibitem{FO} S.~Friedland and G.~Ottaviani: {\em The number of singular vector tuples and uniqueness of best rank one approximation of tensors},
Found. Comput. Math., DOI 10.1007/s10208-014-9194-z .

\bibitem{FSS} J.-C.~Faug\`ere, M.~Safey El Din and
P.-J.~Spaenlehauer: {\em Gr\"obner bases of bihomogeneous ideals generated by polynomials of bidegree (1,1): algorithms and complexity},
 J. Symbolic Comput. {\bf 46} (2011) 406--437.

\bibitem{FultonToric} W.~Fulton: {\em Introduction to Toric Varieties},
Princeton University Press, 1993.

\bibitem{Fulton} W.~Fulton: {\em Intersection Theory}, Springer,
Berlin, 1998.

\bibitem{GKZ} I.M.~Gelfand, M.M.~Kapranov, and A.V.~Zelevinsky:
  {\em Discriminants, Resultants and Multidimensional
  Determinants}, Birkh\"auser, Boston, 1994.

\bibitem{M2} D.~Grayson and M.~Stillman:
{\em Macaulay2, a software system for research in algebraic geometry}, available at
{\tt www.math.uiuc.edu/Macaulay2/}.

\bibitem{Schubert2} D.~Grayson, M.~Stillman, S.~Str{\o}mme,
D.~Eisenbud, and C.~Crissman: {\em Schubert2, computations
of characteristic classes for varieties without equations},
available at {\tt www.math.uiuc.edu/Macaulay2/}.

\bibitem{HartleySturm} R.~Hartley and P.~Sturm: {\em Triangulation}, Computer Vision and Image Understanding: CIUV (1997) 68(2): 146--157.

\bibitem{Har}
R.~Hartshorne: {\em Algebraic Geometry},
Graduate Texts in Mathematics {\bf 52}, Springer-Verlag,
New York, 1977.

\bibitem{HCV}
D.~Hilbert and S.~Cohn-Vossen:
{\em Anschauliche Geometrie},
Springer-Verlag, Berlin, 1932.

\bibitem{Holme} A.~Holme: {\em The geometric and numerical properties of duality in projective algebraic geometry}, Manuscripta Math. {\bf 61}  (1988) 145--162.

\bibitem{HKS}  S.~Ho\c sten, A.~Khetan, and B.~Sturmfels:
{\em Solving the likelihood equations}, Foundations of Computational Mathematics
{\bf 5} (2005) 389--407.

\bibitem{HS} J.~Huh and B.~Sturmfels: {\em Likelihood geometry},
  in Combinatorial Algebraic Geometry (eds. Aldo Conca et al.), Lecture Notes in Mathematics {\bf 2108}, Springer, (2014) 
  63--117.
 
\bibitem{Ily} N.V.~Ilyushechkin: {\em
The discriminant of the characteristic polynomial of a normal matrix},
 Mat. Zametki {\bf 51} (1992) 16--23; translation in
Math. Notes {\bf 51}(3-4) (1992) 230--235. 


\bibitem{JP} A.~Josse and F.~P\`ene: {\em On the degree of caustics
by reflection}, Commun. Algebra {\bf 42} (2014), 2442--2475.

\bibitem{JPb} A.~Josse and F.~P\`ene: {\em On the normal class of
curves and surfaces}, {\tt arXiv:1402.7266}.

\bibitem{Lau} M.~Laurent: {\em Cuts, matrix completions and graph rigidity},
Mathematical Programming {\bf 79} (1997) 255--283.

\bibitem{Lim} L.-H. Lim: {\em Singular values and eigenvalues of tensors: a variational approach,}
Proc. IEEE International Workshop on Computational Advances
in Multi-Sensor Adaptive Processing (CAMSAP ’05), 1 (2005), 129-132.

\bibitem{MS}
E.~Miller and B.~Sturmfels: {\em Combinatorial Commutative Algebra},
Graduate Texts in Mathematics {\bf 227}, Springer, New York, 2004.

\bibitem{OEIS} The Online Encyclopedia of Integer Sequences,
{\tt http://oeis.org/}.

\bibitem{OSS}
G.~Ottaviani, P.J.~Spaenlehauer, B.~Sturmfels:
{\em  Exact solutions in structured low-rank approximation}, to appear in SIAM Journal on Matrix Analysis and Applications, {\tt arXiv:1311.2376}. 

\bibitem{PabloThesis}
P.A.~Parrilo:  {\em Structured Semidefinite Programs and Semialgebraic
Geometry Methods in Robustness and Optimization},
PhD Thesis, Caltech,
Pasadena, CA, May 2000.

\bibitem{Pie} R.~Piene: {\em Polar classes of singular varieties},
Ann. Sci. \'Ecole Norm. Sup. (4) {\bf 11} (1978), no. 2, 247--276.

\bibitem{RS} P.~Rostalski and B.~Sturmfels: {\em Dualities},
Chapter 5 in G.~Blekherman, P.~Parrilo and R.~Thomas:
{\em Semidefinite Optimization and Convex Algebraic Geometry},
pp.~203--250,
MPS-SIAM Series on Optimization, SIAM, Philadelphia, 2013.

\bibitem{Sal}  G.~Salmon: {\em A Treatise on the Higher Plane Curves}, Dublin, 1879,
available on the web at {\tt  http://archive.org/details/117724690}.

\bibitem{SC} A.~Stegeman and P.~Comon:
{\em Subtracting a best rank-{$1$} approximation does not necessarily decrease tensor rank}, Linear Algebra Appl.~{\bf 433} (2010) 1276--1300.

\bibitem{SSN}
H. Stew{\'e}nius, F. Schaffalitzky, and D. Nist{\'e}r: {\em How hard is 3-view triangulation really?}, Proc.~{\em International Conference on Computer Vision}, Beijing, China (2005) 686--693.

\bibitem{Stu} B.~Sturmfels: {\em Solving systems of polynomial equations}, CBMS Regional Conference Series in Mathematics {\bf 97}, Amer. Math. Soc., Providence, 2002.

\bibitem{TV} T.~Tao and V.~Vu:
{\em A central limit theorem for the determinant of a Wigner matrix},
Adv. Math.~{\bf 231} (2012) 74--101.

\bibitem{TJD} J.~Thomassen, P.~Johansen, and T.~Dokken:
{\em Closest points, moving surfaces, and algebraic geometry},
 Mathematical methods for curves and surfaces: Troms\o, 2004, 351--362, Mod. Methods Math., Nashboro Press, Brentwood, TN, 2005

\bibitem{Tri} C.~Trifogli: {\em Focal loci of algebraic
hypersurfaces: a general theory},
Geometriae Dedicata {\bf 70} (1998) 1--26.

\bibitem{We} J.~Weyman: {\em Cohomology of Vector Bundles and Syzygies},
Cambridge Tracts in Mathematics, {\bf 14}, Cambridge University Press, Cambridge, 2003.

\end{thebibliography}
\end{document}